\documentclass[11pt,twoside]{amsart}
\usepackage{amssymb,amsmath}
\usepackage{amscd}
\usepackage{amsthm}
\usepackage{latexsym}
\usepackage[noadjust]{cite}
\usepackage{mathrsfs}
\usepackage{indentfirst}
\usepackage{graphicx}
\usepackage{subfigure}
\usepackage{color}
\usepackage{enumerate}
\usepackage{tikz}
\usepackage{pgfplots}
\numberwithin{equation}{section}
\pgfplotsset{compat=1.18}
\pgfplotsset{compat=1.18}
\usepgfplotslibrary{fillbetween}

\newtheorem{theorem}{Theorem}[section]
\newtheorem{lemma}[theorem]{Lemma}
\newtheorem{proposition}[theorem]{Proposition}

\theoremstyle{definition}
\newtheorem{definition}[theorem]{Definition}
\newtheorem{remark}[theorem]{Remark}

\usepackage[colorlinks,
linkcolor=blue,       %%修改此处为你想要的颜色
anchorcolor=red,  %%修改此处为你想要的颜色
citecolor=blue,        %%修改此处为你想要的颜色，例如修改blue为red
]{hyperref}
\makeatletter
\newcommand{\subsectionnotoc}{%
  \@startsection{subsection}{\@M}%
  \z@{.5\linespacing\@plus.7\linespacing}{-.5em}%
  {\normalfont\bfseries}}
\makeatother

\title[Pseudo-Relativistic Schr\"odinger Equations with Logarithm-Law]
{Ground states for pseudo-relativistic Schr\"odinger equations with logarithm-law: qualitative properties and nonrelativistic limit}

\author[P. Guan]{Pingkang Guan}
\address{School of Mathematics, Jilin University, Changchun 130012, PR CHINA}
\email{pingkangguan24@gmail.com}

\author[Y. Wei]{Yuanhong Wei}
\address{School of Mathematics \& International Institute for Mathematical Sciences, Jilin University, Changchun 130012, China}
\email{weiyuanhong@jlu.edu.cn}

\begin{document}

	\begin{abstract}
In this paper, we study the pseudo-relativistic  Schr\"odinger equation with logarithmic nonlinearity
\begin{equation*}
    \left(\sqrt{-c^{2}\Delta+m^{2}c^{4}}-mc^{2}\right)u+\mu u
=u\log|u|
\quad\text{in }\mathbb{R}^{N}.
\end{equation*}
The existence and qualitative properties of action ground states are established, by means of a power-law approximation. We then prove that these states converge to the Gausson in $H^{1}(\mathbb{R}^{N})$ as $c\to\infty$. For sufficiently large $c$, we further obtain uniqueness of action ground states up to translation and sign change, as well as their nondegeneracy. Finally, we prove an exact correspondence between action and energy ground states, thereby extending the preceding results to energy ground states at arbitrary prescribed mass.
	\end{abstract}
    \maketitle
    \\

  \tableofcontents
	
\section{Introduction}
In this paper, we consider the pseudo-relativistic Schr\"odinger equation with logarithm-law
\begin{equation}\label{problem}\tag{$P$}
    \left(\sqrt{-c^2\Delta+m^2c^4}-mc^2\right)u+\mu u=u\log |u|\quad \text{in}~\mathbb{R}^N,
\end{equation}
where $N\geq2$, $\mu>0$, $c>0$ is the speed of light, and $m>0$ represents the particle mass. The pseudo-relativistic operator
\begin{equation*}
    \sqrt{-c^2\Delta+m^2c^4}
\end{equation*}
is defined as the Fourier multiplier with symbol
\begin{equation*}
   \sqrt{c^2|\xi|^2+m^2c^4}.
\end{equation*}
Problem \eqref{problem} arises naturally from the study of standing waves for the time-dependent pseudo-relativistic
logarithmic Schr\"odinger equation
\begin{equation}\label{timedependent}
    i\partial_t\psi=\left(\sqrt{-c^2\Delta+m^2c^4}-mc^2\right)\psi-\psi\log|\psi|\quad\text{in }\mathbb{R}\times\mathbb{R}^N.
\end{equation}
From a physical perspective, the evolution equation \eqref{timedependent} describes relativistic spin-0 bosons with a logarithmic self-interaction.
Substituting the standing-wave ansatz
\begin{equation*}
    \psi(t,x)=e^{i\mu t}u(x)
\end{equation*}
into \eqref{timedependent} gives \eqref{problem}, where $\mu$ is the prescribed temporal frequency of the standing wave. 

The dynamical evolution conserves both mass and Hamiltonian energy, defined as
\[
    \mathcal{M}(u)=|u|_2^2
    \]
and   \[
  \mathcal{H}_c(u)=\frac{1}{2}\int_{\mathbb{R}^N} \left(
        \sqrt{c^2|\xi|^2+m^2c^4}-mc^2
    \right)|\widehat u(\xi)|^2\,d\xi
    +\frac{1}{4}\int_{\mathbb{R}^N}u^2\,dx-
    \frac{1}{2}\int_{\mathbb R^N}u^2\log |u|\,dx,
\]
respectively. 
These two conserved quantities give rise to two distinct variational definitions of ground states.  If the frequency $\lambda\in\mathbb{R}$ is prescribed, one considers the action
\begin{equation}\label{definition-Scrho}
    \mathcal{S}_{c,\lambda}(u)=\mathcal{H}_c(u)+\frac{\lambda}{2}\mathcal{M}(u).
\end{equation}
An {\sl action ground state} at frequency $\lambda$ is a nontrivial solution that minimizes this action among all nontrivial solutions. Alternatively, if the mass $\rho>0$ is prescribed, one considers
 \begin{equation}\label{energygroundstatelevel}
      e_c(\rho)
      =
      \inf
      \left\{
       \mathcal{H}_c(u):
       u\in\mathcal{X},\ 
       \mathcal{M}(u)=\rho
      \right\},
  \end{equation}
where 
\begin{equation*}
	\mathcal{X}=\left\{u\in H^{1/2}(\mathbb{R}^N):\int_{\mathbb{R}^N}u^2|\log |u||\,dx<\infty\right\}.
\end{equation*}
A function $u\in\mathcal{X}$ with $\mathcal{M}(u)=\rho$ is called an {\sl energy ground state} of mass $\rho$ if it satisfies $\mathcal{H}_c(u)=e_c(\rho)$.  In this constrained problem the frequency is not fixed beforehand: it appears as the Lagrange multiplier associated with the mass constraint.

The pseudo-relativistic Schr\"odinger equation with logarithm-law   exhibits a number of distinctions relative to the classical Schr\"odinger equation. One notable feature is that the logarithmic term brings scaling property and opens up new perspectives for investigation. For general nonlinear Schrödinger equations without a logarithmic term, action ground states need not coincide with energy ground states. A systematic  treatment of this distinction was provided by
Dovetta, Serra and Tilli \cite{DovettaSerraTilli2023}.
The logarithmic nonlinearity, however,
possesses an extra scaling property 
 \begin{equation*}
      \log|\alpha u|
      =
      \log\alpha+\log|u|,
     \quad \alpha>0, 
  \end{equation*}
which changes the situation. Multiplying a solution by a positive constant therefore only  changes its frequency. This amplitude-frequency relation yields an exact correspondence between action ground states and energy ground states for every prescribed mass. Its precise variational formulation was made explicit by d'Avenia, Montefusco and Squassina \cite{dAvenia2014}, who established a correspondence between critical levels on the Nehari manifold and those on prescribed $L^2$-spheres.

It is of particular interest to explore the relationship between the pseudo‑relativistic Schr\"odinger equation and its classical counterpart by characterizing the limiting behaviour as the speed of light $c$ becomes sufficiently large. 
The operator
\begin{equation*}
    P_c(D)=\sqrt{-c^2\Delta+m^2c^4}-mc^2
\end{equation*}
describes the relativistic kinetic energy after subtracting the rest energy $mc^2$. Its Fourier symbol has the formal expansion
\begin{equation}\label{Taylorexpansion}
   P_c(\xi)=\sqrt{c^2|\xi|^2+m^2c^4}-mc^2=\frac{|\xi|^2}{2m}-\frac{|\xi|^4}{8m^3c^2}+O\left(\frac{|\xi|^6}{c^4}\right)\quad\text{as }c\to\infty.
\end{equation}
 Thus the relativistic kinetic operator formally converges to its classical Schr\"odinger counterpart as $c$ tends to infinity.
This expansion suggests both the limiting stationary equation and the natural order of the approximation.

A substantial literature treats existence, qualitative properties, and nonrelativistic limits for pseudo-relativistic equations. Two principal model nonlinearities are the power-law equation
\begin{equation}\label{eq:CS-power0}
     \left(\sqrt{-c^2\Delta+m^2c^4}-mc^2\right)u+\mu u=|u|^{p-2}u\quad\text{in }\mathbb{R}^N,
\end{equation}
and the Hartree equation
\begin{equation*}
     \left(\sqrt{-c^2\Delta+m^2c^4}-mc^2\right)u+\mu u=(|x|^{-1}*|u|^2)u\quad\text{in }\mathbb{R}^3.
\end{equation*}
The extension method of Coti Zelati and Nolasco \cite{Zelati2011}, which realizes the square-root operator as a Dirichlet-to-Neumann map for a local equation in the upper half-space, provides an effective variational framework for these nonlocal problems. 
For power nonlinearities, Choi and Seok \cite{ChoiSeok2016} proved existence, sign and symmetry properties, and convergence of positive radial ground states to their nonrelativistic counterparts.
For the pseudo-relativistic Hartree equation, Lieb and Yau \cite{LiebYau1987} established the existence of positive radial ground states, while Lenzmann \cite{Lenzmann2009} proved uniqueness and nondegeneracy results for ground states in the regime of small mass by using the nonrelativistic limit.
 For power nonlinearities with integer exponents and Hartree nonlinearities, Choi, Hong and Seok \cite{ChoiHongSeok2018} subsequently obtained higher-order Sobolev regularity and the sharp convergence estimate
\begin{equation*}
    \frac{A}{c^2}\leq\Vert u_c-u_\infty\Vert_{H^s(\mathbb{R}^N)}\leq\frac{B}{c^2},\quad s\geq1,
\end{equation*}
confirming the order $c^{-2}$ predicted by the expansion \eqref{Taylorexpansion}.
Higher-order asymptotic expansions for pseudo-relativistic Hartree ground states were recently obtained in \cite{ChenCotiWei2026}.
Further results on pseudo-relativistic Schr\"odinger equations and related models can be found in \cite{ChenDingGuoWang2024, ChenDingGuo2026,
ChenWang2025, CotiZelatiNolasco2013, CotiZelatiNolasco2019, GuoZeng2017, GuoZeng2020} and the references therein.

In this paper, the logarithmic nonlinearity introduces a different variational difficulty. In the classical setting, a typical stationary equation takes the form
\begin{equation*}
    -\Delta u+V(x)u=u\log u^2\quad\text{in }\mathbb{R}^N.
\end{equation*}
Unlike smooth power and Hartree nonlinearities, the logarithmic term is singular at the origin. As a result, the functional
\begin{align*}
    \mathcal{I}(u)=\int_{\mathbb{R}^N}u^2\log u^2\,dx
\end{align*}
is not necessarily finite on the Sobolev space $H^s(\mathbb{R}^N)$.
The fractional logarithmic Sobolev inequality \cite{Cotsiolis2005} states that, for every $u\in H^s(\mathbb{R}^N)$ and $\alpha>0$,
\begin{equation}\label{fls}
    \int_{\mathbb{R}^N}u^2\log\left(\frac{u^2}{|u|_2^2}\right)\,dx+\left(N+\frac{N}{s}\log\alpha+\log\frac{s\Gamma\left(\frac{N}{2}\right)}{\Gamma\left(\frac{N}{2s}\right)}\right)|u|_2^2\leq\frac{\alpha^2}{\pi^s}|(-\Delta)^{\frac{s}{2}}u|_2^2.
\end{equation}
 This inequality rules out the possibility that $\mathcal{I}(u)=+\infty$, but it does not prevent its negative part from diverging. \footnote{ For example,
 one may choose a smooth function satisfying
\[
u(x)=
\begin{cases}
	(|x|^\frac{N}{2}\log |x|)^{-1},\quad &|x|\geq3,\\
	0,\quad &|x|\leq2.
\end{cases}
\]
Then $u\in H^s(\mathbb{R}^N)$ for every $s\geq0$, whereas $\mathcal{I}(u)=-\infty$.} Thus, the logarithmic functional is neither finite nor continuously differentiable on the Sobolev space.

Several approaches have been developed to overcome this difficulty. Cazenave \cite{Cazenave1983} introduced an Orlicz space adapted to the logarithmic term and recovered differentiability of the associated functional.  Tanaka and Zhang \cite{TanakaZhang2017} transferred the problem to a standard variational one, by using the penalization argument.
 Squassina and Szulkin \cite{SquassinaSzulkin2015} developed a nonsmooth critical point theory by decomposing the functional into the sum of a $C^1$-functional and a convex lower semicontinuous functional. These approaches have proved effective in studying existence, multiplicity, concentration phenomena, etc., see \cite{AlvesJi2020, Ardila2017, ArdilaCelySquassina2020, GalloMoscaoniSquassina2026,  JiSzulkin2016, LiuWeiZou2026, ZhangZhang2020}.

An alternative approach to the logarithmic nonlinearity is based on the limiting relation
\begin{equation}\label{powertologarithm}
    \frac{|s|^{p-2}-1}{p-2}s\to s\log|s|\quad\text{as }p\to2^+.
\end{equation}
This relation motivates the treatment of logarithm-law equations as limits of appropriately normalized families of power-law equations. This approximation was proposed formally by Troy \cite{Troy2016} for time-dependent equations, and by Wang and Zhang \cite{WangZhang2019} for nonlinear scalar field equations. Within this framework, the logarithmic problem is treated through smooth power-law approximations, thereby avoiding the direct treatment of the nonsmooth logarithmic nonlinearity.

Motivated by the idea above, we study the pseudo-relativistic Schr\"odinger equation with logarithm-law by using the power-law approximation.
Our first main result is about the existence of ground states.
\begin{theorem}[\textbf{Existence of action ground state}]\label{result1}
    Fix $c\geq 1$. Then, \eqref{problem} possesses a positive action ground state.
\end{theorem}
\begin{remark}
     	For the pseudo-relativistic power-law equation \eqref{eq:CS-power0}, Choi and Seok \cite{ChoiSeok2016} established the existence of positive radial ground states. 
    We will show that the positive ground state obtained by the theorem above is actually the limit of a suitable sequence of positive ground states of \eqref{eq:CS-power0}. To the best of our knowledge, this construction for pseudo-relativistic logarithmic problems does not appear in the literature.
\end{remark}
We are also concerned with the qualitative properties of the action ground states.
\begin{theorem}[\textbf{Qualitative properties of action ground state}]\label{result3}
Up to translations and a change of sign, every action ground state of \eqref{problem} is a positive, radial and strictly decreasing classical solution. Besides, the action ground state decays exponentially at infinity whenever $c>\sqrt{\mu/2m}$.
\end{theorem}

We next turn to the nonrelativistic limit, which connects the pseudo-relativistic Schr\"odinger equation \eqref{problem} with its classical counterpart
\begin{equation}\label{limitproblem0}\tag{$P_\infty$}
    -\frac{1}{2m}\Delta u+\mu u=u\log|u|\quad\text{in }\mathbb{R}^N.
\end{equation}
The existence and uniqueness of \eqref{limitproblem0} are obtained by
 d'Avenia-Montefusco-Squassina \cite{dAvenia2014} and Troy \cite{Troy2016}.
Combining their results, we conclude that for $N\geq2$, up to translations, there exists a unique positive $C^2$ solution of \eqref{limitproblem0} which tends to zero at infinity. This solution is the so-called Gausson 
\begin{equation*}
    \mathfrak{g}(x)=e^{\mu+\frac{N}{2}}e^{-\frac{m}{2}|x|^2}.
\end{equation*}
See also Wang and Zhang \cite[Remark1.1]{WangZhang2019}.

For the fixed frequency $\mu$, let
$E_c:\mathcal{X}\to \mathbb{R}$,
\begin{equation*}
     E_c(u)
    =
    \frac{1}{2}
    \int_{\mathbb R^N}
    \left(
        \sqrt{c^2|\xi|^2+m^2c^4}-mc^2+\mu
    \right)|\widehat u(\xi)|^2\,d\xi
    +\frac{1}{4}\int_{\mathbb{R}^N}u^2\,dx-
    \frac{1}{2}\int_{\mathbb R^N}u^2\log |u|\,dx.
\end{equation*}
One can easily check that $E_c=\mathcal{S}_{c,\mu}$, where $\mathcal{S}_{c,\mu}$ is given by \eqref{definition-Scrho}, so $E_c$ is exactly the action functional of \eqref{problem}.
The action ground-state level of $E_c$ is
\begin{equation}\label{actiongroundstatelevel}
    m_c=\min\left\{E_c(u):u\in\mathcal{X}\setminus\{0\}\text{ is a solution of }\eqref{problem}\right\}.
\end{equation}

The next result identifies the Gausson as the nonrelativistic limit of all positive action ground states.
\begin{theorem}[\textbf{Nonrelativistic limit}]\label{result2}
    There exists $c_1>0$ such that every family $\{u_c\}_{c\geq c_1}$
    of positive action ground states of \eqref{problem} converges, after suitable translations and passage to a subsequence,
     to the Gausson $\mathfrak{g}$ in $H^1(\mathbb{R}^N)$.
    Moreover, the action ground-state levels satisfy
    \begin{equation*}
         m_c\to m_\infty
        =
        \frac{1}{4}|\mathfrak{g}|_2^2=\frac{1}{4}e^{2\mu+N}\left(\frac{\pi}{m}\right)^\frac{N}{2}\quad\text{as }c\to\infty.
    \end{equation*}
\end{theorem}

By virtue of the convergence and limiting behavior above, we also show the uniqueness of low-action level solutions.
\begin{theorem}[\textbf{Uniqueness of low-action level}]\label{result4}
    There exists $c_*> 0$ such that, for every $c\geq c_*$, any two nontrivial solutions $u, v\in\mathcal{X}$ of \eqref{problem} whose actions are below $2m_c$ coincide up to a translation and sign change. That is,
    there exist $\sigma\in\{-1,1\}$ and $x_0\in\mathbb{R}^N$ such that
    \begin{equation*}
        u(x)=\sigma v(x-x_0)\quad\text{for all }x\in\mathbb{R}^N.
    \end{equation*}
    In particular, for $c\geq c_*$, \eqref{problem} admits a unique positive action ground state up to translations.
\end{theorem}
\begin{remark}
    Theorem \ref{result4} also implies that 
    the action ground-state level is isolated. For $c\geq c_*$, every nontrivial solution $u\in\mathcal{X}$ satisfies either $E_c(u)=m_c$ or $E_c(u)\geq 2m_c$.
    Hence no action level in $(m_c,2m_c)$ is attained.
    Actually, one can show that if a nontrivial solution $u$ satisfies 
    \begin{equation*}
            E_c(u)\leq\theta m_\infty
           =
            \frac{\theta}{4}e^{2\mu+N}
            \left(\frac{\pi}{m}\right)^\frac{N}{2},\quad \theta\in(1,2),\quad c\text{ large enough,}
      \end{equation*}
      then it
      is an action ground state. 
    \end{remark}
We next formulate nondegeneracy.
Motivated by the weak formulation for classical logarithmic Schr\"odinger equations adopted by Ardila and Squassina \cite{ArdilaSquassina2018},
for a positive action ground state $u_c$ of \eqref{problem},
we introduce the Hilbert space
\begin{equation*}
    \Sigma_c=\left\{\varphi\in H^{1/2}(\mathbb{R}^N):\int_{\mathbb{R}^N}|\log u_c|\varphi^2\,dx<\infty\right\}
\end{equation*}
equipped with the norm
\begin{equation*}
    \Vert \varphi\Vert_{\Sigma_c}^2=\Vert \varphi\Vert_{H^{1/2}(\mathbb{R}^N)}^2+\int_{\mathbb{R}^N}|\log u_c|\varphi^2\,dx.
\end{equation*}
%Completeness follows by identifying the limits of a Cauchy sequence in $H^{1/2}(\mathbb R^N)$ and in the weighted $L^2$ space. Moreover, since $u_c$ is strictly positive and continuous, the weight $|\log u_c|$ is locally bounded, and a standard cutoff and mollification argument yields the density of $C_c^\infty(\mathbb{R}^N)$ in $\Sigma_c$.

We define the symmetric bilinear form
\begin{equation*}
    \ell_c(\varphi,\psi)=\int_{\mathbb{R}^N}P_c(\xi)\widehat{\varphi}(\xi)\overline{\widehat{\psi}(\xi)}\,d\xi+\int_{\mathbb{R}^N}(\mu-1-\log u_c)\varphi\psi\,dx,\quad \varphi,\psi\in\Sigma_c.
\end{equation*}
By the continuity of $\ell_c$, there exists a unique bounded linear operator $\mathcal{L}_c:\Sigma_c\to\Sigma_c^*$
such that
\begin{equation*}
    \left\langle \mathcal{L}_c\varphi, \psi\right\rangle_{\Sigma_c^*,\Sigma_c}=\ell_c(\varphi,\psi),\quad \varphi,\psi\in\Sigma_c.
\end{equation*}
%Its kernel is defined by
%\begin{equation*}
%    \text{Ker}(\mathcal{L}_c)
%    =
%   \left\{
%    \varphi\in\Sigma_c:
%    \ell_c(\varphi,\psi)=0
%    \text{ for every }\psi\in\Sigma_c
%    \right\}.
%\end{equation*}
We call $u_c$ nondegenerate if
\begin{equation*}
        \text{Ker}(\mathcal{L}_c)=\text{span}\{\partial_{x_1}u_c,\ldots,\partial_{x_N}u_c\}.
    \end{equation*}
\begin{theorem}[\textbf{Nondegeneracy of action ground state}]\label{result5}
    There exists $c^*>0$ such that, for every $c\geq c^*$, the positive action ground state $u_c$ of \eqref{problem} is nondegenerate.
\end{theorem}

We finally return to the energy ground states, which correspond to the attainment functions of \eqref{energygroundstatelevel}. Denote the set of energy ground states by
\begin{equation*}
    \mathcal{G}^{e}_{c,\rho}=\left\{u\in\mathcal{X}:|u|_2^2=\rho,\mathcal{H}_c(u)= e_c(\rho)\right\},
\end{equation*}
and the set of action ground states at frequency $\mu$ by
\begin{equation*}
    \mathcal{G}^{a}_{c,\mu}=\left\{u\in\mathcal{X}\setminus\{0\}: u\text{ is a solution of \eqref{problem}, }E_c(u)=m_c\right\}.
\end{equation*}
We reveal the relationship between action and energy ground states as follows.
\begin{theorem}[\textbf{Action-energy ground state correspondence}]\label{result6}
    Let $\rho>0$. For every $c\geq1$,
    $e_c(\rho)>-\infty$ is attained 
    and 
    \begin{equation*}
       \mathcal{G}^{e}_{c,\rho}=\alpha_{c,\rho} \mathcal{G}^{a}_{c,\mu},\quad \alpha_{c,\rho}
      =
      \left(\frac{\rho}{4m_c}\right)^\frac{1}{2}.
    \end{equation*}
    Moreover, every function $u\in\mathcal{G}^{e}_{c,\rho}$ is a weak solution of 
    \begin{equation}\label{equation-energy}
        (P_c(D)+\lambda_{c,\rho})u=u\log|u|\quad\text{in }\mathbb{R}^N,
    \end{equation}
    where
    \begin{equation*}
        \lambda_{c,\rho}
        =
       \mu+\log\alpha_{c,\rho}
        =
       \mu+\frac{1}{2}\log\left(\frac{\rho}{4m_c}\right).
    \end{equation*}
    Consequently, energy ground states inherit the sign, symmetry, radial monotonicity, regularity, and decay properties established for action ground states.
\end{theorem}

   \begin{theorem}\label{result7}
       Fix $\rho>0$. For each $c\geq1$, let $u_{c,\rho}$ be a positive energy ground state of \eqref{equation-energy}. Then the following assertions hold:
       \begin{enumerate}
       \item (Nonrelativistic limit) After suitable translations and passage to a subsequence,
        \begin{equation*}
            u_{c,\rho}\to\mathfrak{g}_{\rho}\quad\text{in }H^1(\mathbb{R}^N)\quad\text{as }c\to\infty,
        \end{equation*}
        where
    \begin{equation*}
        \mathfrak{g}_\rho(x)
          =
         \rho^\frac{1}{2}
          \left(\frac{m}{\pi}\right)^\frac{N}{4}
          e^{-\frac{m}{2}|x|^2},\quad |\mathfrak{g}_\rho|_2^2=\rho.
    \end{equation*}
    The corresponding Lagrange multipliers satisfy
    \begin{equation*}
         \lambda_{c,\rho}
        \to
        \lambda_{\infty,\rho}
          =
            \frac{1}{2}\log\rho
           +\frac{N}{4}\log\left(\frac{m}{\pi}\right)
         -\frac{N}{2}.
    \end{equation*}
     The energy ground-state level is
    \begin{equation*}
        e_c(\rho)
       =
      \frac{\rho}{4}
      -\frac{\rho}{2}\lambda_{c,\rho}
      =
       \frac{\rho}{4}
         -\frac{\mu\rho}{2}
        -\frac{\rho}{4}
       \log\left(\frac{\rho}{4m_c}\right),
    \end{equation*}
    and
    \begin{equation*}
        e_c(\rho)
      \to
       e_\infty(\rho)
       =
       \frac{\rho}{4}
      \left(N+1-\log\rho-\frac{N}{2}\log\left(\frac{m}{\pi}\right)\right)
       \quad\text{as }c\to\infty.
    \end{equation*}
    \item (Uniqueness)  
   For every $c\geq c_*$, the energy ground state of mass $\rho$ is unique up to translations and a change of sign. 
    \end{enumerate}
   \end{theorem}

Now we investigate the nondegeneracy of energy ground states.
For $c\geq1$ and fixed $\rho>0$, let $u_{c,\rho}$ be a positive energy ground state of mass $\rho$.
Likewise, we define the Hilbert space
   \begin{equation*}
       \Sigma_{c,\rho}=\left\{\varphi\in H^{1/2}(\mathbb{R}^N):\int_{\mathbb{R}^N}|\log u_{c,\rho}|\varphi^2\,dx<\infty\right\},
   \end{equation*}
   equipped with the norm
   \begin{equation*}
        \Vert \varphi\Vert_{\Sigma_{c,\rho}}^2=\Vert \varphi\Vert_{H^{1/2}(\mathbb{R}^N)}^2+\int_{\mathbb{R}^N}|\log u_{c,\rho}|\varphi^2\,dx.
   \end{equation*}
  Define the symmetric bilinear form
   \begin{equation*}
       \ell_{c,\rho}(\varphi,\psi)=\int_{\mathbb{R}^N}P_c(\xi)\widehat{\varphi}(\xi)\overline{\widehat{\psi}(\xi)}\,d\xi+\int_{\mathbb{R}^N}(\lambda_{c,\rho}-1-\log u_{c,\rho})\varphi\psi\,dx,\quad\varphi,\psi\in\Sigma_{c,\rho}.
   \end{equation*}
   Set the mass sphere
   \begin{equation*}
       S_\rho=\left\{u\in \Sigma_{c,\rho}:|u|_2^2=\rho\right\},
   \end{equation*}
   and its tangent space at $u_{c,\rho}$ by
   \begin{equation*}
       T_{u_{c,\rho}}S_\rho=\left\{\varphi\in\Sigma_{c,\rho}:\int_{\mathbb{R}^N}u_{c,\rho}\varphi\,dx=0\right\}.
   \end{equation*}
   Define the constrained linearized operator $ \mathcal{L}_{c,\rho}:T_{u_{c,\rho}}S_\rho\to \left(T_{u_{c,\rho}}S_\rho\right)^*$
   by
   \begin{equation*}
       \left\langle\mathcal{L}_{c,\rho}\varphi,\psi\right\rangle_{\left(T_{u_{c,\rho}}S_\rho\right)^*,T_{u_{c,\rho}}S_\rho}=\ell_{c,\rho}(\varphi,\psi),\quad\varphi,\psi\in\ T_{u_{c,\rho}}S_\rho.
   \end{equation*}
  % Its kernel is
  % \begin{equation*}
  %     \text{Ker}(\mathcal{L}_{c,\rho})=\left\{\varphi\in T_{u_{c,\rho}}S_\rho:\ell_{c,\rho}(\varphi,\psi)=0\text{ for every }\psi\in T_{u_{c,\rho}}S_\rho\right\}.
  % \end{equation*}
   We call $u_{c,\rho}$ nondegenerate if 
   \begin{equation*}
       \text{Ker}(\mathcal{L}_{c,\rho})
        =
     \text{span}
     \left\{
      \partial_{x_1}u_{c,\rho},
       \ldots,
       \partial_{x_N}u_{c,\rho}
      \right\}.
   \end{equation*}
   \begin{theorem}[\textbf{Nondegeneracy of energy ground state}]\label{result8}
   Fix $\rho>0$. For every $c\geq c^*$, the positive energy ground state $u_{c,\rho}$ is nondegenerate.
   \end{theorem}

The main results reveal a connection between the ground states of the pseudo-relativistic Schr\"odinger  and classical Schr\"odinger equations. By analyzing the nonrelativistic limit, we prove that action ground states converge to the Gausson as $c\to\infty$. This convergence provides the basis for establishing the uniqueness and nondegeneracy of action ground states for all sufficiently large $c$. Moreover, the scaling structure of the logarithmic nonlinearity yields a correspondence between action and energy ground states, allowing us to study them within a unified framework. 

The proof for existence relies on a power-law approximation, but such procedure does not work for nonrelativistic limit. In the study of nonrelativistic limit, an essential step is upgrading a uniform $H^{1/2}(\mathbb{R}^N)$ bound to a uniform $H^1(\mathbb{R}^N)$ bound. In the power-law setting, Choi and Seok \cite{ChoiSeok2016} achieved this upgrade through an $L^q$ estimate with $q>2$. For the logarithmic nonlinearity, one can not obtain this kind of estimate. We instead establish exponential decay of the ground states uniformly in $c$ to realize such upgrade.
The logarithmic term also challenges the proof of uniqueness and nondegeneracy. For the pseudo-relativistic Hartree equation, Lenzmann \cite{Lenzmann2009} obtained local uniqueness using the implicit function theorem. In our setting, the singular behavior of the logarithmic term prevents a direct application of that method. We establish uniqueness through a normalization argument that combines local estimates with the nondegeneracy of the Gausson.
 
The paper is organized as follows. 
Section 2 introduces the variational setting, the local extension problem, and the power-law approximation.
In Section 3, we construct action ground states and establish their qualitative properties.
Section 4 is devoted to the nonrelativistic limit, together with the uniqueness and nondegeneracy of positive action ground states for sufficiently large $c$.
Finally, in Section 5, we establish the exact action-energy correspondence and apply it to energy ground states of arbitrary prescribed mass.

\subsectionnotoc{Notation} Throughout this paper, we use the following notations.
\begin{itemize}
	\item  For $1\leq p\leq\infty$, $|\cdot|_p$ denotes the norm in $L^p(\mathbb{R}^N)$;
    \item  $\langle\cdot,\cdot\rangle$ denotes the real $L^2(\mathbb{R}^N)$ inner product;
    \item $2^*$ denotes the Sobolev critical exponent for $H^1(\mathbb{R}^N)$, with $2^*=\frac{2N}{N-2}$ for $N\geq3$ and $2^*=\infty$ for $N=2$;
    \item For a constant $\alpha>0$ and a set $V$ in a vector space, we set
    \begin{equation*}
        \alpha V=\left\{\alpha v:v\in V\right\};
    \end{equation*}
	\item The symbols $\to$ and $\rightharpoonup$ denote strong and weak convergence, respectively, in the relevant space;
	\item $C, C_1, C_2, \ldots$ denote positive constants whose values may vary from line to line;
	\item $B_R(y)$ denotes the open ball in $\mathbb{R}^N$ with radius $R$ centered at $y$. In particular, $B_R:=B_R(0)$;
    \item $u^+=\max\{u, 0\}$ and $u^-=\max\{-u, 0\}$ denote the positive and negative parts of a function $u$.
\end{itemize}

\section{Preliminaries}
In this section, we introduce the variational framework and the extension problem. We also collect the preliminary
results on the power-law approximation and the logarithmic nonlinearity that will be used below.
\subsection{Variational setting and extension problem}
\leavevmode\par
Fix $m, \mu>0$, and $c\geq1$. Set
\begin{equation*}
    P_c(D):=\sqrt{-c^2\Delta+m^2c^4}-mc^2
\end{equation*}
and denote its symbol by
\begin{equation*}
    P_c(\xi):=\sqrt{c^2|\xi|^2+m^2c^4}-mc^2.
\end{equation*}
In the sequel, for $s\geq0$,
the Sobolev space $H^s(\mathbb{R}^N)$ is defined by
\begin{equation}\label{Hsnorm}
    H^s(\mathbb{R}^N):=
\left\{
u\in L^2(\mathbb{R}^N):
\int_{\mathbb{R}^N}
(1+|\xi|^2)^s|\widehat{u}(\xi)|^2\,d\xi<\infty
\right\},
\end{equation}
endowed with the norm
\begin{equation*}
    \Vert u\Vert_{H^s(\mathbb{R}^N)}
:=
\left(
\int_{\mathbb{R}^N}
(1+|\xi|^2)^s|\widehat{u}(\xi)|^2\,d\xi
\right)^{1/2}.
\end{equation*}
Following \cite{ChoiSeok2016}, we endow the Sobolev space
\begin{equation*}
    H^1(\mathbb{R}^{N+1}_+):=\left\{U\in L^2(\mathbb{R}^{N+1}_+):\nabla U\in L^2(\mathbb{R}^{N+1}_+)\right\}
\end{equation*}
with the  equivalent norm
\begin{equation*}
   \Vert U\Vert_{H^1(\mathbb{R}^{N+1}_+)}^2:=\int_{\mathbb{R}^{N+1}_+}
   \left(c^2|\nabla U(x,y)|^2+m^2c^4U(x,y)^2\right)\,dx\,dy+c(-mc^2+\mu)\int_{\mathbb{R}^N}U(x,0)^2\,dx.
\end{equation*}
Denote by $H_r^{1/2}(\mathbb{R}^N)$ the subspace of radial functions in $H^{1/2}(\mathbb{R}^N)$. The following embedding theorem is well-known in the literature. 

\begin{lemma}\label{Sobolevembedding}
    For every $p\in\left[2,\frac{2N}{N-1}\right]$, the embedding $H^{1/2}(\mathbb{R}^N)\hookrightarrow L^p(\mathbb{R}^N)$ is continuous. Moreover, for every $p\in\left(2,\frac{2N}{N-1}\right)$, the embedding $H_r^{1/2}(\mathbb{R}^N)\hookrightarrow L^p(\mathbb{R}^N)$ is compact.
\end{lemma}

To state the nondegeneracy of the Gausson $\mathfrak{g}$, we introduce the Hilbert space
\begin{equation*}
    \Sigma_\infty
    :=
     \left\{
    \varphi\in H^1(\mathbb R^N):
     |x|\varphi\in L^2(\mathbb R^N)
     \right\},
\end{equation*}
equipped with the norm
 \begin{equation*}
     \Vert\varphi\Vert_{\Sigma_\infty}^2
     :=
    \Vert\varphi\Vert_{H^1(\mathbb R^N)}^2
     +
    \int_{\mathbb{R}^N}|x|^2\varphi^2\,dx.
 \end{equation*}
We define the symmetric bilinear form
\begin{equation*}
    \ell_\infty(\varphi,\psi)
:=
\frac{1}{2m}
\int_{\mathbb{R}^N}
\nabla\varphi\cdot\nabla\psi\,dx
+
\int_{\mathbb{R}^N}
\left(
\frac{m}{2}|x|^2-\frac{N}{2}-1
\right)\varphi\psi\,dx,\quad \varphi,\psi\in\Sigma_\infty.
\end{equation*}
The associated linearized operator $\mathcal{L}_\infty:
    \Sigma_\infty\to\Sigma_\infty^*$
is defined by
\begin{equation*}
    \left\langle
\mathcal{L}_\infty\varphi,\psi
\right\rangle_{\Sigma_\infty^*,\Sigma_\infty}
:=
\ell_\infty(\varphi,\psi),
\quad
\varphi,\psi\in\Sigma_\infty.
\end{equation*}
The following lemma states the nondegeneracy of the Gausson, see \cite[Theorem 1.3]{dAvenia2014} and \cite[Lemma 4.4]{ArdilaSquassina2018}.
\begin{lemma}
    The kernel of $\mathcal{L}_\infty$ is given by
\end{lemma}
\begin{equation*}
    \text{Ker}(\mathcal{L}_\infty)=\text{span}\{\partial_{x_1}\mathfrak{g},\ldots,\partial_{x_N}\mathfrak{g}\}.
\end{equation*}

For $u\in H^{1/2}(\mathbb R^N)$, let $U\in H^1(\mathbb R^{N+1}_+)$
be the unique solution of 
\begin{equation}\label{boundaryvalueproblem}
	\left\{
	\begin{aligned}
		&\left(-c^2\Delta_{x,y}+m^2c^4\right)U(x,y)=0 &&\text{in}~\mathbb{R}^{N+1}_+,\\
		&U(x,0)=u(x)&&\text{in}~\mathbb{R}^{N}.
	\end{aligned}
	\right.
\end{equation}
The corresponding Dirichlet-to-Neumann identity reads
\begin{equation*}
    -c\frac{\partial U}{\partial y}(\cdot,0)=\sqrt{-c^2\Delta+m^2c^4}u,\quad x\in\mathbb{R}^N.
\end{equation*}
This identity allows us to reformulate the nonlocal equation \eqref{problem} as the local boundary value problem
%Consequently, if $u\in H^{1/2}(\mathbb{R}^N)$ is a solution of \eqref{problem}, then its extension $U$ satisfies
\begin{equation}\label{boundaryvalueproblem1}
	\left\{
	\begin{aligned}
		&\left(-c^2\Delta_{x,y}+m^2c^4\right)U(x,y)=0 &&\text{in}~\mathbb{R}^{N+1}_+,\\
		&-c\frac{\partial U}{\partial y}(x,0)=(mc^2-\mu)u+u\log |u|&&\text{in}~\mathbb{R}^N.
	\end{aligned}
	\right.
\end{equation}
To formulate this correspondence precisely, set
\begin{equation*}
    \mathcal{X}:=\left\{u\in H^{1/2}(\mathbb{R}^N):\int_{\mathbb{R}^N}u^2|\log |u||\,dx<\infty\right\},\quad   \mathcal{D}:=\left\{U\in H^1(\mathbb{R}^{N+1}_+):U(\cdot,0)\in\mathcal{X}\right\}.
\end{equation*}

We use the following definition of the weak solutions to \eqref{problem} and \eqref{boundaryvalueproblem1} as in
\cite{WangZhang2019,TanakaZhang2017}.
\begin{definition}
    We say that $u\in\mathcal X$ is a weak solution of \eqref{problem} if
    \begin{equation*}
            \int_{\mathbb{R}^N}
    \left(\sqrt{c^2|\xi|^2+m^2c^4}-mc^2+\mu\right)
    \widehat{u}(\xi)\,\overline{\widehat\varphi(\xi)}\,d\xi=
    \int_{\mathbb{R}^N}u\varphi\log|u|\,dx
    \end{equation*}
    for every $\varphi\in C_c^\infty(\mathbb{R}^N)$.
    
    Similarly, we say that $U\in\mathcal{D}$, with trace $u=U(\cdot,0)$, is a weak solution of \eqref{boundaryvalueproblem1} if
    \begin{equation*}
            \frac{1}{c}\int_{\mathbb R^{N+1}_+}
      \left(c^2\nabla U\cdot\nabla\Phi+m^2c^4U\Phi\right)\,dx\,dy+(-mc^2+\mu)\int_{\mathbb{R}^N}u\phi\,dx=
      \int_{\mathbb R^N}
      u\phi\log |u|\,dx
    \end{equation*}
    for every $\Phi\in C_c^\infty(\overline{\mathbb R^{N+1}_+})$,
    where $\phi=\Phi(\cdot,0)$.
\end{definition}
Every weak solution $U\in\mathcal{D}$ of \eqref{boundaryvalueproblem1} satisfies the interior equation in \eqref{boundaryvalueproblem} and hence is the unique extension of its trace.
By the Dirichlet-to-Neumann identity, $u\in\mathcal {X}$ is a weak solution of \eqref{problem} if and only if its extension $U\in\mathcal{D}$ is a weak solution of \eqref{boundaryvalueproblem1}.

\begin{lemma}[\cite{ChoiSeok2016}]\label{norms}
    Let $U\in H^1(\mathbb{R}^{N+1}_+)$ and $u=U(\cdot,0)$. Then 
    \begin{equation*}
        \int_{\mathbb{R}^N}\left(c^2|\xi|^2+m^2c^4\right)^\frac{1}{2}|\widehat{u}(\xi)|^2\,d\xi\leq\frac{1}{c}\int_{\mathbb{R}^{N+1}_+}\left(c^2|\nabla U|^2+m^2c^4U^2\right)\,dx\,dy.
    \end{equation*}
    The equality holds if and only if $U$ satisfies \eqref{boundaryvalueproblem} with boundary value $u$.
\end{lemma}

 The action functional associated with the local problem \eqref{boundaryvalueproblem1} is $I_c:\mathcal{D}\to\mathbb{R}$, defined by
 \begin{equation}\label{Ic}
\begin{aligned}
    I_c(U)
    &:=
    \frac{1}{2c}
    \int_{\mathbb{R}^{N+1}_+}
    \left(
        c^2|\nabla U(x,y)|^2+m^2c^4U(x,y)^2
    \right)\,dxdy
    +\frac{(-mc^2+\mu)}{2}
    \int_{\mathbb{R}^N}U(x,0)^2\,dx\\
    &\quad+\frac{1}{4}\int_{\mathbb{R}^N}U(x,0)^2\,dx
    -
    \frac{1}{2}\int_{\mathbb{R}^N}U(x,0)^2\log |U(x,0)|\,dx.
\end{aligned}
\end{equation}
The corresponding nonlocal action functional $E_c:\mathcal{X}\to\mathbb{R}$ is given by
\begin{equation}\label{eq:Icp-nonlocal}
    E_c(u)
    :=
    \frac{1}{2}
    \int_{\mathbb R^N}
    \left(
        \sqrt{c^2|\xi|^2+m^2c^4}-mc^2+\mu
    \right)|\widehat u(\xi)|^2\,d\xi
    +\frac{1}{4}\int_{\mathbb{R}^N}u^2\,dx-
    \frac{1}{2}\int_{\mathbb R^N}u^2\log |u|\,dx.
\end{equation}
For every $u\in\mathcal{X}$, let $U$ be its unique extension determined by \eqref{boundaryvalueproblem}. The equality case of Lemma \ref{norms} gives $I_c(U)=E_c(u)$.
Thus, the extension establishes a one-to-one correspondence between weak solutions of the two problems and preserves their action levels. In what follows, we shall work mainly with the local extension problem \eqref{boundaryvalueproblem1}.

An action ground state of the local problem \eqref{boundaryvalueproblem1} is a nontrivial weak solution that minimizes $I_c$ among all nontrivial weak solutions.
\begin{lemma}\label{groundstates}
   Let $U$ be an action ground state of \eqref{boundaryvalueproblem1}. Then its trace $u:=U(\cdot,0)$ is an action ground state of \eqref{problem}. Conversely, let $u$ be an action ground state of \eqref{problem}. Then its unique extension $U$ determined by \eqref{boundaryvalueproblem} is an action ground state of
   \eqref{boundaryvalueproblem1}.
\end{lemma}
\begin{proof}
    The trace map is a bijection between the nontrivial weak solutions of the two problems. For each corresponding pair $(U,u)$, we have $I_c(U)=E_c(u)$. Hence $U$ minimizes $I_c$ among nontrivial weak solutions of \eqref{boundaryvalueproblem1} if and only if $u$ minimizes $E_c$ among nontrivial weak solutions of \eqref{problem}.
\end{proof}

\subsection{Power-law approximation}
\leavevmode\par
We begin by recalling the following existence and symmetry result for the corresponding power-law problem
from Choi and Seok \cite{ChoiSeok2016}.
\begin{lemma}
\label{thm:Choi-Seok}
Let $m,\mu>0$, $c\geq 1$, and $p\in\left(2,\frac{2N}{N-1}\right)$.
The equation
\begin{equation}\label{eq:CS-power}
   \left(P_c(D)+\mu\right) u
    =
    |u|^{p-2}u
    \quad\text{in }\mathbb{R}^N
\end{equation}
admits a positive radial action ground state
   $u\in H^{1/2}(\mathbb{R}^N)$,
which is nonincreasing with respect to $|x|$.
Moreover, every action ground state of \eqref{eq:CS-power} has one sign and is radial up to translations.
\end{lemma}
%For later use, we also recall the variational characterization of the ground
%states established in \cite{ChoiSeok2016}. The $C^1$ functional $I_e: H^1(\mathbb{R}^{N+1}_+)\to \mathbb{R}$ associated with the extension problem of \eqref{eq:CS-power} is given by
%\begin{equation}\label{Ie}
%    \begin{aligned}
%    I_e(U)
%&:=
%\frac{1}{2c}\int_{\mathbb{R}^{N+1}_+}
%\left(c^2|\nabla U(x,y)|^2+m^2c^4U(x,y)^2\right)\,dx\,dy\\
%&\quad+\frac{(-mc^2+\mu)}{2}\int_{\mathbb{R}^N}U(x,0)^2\,dx
%-\frac1p\int_{\mathbb{R}^N}|U(x,0)|^p\,dx.
%\end{aligned}
%\end{equation}
%Set
%\begin{equation*}
%    J_e(U):=I_e'(U)U
%\end{equation*}
%and define the Nehari manifold by
%\begin{equation*}
%    \mathcal N_e
%:=
%\left\{
%U\in H^1(\mathbb{R}^{N+1}_+)\setminus\{0\}:
%J_e(U)=0
%\right\}.
%\end{equation*}
%Moreover, define
%\begin{equation*}
%    M_e:=\inf_{U\in\mathcal N_e} I_e(U).
%\end{equation*}
%It was proved in \cite{ChoiSeok2016} that $M_e$ is achieved by some 
%$U\in\mathcal N_e$. In particular, $U$ is a critical
%point of $I_e$, and its trace
%$u=U(\cdot,0)$
%is a ground state solution of \eqref{eq:CS-power}.

For $p\in\left(2,\frac{2N}{N-1}\right)$, put $\mu_p=\mu+\frac{1}{p-2}$, choose a positive radial action ground state $w_p$ of \eqref{eq:CS-power} with parameter $\mu_p$, and set
\begin{equation*}
     u_p
    :=
    (p-2)^{\frac{1}{p-2}}w_p.
\end{equation*}
Then $u_p$ is a positive radial action ground state of
\begin{equation}\label{eq:power-approximation}
   \left(P_c(D)+\mu\right) u_p
    =
    \frac{|u_p|^{p-2}-1}{p-2}
    u_p
    \quad\text{in }\mathbb{R}^N.
\end{equation}
The corresponding extension functional is
\begin{equation}\label{Je}
    \begin{aligned}
     I_{c,p}(U)
&:=
\frac{1}{2c}\int_{\mathbb{R}^{N+1}_+}
\left(c^2|\nabla U(x,y)|^2+m^2c^4U(x,y)^2\right)\,dx\,dy+\frac{(-mc^2+\mu)}{2}\int_{\mathbb{R}^N}U(x,0)^2\,dx\\
&\quad+\frac{1}{2p}\int_{\mathbb{R}^N}|U(x,0)|^p\,dx
-\frac{1}{2}\int_{\mathbb{R}^N}\frac{|U(x,0)|^p-U(x,0)^2}{p-2}\,dx.
\end{aligned}
\end{equation}
For convenience, set 
\begin{equation}\label{D*H*}
    \mathcal{D}_*:=\left\{U\in\mathcal{D}:U(\cdot,0)\neq0\right\},\quad H_*:=\left\{U\in H^1(\mathbb{R}^{N+1}_+):U(\cdot,0)\neq0\right\}. 
\end{equation}
We define the Nehari sets associated with $I_c$ and $I_{c,p}$, respectively, by
\begin{equation*}
    \mathcal{N}_c:=\left\{U\in\mathcal{D}_*:J_c(U)=0\right\}\quad\text{and}\quad \mathcal{N}_{c,p}:=\left\{U\in H_*:J_{c,p}(U)=0\right\},
\end{equation*}
where 
\begin{align*}
    J_c(U):=&\frac{1}{c}\int_{\mathbb R^{N+1}_+}
\left(c^2|\nabla U|^2+m^2c^4U^2\right)\,dx\,dy+(-mc^2+\mu)\int_{\mathbb{R}^N}U(x,0)^2\,dx\\
&\quad-
\int_{\mathbb R^N}
U(x,0)^2\log |U(x,0)|\,dx,
\end{align*}
and
\begin{align*}
    J_{c,p}(U):=& \frac{1}{c}\int_{\mathbb R^{N+1}_+}
\left(c^2|\nabla U|^2+m^2c^4U^2\right)\,dx\,dy+(-mc^2+\mu)\int_{\mathbb{R}^N}U(x,0)^2\,dx\\
&\quad-\int_{\mathbb{R}^N}\frac{|U(x,0)|^p-U(x,0)^2}{p-2}\,dx.
\end{align*}

\begin{lemma}\label{notempty}
    \begin{enumerate}[\rm(i)]
        \item For every $U\in \mathcal{D}_*$, there exists a unique $t_U>0$ such that $t_UU\in\mathcal{N}_c$. Moreover, $I_c(t_UU)>I_c(tU)$ for all $t>0$ and $t\neq t_U$.
        \item For every $V\in H_*$, there exists a unique $t_V>0$ such that $t_VV\in\mathcal{N}_{c,p}$. Moreover,
        $I_{c,p}(t_VV)>I_{c,p}(tV)$ for all $t>0$ and $t\neq t_V$.
    \end{enumerate}
\end{lemma}
\begin{proof}
    We only prove (i) since the proof of (ii) is similar. Let $U\in \mathcal{D}_*$ and consider the fiber map
    \begin{align*}
        \Phi(t):=I_{c}(tU)=
    t^2\left(I_c(U)-\frac{\log t}{2}\int_{\mathbb{R}^N} U(x,0)^2\,dx\right),\quad t>0.
    \end{align*}
    Since
    \begin{equation*}
        \Phi'(t)=2t\left(I_c(U)-\frac{(1+2\log t)}{4}\int_{\mathbb{R}^N} U(x,0)^2\,dx\right)=\frac{J_c(tU)}{t},
    \end{equation*}
    we have
    \begin{equation*}
        	tU\in \mathcal{N}_c\iff \Phi'(t)=0.
    \end{equation*}
   It is clear that the equation $\Phi'(t)=0$ has a unique solution $t=t_U>0$. Moreover, $\Phi$ is increasing on $(0,t_U)$ and decreasing on $(t_U,\infty)$. Thus, $t_U$ is the unique global maximum point of $\Phi$, and hence $I_c(t_UU)>I_c(tU)$ for all $t>0$ and $t\neq t_U$. This proves (i).
\end{proof}
For $U\in\mathcal{N}_c$ and $V\in\mathcal{N}_{c,p}$, respectively,
\begin{equation*}
    I_c(U)=\frac{1}{4}\int_{\mathbb{R}^N}U(x,0)^2\,dx,\quad  I_{c,p}(V)=\frac{1}{2p}\int_{\mathbb{R}^N}|V(x,0)|^p\,dx.
\end{equation*}
Thus, we define the Nehari levels
\begin{equation*}
    m_c:=\inf_{U\in\mathcal{N}_c}I_c(U),
\quad
  m_{c,p}:=\inf_{V\in\mathcal{N}_{c,p}}I_{c,p}(V).
\end{equation*}
These levels are well-defined. The use of the same notation $m_c$
as in \eqref{actiongroundstatelevel} will be justified in the proof of Theorem \ref{result1}, where we show that the Nehari level
coincides with the action ground-state level.

The following results from \cite{WangZhang2019} will play a key role in the comparison between $m_c$ and $m_{c,p}$.

\begin{lemma}\label{estimate1}
\begin{enumerate} [\rm(i)]
    \item  For any $q>2$, there exists $C_q>0$ such that
    \begin{equation*}
        \frac{x^{p-2}-1}{p-2}\leq
         C_qx^{q-2}
    \end{equation*}
    holds for all $p\in(2,q)$ and $x\geq0$.
        \item Let $s>0$, $\delta>0$, then
        \begin{equation*}
            \frac{x^s(x^\delta-1)}{\delta}\to x^s\log x\text{ in }C_{loc}^{m,\alpha}[0,+\infty)
        \end{equation*}
        as $\delta\to0$, where $m$ is the largest integer with $m<s$, and $\alpha\in (0,s-m)$.
\end{enumerate}
\end{lemma}

Let
\begin{equation}\label{f}
    f(s):=
\begin{cases}
s\log|s|,&\quad s\neq0,\\
0,&\quad s=0,
\end{cases}
\end{equation}
and, for $(s,t)\neq(0,0)$, define
\begin{equation}\label{Df}
    \mathcal{D}_f(s,t):=
\begin{cases}
\displaystyle\frac{f(s)-f(t)}{s-t},&\quad s\neq t,\\
f'(s),&\quad s=t\neq0.
\end{cases}
\end{equation}
At the end of this section, we present some auxiliary properties of the logarithmic nonlinearity that will be useful in our analysis below.
\begin{lemma}[\cite{Shuai2019}]\label{logBrezisLieb}
      	Let $\{u_n\}\subset H^1(\mathbb{R}^N)$ be a bounded sequence such that $u_n\to u$ a.e. in $\mathbb{R}^N$, and assume that the sequence $\{u_n^2\log u_n^2\}\subset L^1(\mathbb{R}^N)$ is bounded. Then $u^2\log u^2\in L^1(\mathbb{R}^N)$, and
      	\begin{align*}
      		\lim_{n \to \infty} \int_{\mathbb{R}^N} \left[ u_n^2 \log u_n^2 - |u_n - u|^2 \log |u_n - u|^2 \right] \, dx = \int_{\mathbb{R}^N} u^2 \log u^2 \, dx.
      	\end{align*}
      \end{lemma}
      
\begin{lemma}\label{log-divided-difference}
The following properties hold.
\begin{enumerate}[\rm(i)]
\item If $0<\alpha\le s,t\leq\beta$, then
\begin{equation*}
    \mathcal{D}_f(s,t)=1+\log\theta
\end{equation*}
for some $\theta\in[\alpha,\beta]$. Moreover,
\begin{equation*}
    |\mathcal{D}_f(s,t)|\leq
\max\{|1+\log\alpha|,\ |1+\log\beta|\}.
\end{equation*}

\item Let $K\subset\mathbb{R}^N$ be compact and let
$v_c,w_c:K\to(0,\infty)$ satisfy
\begin{equation*}
    v_c(x),w_c(x)\to g(x)\quad\text{for a.e. }x\in K,
\end{equation*}
where $g>0$ on $K$. Suppose that there exist constants
$0<\alpha_K\leq\beta_K$ such that
\begin{equation*}
    \alpha_K\leq v_c(x),w_c(x)\leq\beta_K
\quad\text{for a.e. }x\in K
\end{equation*}
and all sufficiently large $c$. Then, for every $1\leq p<\infty$,
\begin{equation*}
    \mathcal{D}_f(v_c,w_c)\to f'(g)=1+\log g
\quad\text{in }L^p(K).
\end{equation*}

\item If $0<s,t\leq e^{-1}$, then $\mathcal{D}_f(s,t)\leq0$.
If $0<s\leq e^{-1}$ and
$-e^{-1}\leq t<0$, then
$\mathcal{D}_f(s,t)\leq-1$.

%\item Let $K\subset\mathbb R^N$ be compact. Suppose that
%$v,w\in L^\infty(K)$ and $w$ satisfies $\inf_K w\geq 2\alpha$
%for some $\alpha>0$. Then there exists
%a constant $C_K>0$, depending only on $\alpha$ and the
%bounds of $v,w$ on $K$, such that
%\begin{equation*}
%    |f(v)-f(w)-f'(w)(v-w)|_{L^1(K)}
%\leq C_K|v-w|_{L^2(K)}^2.
%\end{equation*}
\end{enumerate}
\end{lemma}

\begin{proof}
The first assertion follows from the mean value theorem. The
second follows from the first assertion and the dominated
convergence theorem.
For the third assertion, if $s,t>0$, the mean value theorem
gives
\begin{equation*}
    \mathcal{D}_f(s,t)=1+\log\theta\le0
\end{equation*}
for some $\theta\in(0,e^{-1}]$. If $s>0>t$, letting
$r=-t\in(0,e^{-1}]$, we have
\begin{equation*}
    \mathcal{D}_f(s,t)
 =\frac{s\log s+r\log r}{s+r}\leq -1.
\end{equation*}

%Finally, set
%\begin{equation*}
%   E:=\{x\in K:|v(x)-w(x)|\leq\alpha\},\quad F:=K\setminus E.
%\end{equation*}
%On $E$, since 
%\begin{equation*}
%    w(x)\geq2\alpha,\quad   |v(x)-w(x)|\leq\alpha,
%\end{equation*}
%we have
%\begin{equation*}
%    v(x)\geq w(x)-|v(x)-w(x)|\geq 2\alpha-\alpha=\alpha.
%\end{equation*}
%Thus, the segment joining $v(x)$ and $w(x)$ is contained in $[\alpha,\infty)$, and Taylor's formula yields
%\begin{equation*}
%    |f(v)-f(w)-f'(w)(v-w)|\leq\frac{1}{2\alpha}|v-w|^2.
%\end{equation*}
%As a result,
%\begin{equation*}
%    \int_{E}|f(v)-f(w)-f'(w)(v-w)|\,dx\leq\frac{1}{2\alpha}|v-w|_{L^2(K)}^2.
%\end{equation*}
%On $F$, the boundedness assumptions give
%\begin{equation*}
%    |f(v)-f(w)-f'(w)(v-w)|\leq C_K,
%\end{equation*}
%where $C_K>0$ depends on $\alpha$ and the bounds of $v, w$ on $K$.
%By the definition of $F$,
%\begin{equation*}
%    \alpha^2|F|\leq\int_{F}|v-w|^2\,dx,
%\end{equation*}
%and consequently,
%\begin{equation*}
%    |F|\leq\frac{1}{\alpha^2}|v-w|_{L^2(K)}^2.
%\end{equation*}
%It follows that
%\begin{equation*}
%    \int_{F} |f(v)-f(w)-f'(w)(v-w)|\,dx\leq\frac{C_K}{\alpha^2}|v-w|_{L^2(K)}^2.
%\end{equation*}
%Combining the estimates on $E$ and $F$, we conclude that
%\begin{equation*}
%    |f(v)-f(w)-f'(w)(v-w)|_{L^1(K)}
%\leq C_K|v-w|_{L^2(K)}^2.
%\end{equation*}
%This proves (iv) and completes the proof.
\end{proof}

\section{Action ground states and qualitative properties}
In this section, we first establish the existence of action ground states by passing to the limit in the power-law approximation introduced in Section 2, and then establish several qualitative properties of the resulting ground states.
\subsection{Existence of action ground states}

\begin{lemma}\label{energycomparison}
    It holds that
    \begin{equation*}
        \limsup_{p\to2^+}m_{c,p}\leq m_c.
    \end{equation*}
\end{lemma}
\begin{proof}
    Fix $U\in\mathcal{N}_c$, and denote its trace by $u=U(\cdot,0)$.
    By Lemma \ref{notempty}, for every $p>2$ sufficiently close to $2$, there exists a unique $t_p>0$ such that $t_pU\in\mathcal{N}_{c,p}$. Then it follows from $J_{c,p}(t_pU)=0$ that
    \begin{equation*}
        t_p^{p-2}=\frac{\int_{\mathbb{R}^N}u^2\,dx+(p-2)\left(\frac{1}{c}\int_{\mathbb R^{N+1}_+}
\left(c^2|\nabla U|^2+m^2c^4U^2\right)\,dx\,dy+(-mc^2+\mu)\int_{\mathbb{R}^N}u^2\,dx\right)}{\int_{\mathbb{R}^N}|u|^p\,dx}.
    \end{equation*}
    Since $U\in\mathcal{N}_c$, we have
    \begin{equation*}
        \frac{1}{c}\int_{\mathbb R^{N+1}_+}
\left(c^2|\nabla U|^2+m^2c^4U^2\right)\,dx\,dy+(-mc^2+\mu)\int_{\mathbb{R}^N}u^2\,dx=
\int_{\mathbb R^N}
u^2\log |u|\,dx.
    \end{equation*}
    As a result,
    \begin{equation}\label{tp}
        t_p^{p-2}=\frac{\int_{\mathbb{R}^N}u^2\,dx+(p-2)\int_{\mathbb{R}^N}u^2\log |u|\,dx}{\int_{\mathbb{R}^N}|u|^p\,dx}.
    \end{equation}
    Fix $2<q<\frac{2N}{N-1}$,
    and assume, without loss of generality, that
    $2<p<q<\frac{2N}{N-1}$.
    Since $U\in\mathcal{D}$, we have $  u^2|\log|u||\in L^1(\mathbb{R}^N)$,
    while the trace embedding theorem gives $u\in L^q(\mathbb{R}^N)$.
    Lemma \ref{estimate1} (i) therefore gives
   \begin{equation*}
       \left|\frac{|u|^p-u^2}{p-2}\right|\leq C\left(|u|^q+u^2|\log|u||\right)\in L^1(\mathbb{R}^N).
   \end{equation*}
   Besides, we also have
   \begin{equation*}
       \frac{|u|^p-u^2}{p-2}\to u^2\log |u|\quad\text{a.e. in}~\mathbb{R}^N,
   \end{equation*}
   as $p\to2^+$. Hence,
   by the dominated convergence theorem,
   \begin{equation*}
       \lim_{p\to2^+}\int_{\mathbb{R}^N}\frac{|u|^p-u^2}{p-2}\,dx=\int_{\mathbb{R}^N}u^2\log |u|\,dx.
   \end{equation*}
   Consequently, by \eqref{tp}, we obtain
   \begin{equation*}
       t_p^{p-2}=1+o(p-2).
   \end{equation*}
   Then
   \begin{equation*}
       (p-2)\log t_p=\log (1+o(p-2))=o(p-2).
   \end{equation*}
  Hence $t_p\to1$ as $p\to2^+$.
   Since $t_pU\in\mathcal{N}_{c,p}$ and $U\in\mathcal{N}_c$,
   \begin{align*}
       I_{c,p}(t_pU)&=I_{c,p}(t_pU)-\frac{1}{2}J_{c,p}(t_pU)\\
       &=\frac{t_p^p}{2p}\int_{\mathbb{R}^N}|u|^p\,dx\to\frac{1}{4}\int_{\mathbb{R}^N}|u|^2\,dx=I_c(U).
   \end{align*}
   Consequently,
   \begin{equation*}
       \limsup_{p\to 2^+}m_{c,p}\leq \limsup_{p\to 2^+}I_{c,p}(t_pU)=I_c(U).
   \end{equation*}
  By the arbitrariness of $U$, we conclude that 
  \begin{equation*}
      \limsup_{p\to 2^+}m_{c,p}\leq m_c.
  \end{equation*}
\end{proof}

To pass to the limit as $p\to2^+$, we next establish a uniform estimate for the approximating problem.
\begin{lemma}\label{d12}
    There exists a constant $p_0\in\left(2,\frac{2N}{N-1}\right)$ such that, for any $d_1, d_2>0$, if $U\in H^1(\mathbb{R}^{N+1}_+)$ and $p\in (2,p_0)$ satisfy
    \begin{equation*}
        I_{c,p}(U)\leq d_1,\quad \Vert I'_{c,p}(U)\Vert_{(H^1(\mathbb{R}^{N+1}_+))^*}\leq d_2,
        \end{equation*}
        then there exists a constant $C(d_1, d_2)>0$ such that
        \begin{equation*}
            \Vert U\Vert_{H^1(\mathbb{R}^{N+1}_+)}, ~\int_{\mathbb{R}^N}\left|\frac{|U(x,0)|^p-U(x,0)^2}{p-2}\right|\,dx\leq C(d_1, d_2).
        \end{equation*}
\end{lemma}
\begin{proof}
    Set $u=U(\cdot,0)$.
    It follows from the assumption that
    \begin{equation*}
        2d_1+d_2\Vert U\Vert_{H^1(\mathbb{R}^{N+1}_+)}\geq 2I_{c,p}(U)-I_{c,p}'(U)U=\frac{1}{p}\int_{\mathbb{R}^N}|u|^p\,dx.
    \end{equation*}
    Hence,
    \begin{equation*}
        \int_{\mathbb{R}^N}|u|^p\,dx\leq C_1(d_1,d_2)\left(1+\Vert U\Vert_{H^1(\mathbb{R}^{N+1}_+)}\right).
    \end{equation*}
    Let $2<p<p_0<\frac{2N}{N-1}$,
    and choose $\theta\in (0,1)$ such that 
    $\frac{1}{p_0}=\frac{1-\theta}{p}+\frac{\theta(N-1)}{2N}$.
    By Lemma \ref{norms} and Lemma \ref{estimate1} (i), there exists a constant $C_{p_0}>0$ such that
    \begin{align*}
        \int_{|u|\geq1}\frac{|u|^p-u^2}{p-2}\,dx&\leq C_{p_0}\int_{\mathbb{R}^N}|u|^{p_0}\,dx\\
        &\leq C_{p_0}\left(\int_{\mathbb{R}^N}|u|^p\,dx\right)^\frac{p_0(1-\theta)}{p}\left(\int_{\mathbb{R}^N}|u|^{\frac{2N}{N-1}}\,dx\right)^\frac{p_0\theta(N-1)}{2N}\\
        &\leq C_{p_0}\left(\int_{\mathbb{R}^N}|u|^p\,dx\right)^\frac{p_0(1-\theta)}{p}\Vert u\Vert_{H^{1/2}(\mathbb{R}^N)}^{p_0\theta}\\
        &\leq C_{p_0}C_1(d_1,d_2)\left(1+\Vert U\Vert_{H^1(\mathbb{R}^{N+1}_+)}\right)^{\frac{p_0(1-\theta)}{p}+p_0\theta}.
    \end{align*}
    Since $\frac{p_0(1-\theta)}{p}+p_0\theta\to1$ as $p_0\to2$ uniformly, we can take $p_0>2$ small enough such that
    \begin{equation*}
        \int_{|u|\geq1}\frac{|u|^p-u^2}{p-2}\,dx\leq C_2(d_1,d_2)\left(1+\Vert U\Vert_{H^1(\mathbb{R}^{N+1}_+)}\right)^\frac{3}{2}.
    \end{equation*}
    As a result,
    \begin{equation}\label{H1bounded}
        \begin{aligned}
        d_1&\geq
   \frac{1}{2c}\Vert U\Vert_{H^1(\mathbb{R}^{N+1}_+)}^2+\frac{1}{2p}\int_{\mathbb{R}^N}|u|^p\,dx-\frac{1}{2}\int_{|u|\leq1}\frac{|u|^p-u^2}{p-2}\,dx\\
   &\quad-C_2(d_1,d_2)\left(1+\Vert U\Vert_{H^1(\mathbb{R}^{N+1}_+)}\right)^\frac{3}{2},
    \end{aligned}
    \end{equation}
    which yields that
    \begin{equation*}
        \Vert U\Vert_{H^1(\mathbb{R}^{N+1}_+)}\leq C_3(d_1,d_2).
    \end{equation*}
    Therefore,
    \begin{equation*}
        \int_{|u|\geq1}\frac{|u|^p-u^2}{p-2}\,dx\leq C_4(d_1,d_2).
    \end{equation*}
    Besides, \eqref{H1bounded} also yields that
    \begin{equation*}
        \int_{|u|\leq1}\frac{u^2-|u|^p}{p-2}\,dx\leq C_5(d_1,d_2).
    \end{equation*}
    Hence we conclude that
    \begin{equation*}
        \int_{\mathbb{R}^N}\left|\frac{|u|^p-u^2}{p-2}\right|\,dx\leq C(d_1, d_2).
    \end{equation*}
    This completes the proof.
\end{proof}
In the sequel, we always take $p_0$ defined in Lemma \ref{d12} and assume that $p\in (2,p_0)$. We also need a uniform positive lower bound for the Nehari level in order to exclude vanishing in the limiting process.

\begin{lemma}\label{lowerboundedness}
 There exists $\delta_0>0$ such that
 \begin{equation*}
     m_{c,p}\geq\delta_0
 \end{equation*}
 for all $p\in (2,p_0).$
\end{lemma}
\begin{proof}
    Let $U\in H_*$ and denote its trace by $u=U(\cdot,0)$. By Lemma \ref{norms} and Lemma \ref{estimate1} (i), arguing as in the proof of Lemma \ref{d12}, we have
    \begin{align*}
        I_{c,p}(U)&=\frac{1}{2c}\Vert U\Vert_{H^1(\mathbb{R}^{N+1}_+)}^2+\frac{1}{2p}\int_{\mathbb{R}^N}|u|^p\,dx-\frac{1}{2}\int_{\mathbb{R}^N}\frac{|u|^p-u^2}{p-2}\,dx\\
        &\geq\frac{1}{2c}\Vert U\Vert_{H^1(\mathbb{R}^{N+1}_+)}^2-\frac{1}{2}\int_{|u|\geq1}\frac{|u|^p-u^2}{p-2}\,dx\\
        &\geq\frac{1}{2c}\Vert U\Vert_{H^1(\mathbb{R}^{N+1}_+)}^2-C_{p_0}\int_{\mathbb{R}^N}|u|^{p_0}\,dx\\
        &\geq\frac{1}{2c}\Vert U\Vert_{H^1(\mathbb{R}^{N+1}_+)}^2-C\Vert U\Vert_{H^1(\mathbb{R}^{N+1}_+)}^{p_0}.
    \end{align*}
    Since $p_0>2$, we deduce that there exists $r_0>0 ,\delta_0>0$ such that
    $I_{c,p}(U)\geq\delta_0$
   whenever $\Vert U\Vert_{H^1(\mathbb{R}^{N+1}_+)}= r_0$.
   Hence, for every $U\in H_*$, choosing
   \begin{equation*}
       t_0=\frac{r_0}{\Vert U\Vert_{H^1(\mathbb{R}^{N+1}_+)}},
   \end{equation*}
   we have $I_{c,p}(t_0U)\geq\delta_0$.
   By Lemma \ref{notempty} (ii),
   \begin{equation*}
       m_{c,p}=\inf_{U\in H_*}\max_{t>0}I_{c,p}(tU),
   \end{equation*}
   and therefore $m_{c,p}\geq \delta_0$
   for all $p\in(2,p_0)$.
\end{proof}

Now we are ready to prove Theorem \ref{result1}. 
\begin{proof}
    For each $p\in (2,p_0)$, let $u_p$ be the positive radially symmetric decreasing action  ground state of \eqref{eq:power-approximation}, and let $U_p$ be its unique extension. By the Nehari characterization of the power-law action ground states and the previous rescaling, their extensions satisfy
    \begin{equation*}
        I_{c,p}(U_p)=m_{c,p},\quad I'_{c,p}(U_p)=0.
    \end{equation*}
    By Lemma \ref{energycomparison}, we have that $I_{c,p}(U_p)$ is bounded. Then it follows from Lemma \ref{d12} that $\Vert U_p\Vert_{H^1(\mathbb{R}^{N+1}_+)}$ and $\int_{\mathbb{R}^N}\left|\frac{|U_p(x,0)|^p-U_p(x,0)^2}{p-2}\right|\,dx$ are bounded. Therefore, up to a subsequence, there exists $U_0\in H^1(\mathbb{R}^{N+1}_+)$ such that
    $U_p\rightharpoonup U_0$ in $H^1(\mathbb{R}^{N+1}_+)$ as $p\to2^+.$
    Denote $u_0=U_0(\cdot,0)$.
    By Lemma \ref{norms} and the trace embedding theorem, after passing to a further subsequence if necessary,
    \begin{equation*}
        U_p(\cdot,0)\to U_0(\cdot,0)\quad\text{in }L^q_{loc}(\mathbb{R}^N)
    \end{equation*}
    for every $1\leq q<\frac{2N}{N-1}$ and 
    \begin{equation}\label{UptoU0ae}
        U_p(x,0)\to U_0(x,0)\quad\text{a.e. in }\mathbb{R}^N.
    \end{equation}
    In particular, $U_0(\cdot,0)$ is radially symmetric decreasing.
    
        We first show that $U_0$ is a nontrivial solution of \eqref{boundaryvalueproblem1}.
        By Lemma \ref{estimate1} (ii) and \eqref{UptoU0ae},
        \begin{equation*}
            \left|\frac{|U_p(x,0)|^p-U_p(x,0)^2}{p-2}\right|\to U_0(x,0)^2|\log|U_0(x,0)||\quad\text{a.e. in }\mathbb{R}^N.
        \end{equation*}
        Hence, by Fatou's lemma,
        \begin{equation*}
            \int_{\mathbb{R}^N}U_0(x,0)^2|\log \left|U_0(x,0)|\right|\,dx\leq\liminf_{p\to2^+}\int_{\mathbb{R}^N}\left|\frac{|U_p(x,0)|^p-U_p(x,0)^2}{p-2}\right|\,dx<\infty.
        \end{equation*}
        Thus, $U_0\in\mathcal{D}$. Since $I_{c,p}'(U_p)=0$, for every $\Phi\in C_c^\infty(\overline{\mathbb{R}^{N+1}_+})$, we have
        \begin{equation}\label{I'cp}
             \begin{aligned}
           0&=\frac{1}{c}\int_{\mathbb R^{N+1}_+}
           \left(c^2\nabla U_p\cdot\nabla\Phi+m^2c^4U_p\Phi\right)\,dx\,dy+(-mc^2+\mu)\int_{\mathbb{R}^N}U_p(x,0)\Phi(x,0)\,dx\\
           &\quad-\int_{\mathbb{R}^N}\frac{|U_p(x,0)|^{p-2}U_p(x,0)-U_p(x,0)}{p-2}\Phi(x,0)\,dx.
        \end{aligned}
        \end{equation}
         If $U_p(x,0)\geq1$, Lemma \ref{estimate1} (i) yields that there exists $C_{p_0}>0$ such that
         \begin{equation}\label{ugeq1}
             0\leq\frac{(|U_p(x,0)|^{p-2}-1)|U_p(x,0)|}{p-2}\leq C_{p_0}|U_p(x,0)|^{p_0-1}.
         \end{equation}
         Since $U_p(\cdot,0)\to U_0(\cdot,0)$ in $L_{loc}^{p_0}(\mathbb{R}^N)$, we deduce from Vitali’s convergence theorem that
         \begin{equation}\label{DCT1}
             \begin{aligned}
                 \int_{U_p(x,0)\geq1}\frac{|U_p(x,0)|^{p-2}U_p(x,0)-U_p(x,0)}{p-2}\Phi(x,0)\,dx\\\to\int_{U_0(x,0)\geq 1}U_0(x,0)\Phi(x,0)\log |U_0(x,0)|\,dx.
             \end{aligned}
         \end{equation}
        If $U_p(x,0)\leq1$, by Lemma \ref{estimate1} (ii), there exists $C>0$ such that
        \begin{equation*}
            0\leq\frac{(1-|U_p(x,0)|^{p-2})|U_p(x,0)|}{p-2}\leq C.
        \end{equation*}
        Since $\Phi$ has a compact support, the dominated convergence theorem yields
        \begin{equation}\label{DCT}
            \begin{aligned}
                \int_{U_p(x,0)\leq1}\frac{|U_p(x,0)|^{p-2}U_p(x,0)-U_p(x,0)}{p-2}\Phi(x,0)\,dx\\\to\int_{U_0(x,0)\leq 1}U_0(x,0)\Phi(x,0)\log |U_0(x,0)|\,dx.
            \end{aligned}
        \end{equation}
        Combining \eqref{DCT1} and \eqref{DCT}, and passing to the limit in \eqref{I'cp}, we conclude that
        \begin{equation*}
             \begin{aligned}
           0&=\frac{1}{c}\int_{\mathbb R^{N+1}_+}
           \left(c^2\nabla U_0\cdot\nabla\Phi+m^2c^4U_0\Phi\right)\,dx\,dy+(-mc^2+\mu)\int_{\mathbb{R}^N}U_0(x,0)\Phi(x,0)\,dx\\
           &\quad-\int_{\mathbb{R}^N}U_0(x,0)\Phi(x,0)\log |U_0(x,0)|\,dx,
        \end{aligned}
        \end{equation*}
        which implies that $U_0$ solves \eqref{boundaryvalueproblem1}. We claim that there exists $\delta_1>0$ such that
        \begin{equation}\label{delta1}
            \inf\limits_{p\in (2,p_0)}\int_{\mathbb{R}^N}|U_p(x,0)|^{p_0}\,dx\geq\delta_1.
        \end{equation}
       Suppose by contradiction that, up to a subsequence, $U_p(\cdot,0)\to0$ in $L^{p_0}(\mathbb{R}^N)$. By Lemma \ref{estimate1} (i), reasoning as for the proof of \eqref{ugeq1}, we have
        \begin{equation*}
        \begin{aligned}
           0=I'_{c,p}(U_p)U_p&= \frac{1}{c}\Vert U_p\Vert_{H^1(\mathbb{R}^{N+1}_+)}^2-\int_{\mathbb{R}^N}\frac{|U_p(x,0)|^p-U_p(x,0)^2}{p-2}\,dx\\
           &\geq \frac{1}{c}\Vert U_p\Vert_{H^1(\mathbb{R}^{N+1}_+)}^2+\int_{U_p(x,0)\leq\frac{1}{2e}}U_p(x,0)^2\,dx-C_{p_0}|U_p(x,0)|_{p_0}^{p_0}.
        \end{aligned}
        \end{equation*}
        Here we also use the fact that, for all $p$ sufficiently close to $2$, $\frac{x^2-x^p}{p-2}\geq x^2$ for $x\in[0,\frac{1}{2e}]$.
        Thus, $\Vert U_p\Vert_{H^1(\mathbb{R}^{N+1}_+)}\to0$ as $p\to 2^+$.
        Moreover, by Lemma \ref{norms} and the trace embedding theorem, we have
        $U_p(\cdot,0)\to0$ in $L^2(\mathbb{R}^N)\cap L^{p_0}(\mathbb{R}^N)$. Since $2<p<p_0$, the interpolation gives $U_p(\cdot,0)\to0$ in $L^p(\mathbb{R}^N)$.
        Nevertheless, Lemma \ref{lowerboundedness} yields that
        \begin{equation*}
            0<\delta_0\leq m_{c,p}=I_{c,p}(U_p)=I_{c,p}(U_p)-\frac{1}{2}J_{c,p}(U_p)=\frac{1}{2p}\int_{\mathbb{R}^N}|U_p(x,0)|^p\,dx,
        \end{equation*}
        which leads to a contradiction. Thus, \eqref{delta1} holds.
       Since $U_p(\cdot,0)$ is radial, the compact embedding theorem implies that
       \begin{equation*}
           \int_{\mathbb{R}^N}|U_0(x,0)|^{p_0}\,dx=\lim_{p\to2^+}\int_{\mathbb{R}^N}|U_p(x,0)|^{p_0}\,dx\geq\delta_1>0.
       \end{equation*}
       Hence, $U_0\neq0$ and $U_0$ is a nontrivial solution of \eqref{boundaryvalueproblem1}.

       Next, we prove that $U_0$ achieves $m_c$ and $m_{c,p}\to m_c$. By Fatou's lemma and Lemma \ref{energycomparison},
       \begin{equation}\label{mcmc}
           \begin{aligned}
           m_c\leq I_c(U_0)&=\frac{1}{4}\int_{\mathbb{R}^N}U_0(x,0)^2\,dx\\
           &\leq \liminf_{p\to 2^+}\frac{1}{2p}\int_{\mathbb{R}^N}|U_p(x,0)|^p\,dx\\
           &=\liminf_{p\to 2^+}I_{c,p}(U_p)=\liminf_{p\to 2^+}m_{c,p}\\
           &\leq\limsup_{p\to 2^+}m_{c,p}\leq m_c.
       \end{aligned}
       \end{equation}
       Hence $U_0$ achieves $m_c$ and
       $m_{c,p}\to m_c$ as $p\to2^+$. Moreover, every nontrivial weak solution $V\in\mathcal{D}$ of \eqref{boundaryvalueproblem1} belongs to $\mathcal{N}_c$, and hence 
       \begin{equation*}
           I_c(V)\geq m_c=I_c(U_0).
       \end{equation*}
       Consequently, $U_0$ is a ground state of \eqref{boundaryvalueproblem1}, and
       \begin{equation*}
       \begin{aligned}
              m_c&=\min\left\{I_c(V):V\in\mathcal{D}\setminus\{0\}\text{ is a weak solution of }\eqref{boundaryvalueproblem1}\right\}\\
              &=\min\left\{E_c(u):u\in\mathcal{X}\setminus\{0\}\text{ is a weak solution of }\eqref{problem}\right\}.
       \end{aligned}
       \end{equation*}
       In particular, a nontrivial weak solution of \eqref{boundaryvalueproblem1} is an action ground state if and only if it attains the minimum of $I_c$ on $\mathcal{N}_c$.
       
      Finally, we prove that $u_p\to u_0$ in $H^{1/2}(\mathbb{R}^N)$. By Lemma \ref{norms}, it is sufficient to show that $U_p\to U_0$ in $H^1(\mathbb{R}^{N+1}_+)$. 
      Since $U_p\in\mathcal{N}_{c,p}$ and $U_0\in\mathcal{N}_c$, we have
      \begin{equation*}
          \frac{1}{c}\Vert U_p\Vert_{H^1(\mathbb{R}^{N+1}_+)}^2=\int_{\mathbb{R}^N}\frac{|U_p(x,0)|^p-U_p(x,0)^2}{p-2}\,dx
      \end{equation*}
      and
      \begin{equation*}
           \frac{1}{c}\Vert U_0\Vert_{H^1(\mathbb{R}^{N+1}_+)}^2=\int_{\mathbb{R}^N}U_0(x,0)^2\log|U_0(x,0)|\,dx.
      \end{equation*}
     Thus, we only need to prove that
      \begin{equation*}
          \int_{\mathbb{R}^N}\frac{|U_p(x,0)|^p-U_p(x,0)^2}{p-2}\,dx\to\int_{\mathbb{R}^N}U_0(x,0)^2\log |U_0(x,0)|\,dx.
      \end{equation*}
      Since $U_p(x,0)=U_p(|x|,0)>0$ is radially nonincreasing and $|U_p(\cdot,0)|_{p_0}$ is bounded, there exists $R>0$ such that
      \begin{equation}\label{updecreasing}
          \left\{x\in\mathbb{R}^N:U_p(x,0)\geq1\right\}\subset B_R(0).
      \end{equation}
      If $U_p(x,0)\geq1$,
     by \eqref{ugeq1}, we have
      \begin{equation*}
          0\leq\frac{|U_p(x,0)|^{p}-U_p(x,0)^2}{p-2}\leq C_{p_0}|U_p(x,0)|^{p_0}.
      \end{equation*}
       Since $U_p(\cdot,0)\to U_0(\cdot,0)$ in $L_{loc}^{p_0}(\mathbb{R}^N)$, using \eqref{updecreasing}, we apply Vitali's convergence theorem to get
       \begin{equation}\label{DCT2}
           \int_{U_p(x,0)\geq1}\frac{|U_p(x,0)|^p-U_p(x,0)^2}{p-2}\,dx\to\int_{U_0(x,0)\geq1
           }U_0(x,0)^2\log |U_0(x,0)|\,dx.
       \end{equation}
       If $U_p(x,0)\leq1$, by Fatou's lemma, 
       \begin{equation}\label{leq1convergence}
           -\int_{U_0(x,0)\leq1
           }U_0(x,0)^2\log |U_0(x,0)|\,dx\leq\liminf_{p\to2^+}\left(-\int_{U_p(x,0)\leq1}\frac{|U_p(x,0)|^p-U_p(x,0)^2}{p-2}\,dx\right).
       \end{equation}
      Since $U_p\in\mathcal{N}_{c,p}$ and $U_0\in\mathcal{N}_c$, we obtain
      \begin{align*}
         \int_{U_0(x,0)\geq1}&U_0(x,0)^2\log |U_0(x,0)|\,dx+\int_{U_0(x,0)\leq1}U_0(x,0)^2\log |U_0(x,0)|\,dx\\
         &=\frac{1}{c}\Vert U_0\Vert_{H^1(\mathbb{R}^{N+1}_+)}^2\leq \liminf_{p\to2^+}\frac{1}{c}\Vert U_p\Vert_{H^1(\mathbb{R}^{N+1}_+)}^2\\
         &=\liminf_{p\to2^+}\left(\int_{U_p(x,0)\geq1}\frac{|U_p(x,0)|^p-U_p(x,0)^2}{p-2}\,dx+\int_{U_p(x,0)\leq1}\frac{|U_p(x,0)|^p-U_p(x,0)^2}{p-2}\,dx\right).
      \end{align*}
      Using \eqref{DCT2}, we deduce that     \begin{equation*}
         \limsup_{p\to2^+}\left(-\int_{U_p(x,0)\leq1}\frac{|U_p(x,0)|^p-U_p(x,0)^2}{p-2}\,dx\right)\leq-\int_{U_0(x,0)\leq1}U_0(x,0)^2\log |U_0(x,0)|\,dx,
      \end{equation*}
      which, combined with \eqref{leq1convergence}, yields
      \begin{equation*}
           \lim_{p\to2^+}\int_{U_p(x,0)\leq1}\frac{|U_p(x,0)|^p-U_p(x,0)^2}{p-2}\,dx=\int_{U_0(x,0)\leq1}U_0(x,0)^2\log |U_0(x,0)|\,dx.
      \end{equation*}
      As a result,
      \begin{equation*}
          \lim_{p\to2^+}\int_{\mathbb{R}^N}\frac{|U_p(x,0)|^p-U_p(x,0)^2}{p-2}\,dx=\int_{\mathbb{R}^N}U_0(x,0)^2\log |U_0(x,0)|\,dx.
      \end{equation*}
      Therefore, we conclude that $U_p\to U_0$ in $H^1(\mathbb{R}^{N+1}_+)$,  and consequently, $u_p\to u_0$ in $H^{1/2}(\mathbb{R}^N)$. Moreover, since $u_p>0$ and 
      \begin{equation*}
          u_0(x)=\lim\limits_{p\to2^+}u_p(x),
      \end{equation*}
      we have $u_0\geq0$.
      Since $U_0\neq0$ and $I_c(U_0)<2m_c$,
      Lemma \ref{constantsign} implies that $u_0$ is strictly positive.  
      This completes the proof.
\end{proof}

\subsection{Qualitative properties}

\begin{lemma}\label{Linfty-regularity}
   If $U\in\mathcal{D}$ is a weak solution of \eqref{boundaryvalueproblem1}, then its trace $u\in L^p(\mathbb{R}^N)$ for all $p\in [2,\infty]$ and $U\in L^\infty(\mathbb{R}^{N+1}_+)$.
\end{lemma}
\begin{proof}
    Set $U^+:=\max\left\{U,0\right\}, U^-:=\max\left\{-U,0\right\}$, and denote their traces by $u^+$ and $u^-$, respectively. We only prove the result for $U^+$ since the argument for $U^-$ is similar.

    Choose $q_0\in\left(2,2+\frac{1}{N-1}\right]$ and define $g(x):=(u^+(x))^{q_0-2}\chi_{\{u^+\geq1\}}(x)$. Since $u^+\in H^{1/2}(\mathbb{R}^N)$, the Sobolev embedding theorem implies $u^+\in L^{\frac{2N}{N-1}}(\mathbb{R}^N)$. Moreover, $2N(q_0-2)\leq\frac{2N}{N-1}$.
    Therefore, on the set $\{u^+\geq 1\}$,
    \begin{equation*}
        |g|^{2N}=(u^+)^{2N(q_0-2)}\leq(u^+)^\frac{2N}{N-1},
    \end{equation*}
    and hence $g\in L^{2N}(\mathbb{R}^N)$.
    For $\gamma>0$ and $T>0$, define $U_T:=\min\left\{U^+,T\right\}$, $u_T:=U_T(\cdot,0)$.
    We take  $\Phi:=U^+U_T^{2\gamma}$ as a test function.
    A direct computation gives
 \begin{align*}
       \frac{2\gamma+1}{c}&\int_{U^+\leq T}c^2|\nabla U^+|^2U_T^{2\gamma}\,dx\,dy+\frac{1}{c}\int_{U^+>T}c^2|\nabla U^+|^2U_T^{2\gamma}\,dx\,dy\\
       &\quad+\frac{1}{c}\int_{\mathbb{R}^{N+1}_+}m^2c^4|U^+|^2U_T^{2\gamma}\,dx\,dy\\
       &=(mc^2-\mu)\int_{\mathbb{R}^N}|u^+|^2u_T^{2\gamma}\,dx+\int_{\mathbb{R}^N}|u^+|^2u_T^{2\gamma}\log u^+\,dx\\
       &\leq C\int_{\mathbb{R}^N}|u^+|^2u_T^{2\gamma}\,dx+C_{q_0}\int_{\mathbb{R}^N}g(x)|u^+|^2u_T^{2\gamma}\,dx.
    \end{align*}
Thus,
\begin{align*}
    \frac{1}{c}&\int_{\mathbb{R}^{N+1}_+}\left(c^2|\nabla(U^+U_T^\gamma)|^2+m^2c^4|U^+U_T^\gamma|^2\right)\,dx\,dy\\
    &\leq(\gamma+1)\left(C\int_{\mathbb{R}^N}|u^+|^2u_T^{2\gamma}\,dx+C_{q_0}\int_{\mathbb{R}^N}g(x)|u^+|^2u_T^{2\gamma}\,dx\right).
\end{align*}
   Since $g\in L^{2N}(\mathbb{R}^N)$, H\"older's inequality gives
   \begin{equation*}
       \int_{\mathbb{R}^N}g(x)|u^+|^2u_T^{2\gamma}\,dx\leq |g|_{2N}|u^+u_T^\gamma|_2|u^+u_T^\gamma|_{\frac{2N}{N-1}}.
   \end{equation*}
   Applying the trace embedding and Young’s inequality at this finite value of $T$, we find a constant $C_*>0$, independent of $T$ and $\gamma$, such that
   \begin{equation*}
       |u^+u_T^\gamma|^2_{\frac{2N}{N-1}}\leq C_*(\gamma+1)^2|u^+u_T^\gamma|_2^2.
   \end{equation*}
   Now assume that $u^+\in L^{2(1+\gamma)}(\mathbb{R}^N)$. Since $|u^+u_T^\gamma|_2^2\leq|u^+|_{2(1+\gamma)}^{2(1+\gamma)}$,
    passing to the limit $T\to\infty$, we obtain
    \begin{equation}\label{Moser3}
    |u^+|_{\frac{2N}{N-1}(1+\gamma)}\leq \left(C_*(\gamma+1)^2\right)^{\frac{1}{2(1+\gamma)}}|u^+|_{2(1+\gamma)}.
     \end{equation}
     Besides, Fatou’s lemma yields
     \begin{equation}\label{Moser1}
    \begin{aligned}
     \frac{1}{c}&\int_{\mathbb{R}^{N+1}_+}c^2|\nabla(|U^+|^{1+\gamma})|^2\,dx\,dy+\frac{1}{c}\int_{\mathbb{R}^{N+1}_+}m^2c^4||U^+|^{1+\gamma}|^2\,dx\,dy\\
     &\leq (\gamma+1)\left(C\int_{\mathbb{R}^N}|u^+|^{2(1+\gamma)}\,dx+C_{q_0}\int_{\mathbb{R}^N}g(x)|u^+|^{2(1+\gamma)}\,dx\right).
\end{aligned}
\end{equation}
Let $\beta=\frac{N}{N-1}$. Since
$u^+\in L^{2\beta}(\mathbb{R}^N)$, choosing $1+\gamma=\beta$ in \eqref{Moser3}, we get 
\begin{equation*}
     |u^+|_{2\beta^2}\leq \left(C_*
     \beta^2\right)^{\frac{1}{2\beta}}|u^+|_{2\beta}.
\end{equation*}
Repeating this argument, we obtain
\begin{equation*}
    |u^+|_{2\beta^{k+1}}\leq\left(\prod_{j=1}^{k}(C_*\beta^{2j})^\frac{1}{2\beta^j}\right)|u^+|_{2\beta}.
\end{equation*}
Note that
\begin{equation*}
    \log \left(\prod_{j=1}^{k}(C_*\beta^{2j})^\frac{1}{2\beta^j}\right)=\frac{1}{2}\sum_{j=1}^{k}\frac{1}{\beta^j}\log C_*+\sum_{j=1}^{k}\frac{j\log\beta}{\beta^j}.
\end{equation*}
Since $\beta>1$, we have
\begin{equation*}
    \sum_{j=1}^{\infty}\frac{1}{\beta^j}<\infty,\quad \sum_{j=1}^{\infty}\frac{j}{\beta^j}<\infty.
\end{equation*}
Hence, there exists a constant $K>0$, independent of $k$, such that
\begin{equation*}
    |u^+|_{2\beta^{k+1}}\leq K|u^+|_{2\beta}.
\end{equation*}
By interpolation, we deduce that $u^+\in L^p(\mathbb{R}^N)$ for all $p\in[2,\infty).$ Moreover,
\begin{equation*}
    |u^+|_\infty=\lim_{k\to\infty}|u^+|_{2\beta^k}<\infty.
\end{equation*}
Consequently, $u^+\in L^\infty(\mathbb{R}^N)$. The same argument gives $u^-\in L^p(\mathbb{R}^N)$ for all $p\in[2,\infty],$ and hence $u\in L^p(\mathbb{R}^N)$ for all 
$p\in[2,\infty]$.

Finally, set $M:=\max\{1,|u^+|_2,|u^+|_\infty\}$.
For every $r\geq 2$, it follows that
\begin{equation*}
   \int_{\mathbb{R}^N}|u^+|^r\,dx\leq |u^+|_\infty^{r-2}|u^+|_2^2\leq  M^r.
\end{equation*}
Since $g(x)=(u^+(x))^{q_0-2}\chi_{\{u^+\geq1\}}(x)$,
we have $|g|_\infty<\infty.$
Returning to \eqref{Moser1}, for every $\gamma>0$,
\begin{align*}
     \Vert |U^+|^{1+\gamma}\Vert_{H^1(\mathbb{R}^{N+1}_+)}^2&\leq C(\gamma+1)\left(\int_{\mathbb{R}^N}|u^+|^{2(1+\gamma)}\,dx+\int_{\mathbb{R}^N}g(x)|u^+|^{2(1+\gamma)}\,dx\right)\\
     &\leq C(\gamma+1)(1+|g|_\infty)M^{2(1+\gamma)}\\
     &\leq C(\gamma+1)M^{2(1+\gamma)}.
\end{align*}
By the Sobolev embedding theorem,
\begin{equation*}
    \Vert U^+\Vert_{\frac{2(N+1)(1+\gamma)}{N-1}}=\Vert |U^+|^{(1+\gamma)}\Vert_{\frac{2(N+1)}{N-1}}^{\frac{1}{1+\gamma}}\leq C^{\frac{1}{2(1+\gamma)}}(\gamma+1)^{\frac{1}{2(1+\gamma)}}M.
\end{equation*}
Letting $\gamma\to\infty$ gives $U^+\in L^\infty(\mathbb{R}^{N+1}_+)$. The same argument applies to $U^-$, and therefore $U\in L^\infty(\mathbb{R}^{N+1}_+)$.
This completes the proof.
\end{proof}

\begin{lemma}\label{C2-regularity}
    If $U\in\mathcal{D}\cap L^\infty(\mathbb{R}^{N+1}_+)$ is a weak solution of \eqref{boundaryvalueproblem1}, then for every $\alpha\in (0,1)$,
    \begin{equation*}
        U\in C^{1,\alpha}(\mathbb{R}^N\times[0,\infty))\cap C^2(\mathbb{R}^{N+1}_+).
    \end{equation*}
    Consequently, $U$ is a classical solution of \eqref{boundaryvalueproblem1}.
\end{lemma}
\begin{proof}
    Let $u=U(\cdot,0)$. By Lemma \ref{Linfty-regularity}, we already know that $u\in L^p(\mathbb{R}^N)$ for all $p\in [2,\infty]$. Define
    \begin{equation*}
        h(x):=(mc^2-\mu)u(x)+u(x)\log|u(x)|.
    \end{equation*}
  For each $q\in(2,\infty)$, choose  $0<\epsilon<q-2$. It holds that 
  \begin{equation*}
      2<q-\epsilon<q<q+\epsilon<\infty.
  \end{equation*}
  Hence $u\in L^{q-\epsilon}(\mathbb{R}^N)\cap L^{q+\epsilon}(\mathbb{R}^N)$. Moreover, there exists a constant $C=C(q,\epsilon)>0$ such that
    \begin{equation*}
        |s\log |s||^q\leq C\left(|s|^{q+\epsilon}+|s|^{q-\epsilon}\right),\quad s\in\mathbb{R}.
    \end{equation*}
   Therefore, $u\log |u|\in L^q(\mathbb{R}^N)$ for all $q\in (2,\infty)$. Besides, it follows from  $u\in L^\infty(\mathbb{R}^N)$ that $u\log|u|\in L^\infty(\mathbb{R}^N)$.
    Consequently, $h\in L^q(\mathbb{R}^N)$ for all $q\in (2,\infty]$.
   Arguing as in the proof of \cite[Proposition 3.9]{Zelati2011}, we obtain that, for every $\alpha\in(0,1)$, 
   \begin{equation*}
       U\in C^{0,\alpha}(\mathbb{R}^N\times [0,\infty))\cap W^{1,q}(\mathbb{R}^N\times(0,R))
   \end{equation*}
   for every $R>0$ and $q\in(2,\infty)$. In particular,
   \begin{equation*}
       u\in C^{0,\alpha}(\mathbb{R}^N).
   \end{equation*}
   Recall that
    \begin{equation*}
         f(s)=\begin{cases}
      s\log |s|,&\quad s\neq0,\\
     0,&\quad s=0.
    \end{cases}
    \end{equation*}
     By \cite{dAvenia2015}, $f\in C^{0,\eta}([a,b])$ for all $a,b\in\mathbb{R}$ with $a<b$ and any $\eta\in(0,1)$. Given $\alpha\in(0,1)$, choose $\gamma,\eta\in(0,1)$ with $\gamma\eta>\alpha$.
     Since $u\in C^{0,\gamma}(\mathbb{R}^N)\cap L^\infty(\mathbb{R}^N)$, the composition $f(u)$, and hence $h$, belongs to $C^{0,\alpha}(\mathbb{R}^N)$. 
     Applying  \cite[Proposition 3.9]{Zelati2011}, we conclude that
     $ U\in C^{1,\alpha}(\mathbb{R}^N\times [0,\infty))\cap C^2(\mathbb{R}^{N+1}_+)$ is a classical solution of \eqref{boundaryvalueproblem1}.
\end{proof}

\begin{lemma}\label{constantsign}
    Let $U\in\mathcal{D}$ be a nontrivial solution of \eqref{boundaryvalueproblem1} with $I_c(U)< 2m_c$. Then its trace $u=U(\cdot,0)$ has a constant sign.
\end{lemma}
\begin{proof}
  Suppose that $u$ changes sign, so the traces $u^\pm=U^\pm(\cdot,0)$ are nontrivial. Since $U$ solves \eqref{boundaryvalueproblem1},
  testing the equation with $U^+$ and $-U^-$ gives
   \begin{equation*}
      J_c(U^+)=J_c(U^-)=0.
   \end{equation*}
   Therefore, $U^\pm\in\mathcal{N}_c$ and $ I_c(U^\pm)\geq m_c$, and hence
   \begin{align*}
       I_c(U)=I_c(U^+)+I_c(U^-)\geq 2m_c,
   \end{align*}
   contradicting $I_c(U)<2m_c$. Thus, $u$ does not change sign. Without loss of generality, assume that $u\geq0$.  
    We next prove that $U\geq0$. Taking $U^-$ as a test function, since $u^-=0$, we obtain
   \begin{equation*}
       -\frac{1}{c}\int_{\mathbb{R}^{N+1}_+}\left(c^2|\nabla U^-|^2+m^2c^4|U^-|^2\right)\,dx\,dy=0.
   \end{equation*}
   Hence, $U^-\equiv0$, and consequently, $U\geq0$. Since $U$ solves 
   \eqref{boundaryvalueproblem1}, the strong maximum principle says that $U>0$ in $\mathbb{R}^{N+1}_+$.
  It remains to prove that $u$ is strictly positive. Arguing indirectly, we suppose that there exists $x_0\in\mathbb{R}^N$ such that $u(x_0)=0$. Then $U$ attains its global minimum at the boundary point $(x_0,0)$. The Hopf lemma implies that 
  \begin{equation*}
      \frac{\partial U}{\partial\nu}(x_0,0)<0.
  \end{equation*}
   Nevertheless, it follows from \eqref{boundaryvalueproblem1} that
   \begin{equation*}
       -\frac{\partial U}{\partial\nu}(x_0,0)=\frac{\partial U}{\partial y}(x_0,0)=\left(-mc+\frac{\mu}{c}\right)u(x_0)-\frac{1}{c}u(x_0)\log |u(x_0)|=0,
   \end{equation*}
   which leads to a contradiction. Therefore, $u$ is strictly positive.
   The case $u\leq0$ follows similarly. Hence, the trace $u$ has a constant sign.
\end{proof}

\begin{lemma}\label{expdecay}
Let $u\in\mathcal{X}
$ be a weak solution of \eqref{problem} with $c>\sqrt{\frac{\mu}{2m}}$.
Then for every
\begin{equation*}
    0<\beta<\left(2m\mu-\frac{\mu^2}{c^2}\right)^{\frac12},
\end{equation*}
there exists a constant $C=C(u,c,\beta)>0$ such that
\begin{equation}\label{exponentialdecay}
    |u(x)|\le C e^{-\beta|x|}
\quad \text{for all }x\in\mathbb R^N.
\end{equation}
\end{lemma}
\begin{proof}
    Let $U$ be the unique extension of $u$ determined by \eqref{boundaryvalueproblem}.
    Taking the Fourier transform with respect to $x$, we obtain
    \begin{equation*}
        \widehat{U}(\xi,y)=\widehat{u}(\xi)e^{-\sqrt{|\xi|^2+m^2c^2}y}.
    \end{equation*}
    Hence,
    \begin{equation*}
        U(x,y)=\frac{1}{(2\pi)^{\frac{N}{2}}}\int_{\mathbb{R}^N}e^{ix\cdot\xi}e^{-\sqrt{|\xi|^2+m^2c^2}y}\widehat{u}(\xi)\,d\xi.
    \end{equation*}
    Moreover,
    \begin{equation}\label{expdecay1}
        |U(x,y)|\leq\frac{1}{(2\pi)^{\frac{N}{2}}}\left(\int_{\mathbb{R}^N}e^{-2\sqrt{|\xi|^2+m^2c^2}y}\,d\xi\right)^{\frac{1}{2}}|u|_2.
    \end{equation}
    For $y\geq1$, we have
    \begin{equation*}
        \int_{\mathbb{R}^N}e^{-2\sqrt{|\xi|^2+m^2c^2}y}\,d\xi=e^{-2mcy}\int_{\mathbb{R}^N}e^{-2(\sqrt{|\xi|^2+m^2c^2}-mc)y}\,d\xi\leq Ce^{-2mcy}.
    \end{equation*}
    Consequently,
    \begin{equation*}
        \sup_{x\in\mathbb{R}^N}|U(x,y)|\leq
         C|u|_2e^{-mcy},\quad y\geq1.
    \end{equation*}
    Thus, for every $\lambda\in(0,mc)$,
    \begin{equation*}
        \sup_{x\in\mathbb{R}^N}|U(x,y)|e^{\lambda y}\leq C|u|_2e^{-(mc-\lambda)y}\to0\quad\text{as }y\to\infty.
    \end{equation*}
    Besides, since $U\in C^{0,\alpha}(\mathbb{R}^N\times [0,\infty))\cap W^{1,q}(\mathbb{R}^N\times (0,R))$ for all $R>0$ and $q\in (2,\infty)$, it follows that
    \begin{equation}\label{yleq1}
        \sup_{0\leq y\leq R}|U(x,y)|\to0\quad\text{as }|x|\to\infty.
    \end{equation}
    Combining these estimates, we conclude that for every $\lambda\in(0,mc)$,
    \begin{equation}\label{expdecay2}
        |U(x,y)|e^{\lambda y}\to0\quad\text{as}~|x|+y\to\infty.
    \end{equation}
    Let
    \begin{equation*}
        0<\beta<\left(2m\mu-\frac{\mu^2}{c^2}\right)^{\frac{1}{2}}.
    \end{equation*} 
    Set $\theta:=\sqrt{m^2c^2-\beta^2}$. Then 
    \begin{equation}\label{expdecay3}
        0<\theta<mc,\quad   c\theta=\sqrt{m^2c^4-\beta^2c^2}> mc^2-\mu.
    \end{equation}
    By \eqref{yleq1}, $|u(x)|\to0$ as $|x|\to\infty$.
    Thus, we can take $R>0$ large enough such that
     \begin{equation}\label{0lequleq1}
        |u(x)|\leq 1\quad\text{for }|x|\geq R.
     \end{equation}
    Since $\theta\in (0,mc)$, it follows from \eqref{expdecay2} that
    \begin{equation*}
        |U(x,y)|e^{\theta y}\to0\quad\text{as }|x|+y\to\infty.
    \end{equation*}
    Therefore, we can choose $A>0$ such that
    \begin{equation}\label{expdecay4}
        |U(x,y)|\leq Ae^{-\theta y}\quad\text{for } |x|= R \text{ and }y\geq0.
    \end{equation}
     Define 
   \begin{equation*}
       W(x,y):=Ae^{-\beta(|x|-R)}e^{-\theta y}\quad\text{for }|x|\geq R\text{ and }y\geq0.
   \end{equation*}
   A direct computation gives
    \begin{equation*}
        (-c^2\Delta_{x,y}+m^2c^4)W=\frac{Ac^2\beta (N-1)}{|x|}e^{-\beta(|x|-R)}e^{-\theta y}\geq0.
    \end{equation*}
   Fix $\sigma\in\{-1,1\}$ and
   set $Z:=\sigma U-W$.
   Then 
    \begin{equation}\label{expdecay5}
        (-c^2\Delta_{x,y}+m^2c^4)Z\leq0\quad\text{for }|x|\geq R\text{ and }y>0.
    \end{equation}
    For $|x|\geq R$ with $Z(x,0)>0$, we have
    \begin{equation*}
        0<W(x,0)<\sigma u(x)=|u(x)|\leq1.
    \end{equation*}
    In particular, $\sigma u(x)\log |u(x)|\leq0$.
    Hence, by \eqref{expdecay3} and \eqref{0lequleq1}, for $|x|\geq R$ with $Z(x,0)>0$,
   \begin{equation}\label{expdecay6}
       \begin{aligned}
       -c\frac{\partial Z}{\partial y}(x,0)&=(mc^2-\mu)\sigma u+\sigma u\log |u|-c\theta W(x,0)\\
       &\leq c\theta\sigma u-c\theta W(x,0)=c\theta Z(x,0).
   \end{aligned}
   \end{equation}
   Thus, although \eqref{expdecay6} is only asserted where $Z(x,0)>0$, it yields
   \begin{equation*}
        \left(-c\frac{\partial Z}{\partial y}(x,0)\right)Z^+(x,0)\leq c\theta |Z^+(x,0)|^2\quad\text{for }|x|\geq R.
   \end{equation*}
   Set $\Omega_R:=\left\{(x,y)\in\mathbb{R}^{N+1}_+:|x|>R\right\}$. Taking $Z^+$ as a test function and using \eqref{expdecay5}--\eqref{expdecay6}, we obtain
   \begin{equation}\label{expdecay7}
   \begin{aligned}
        \int_{\Omega_R}\left(c^2|\nabla Z^+|^2+m^2c^4|Z^+|^2\right)\,dx\,dy&\leq\int_{|x|> R}c^2\left(-\frac{\partial Z}{\partial y}(x,0)\right)Z^+(x,0)\,dx\\
        &\leq c^2\theta\int_{|x|> R}|Z^+(x,0)|^2\,dx.
   \end{aligned}
   \end{equation}
   On the other hand, extending
   $Z^+$ by zero across $\{|x|=R\}$, we may apply 
   Lemma \ref{norms} to get
   \begin{equation}\label{expdecay8}
       \frac{1}{c}\int_{\Omega_R}\left(c^2|\nabla Z^+|^2+m^2c^4|Z^+|^2\right)\,dx\,dy\geq mc^2\int_{|x|> R}|Z^+(x,0)|^2\,dx.
   \end{equation}
   Combining \eqref{expdecay7} and \eqref{expdecay8} gives
   \begin{equation*}
       mc\int_{|x|> R}|Z^+(x,0)|^2\,dx\leq \theta\int_{|x|> R}|Z^+(x,0)|^2\,dx.
   \end{equation*}
   Since $\theta<mc$, we deduce that $Z^+(x,0)=0$ for $|x|>R$.
   Hence, 
   \begin{equation*}
       \sigma u(x)\leq W(x,0)= Ae^{-\beta(|x|-R)}\quad\text{for }|x|>R.
   \end{equation*}
   As $\sigma\in\{-1,1\}$ was arbitrary, we conclude that
   \begin{equation*}
       |u(x)|\leq Ae^{-\beta(|x|-R)}\quad\text{for }|x|>R.
   \end{equation*}
    Finally, since $u\in L^\infty(\mathbb{R}^N)$, taking $C:=e^{\beta R}\max\{A,|u|_\infty\}$
    gives
   \begin{equation*}
       |u(x)|\leq C e^{-\beta|x|}\quad\text{for all } x\in\mathbb{R}^N.
   \end{equation*}
   This completes the proof.
\end{proof}

We conclude this section by proving the symmetry and strict radial monotonicity of positive solutions. Our argument is based on the moving-plane method applied to the local extension problem.
\begin{lemma}\label{radiallysymmetric}
    Let $u\in\mathcal{X}$ be a positive solution of \eqref{problem}. Then $u$ is radial with respect to some point $x_0\in\mathbb{R}^N$, and strictly decreasing with respect to the radial variable.
\end{lemma}
\begin{proof}
    Let $U$ be the unique extension of $u$ determined by \eqref{boundaryvalueproblem}. By Lemmas \ref{Linfty-regularity} and \ref{C2-regularity}, for every $\alpha\in(0,1)$,
    \begin{equation*}
        U\in L^\infty(\mathbb{R}^{N+1}_+)\cap C^{1,\alpha}(\mathbb{R}^N\times [0,\infty))\cap C^2(\mathbb{R}^{N+1}_+),
    \end{equation*}
    and $U$ is a classical solution of \eqref{boundaryvalueproblem1}. Moreover, it follows from \eqref{expdecay2} that
    \begin{equation*}
        \lim_{|x|+y\to\infty}|U(x,y)|=0.
    \end{equation*}
    In particular, $u(x)\to0$ as $|x|\to\infty$.
    For $\lambda\in\mathbb{R}$, define
    \begin{equation*}
        M_\lambda:=\left\{(\lambda, x_2,\ldots, x_N, y): x_2,\ldots, x_N\in\mathbb{R}, y\geq0\right\}
    \end{equation*}
    and
    \begin{equation*}
        R_\lambda:=\left\{(x_1,\ldots, x_N, y)\in\mathbb{R}^{N+1}_+: x_1>\lambda\right\}.
    \end{equation*}
    Set
    \begin{equation*}
        U_\lambda(x_1, x_2,\ldots, x_N,y):=U(2\lambda-x_1, x_2,\ldots, x_N,y)
    \end{equation*}
    and $W_\lambda:=U_\lambda-U$.
    Denote their traces by $u_\lambda(x):=U_\lambda(x,0)$, $w_\lambda(x):=W_\lambda(x,0)$.
    Then $W_\lambda$ satisfies
   \begin{equation*}
       \left\{
	\begin{aligned}
		&\left(-c^2\Delta_{x,y}+m^2c^4\right)W_\lambda(x,y)=0 &&\text{in}~R_\lambda,\\
		&-c\frac{\partial W_\lambda}{\partial y}(x,0)=\left(C_\lambda(x)+mc^2-\mu\right)w_\lambda(x)&&\text{for }x_1>\lambda,
	\end{aligned}
	\right.
   \end{equation*}
   where $C_\lambda(x):=\mathcal{D}_f(u_\lambda(x), u(x))$
   and 
   $\mathcal{D}_f$ is defined by \eqref{Df}.
   Set $W_\lambda^-:=\max\{-W_\lambda,0\}$, $w_\lambda^-:=\max\{-w_\lambda,0\}$.
   On the support of $w_\lambda^-$, we have
   $ 0<u_\lambda(x)<u(x)$.
   The mean value theorem yields that
   \begin{equation*}
       C_\lambda(x)=1+\log\theta_\lambda(x)\leq1+\log u(x)
   \end{equation*}
   for some $ u_\lambda(x)<\theta_\lambda(x)<u(x)$.
   Since $u(x)\to0$ as $|x|\to\infty$, for all sufficiently large $\lambda$,
   \begin{equation*}
       C_\lambda(x)+mc^2-\mu<0
   \end{equation*}
   whenever $w_\lambda^->0$ and $x_1>\lambda$.
   Now, taking $-W_\lambda^-$ as a test function, we find
   \begin{equation*}
       \int_{R_\lambda}\left(c^2|\nabla W_\lambda^-|^2+m^2c^4|W_\lambda^-|^2\right)\,dx\,dy=c\int_{x_1>\lambda}\left( C_\lambda(x)+mc^2-\mu\right)|w_\lambda^-|^2\,dx.
   \end{equation*}
   Observe that the right-hand side is nonpositive, while the left-hand side is nonnegative. Consequently, $W_\lambda^-\equiv0$ in $R_\lambda$.
   Thus, for every sufficiently large $\lambda$, $W_\lambda\geq0$ in $R_\lambda$.
   Define
   \begin{equation*}
       \eta:=\inf\{s\in\mathbb{R}: W_\lambda\geq0\text{ in }R_\lambda\text{ for every }\lambda\geq s\}.
   \end{equation*}
   We first prove that $\eta>-\infty$.
   Suppose instead that $\eta=-\infty$.
   Then $W_\lambda\geq0$ in $R_\lambda$
   for every $\lambda\in\mathbb{R}$. Given $s<t$, take $\lambda=\frac{s+t}{2}$.
   Since the reflection of $(t,x_2,\ldots,x_N)$ across $x_1=\lambda$ is $(s,x_2,\ldots,x_N)$, we obtain
   \begin{equation*}
       u(s,x_2,\ldots,x_N)\geq  u(t,x_2,\ldots,x_N).
   \end{equation*}
   Letting $s\to-\infty$ gives
   $u(t,x_2,\ldots,x_N)\leq0$, which contradicts the positivity of $u$. Hence $\eta$ is finite.
   For every $\lambda>\eta$, the definition of $\eta$ implies $W_\lambda\geq0$ in $R_\lambda$.
   Passing to the limit as $\lambda\downarrow\eta$ and using the continuity of $U$, we obtain $W_\eta\geq0$ in $R_\eta$.
   We claim that $W_\eta\equiv0$ in $R_\eta$.
   Suppose, to the contrary, that $W_{\eta}\not\equiv0$. By the strong maximum principle, $W_\eta>0$ in $\{(x,y)\in R_\eta:y>0\}$.
   We also have
   $w_\eta(x)\geq0$ for $x_1>\eta$.
   In fact,
   \begin{equation}\label{weta>0}
       w_\eta(x)>0\quad\text{for }x_1>\eta.
   \end{equation}
   Indeed, if $w_\eta(x_0)=0$ at some point with $(x_0)_1>\eta$, then $W_\eta$ attains its minimum at $(x_0,0)$. The Hopf lemma yields
   \begin{equation*}
       \frac{\partial W_\eta}{\partial y}(x_0,0)>0,
   \end{equation*}
   whereas the boundary condition gives
   \begin{equation*}
       -c\frac{\partial W_\eta}{\partial y}(x_0,0)=\left(C_\eta(x_0)+mc^2-\mu\right)w_\eta(x_0)=0,
   \end{equation*}
   which is a contradiction. This proves \eqref{weta>0}.
  By the definition of $\eta$, there exists a sequence $\lambda_j<\eta$ with $\lambda_j\to\eta$ such that $W_{\lambda_j}^{-}\not\equiv0$. Choose $r_0>0$ sufficiently large so that
   \begin{equation*}
       mc^2-\mu+1+\log u(x)<0\quad\text{for }|x|>r_0.
   \end{equation*}
   On the support of $w_{\lambda_j}^-$, we have
   \begin{equation*}
       C_{\lambda_j}(x)\leq1+\log u(x).
   \end{equation*}
   Consequently,
   \begin{equation*}
        C_{\lambda_j}(x)+mc^2-\mu<0\quad\text{for }|x|>r_0
   \end{equation*}
    on the support of $w_{\lambda_j}^-$. Set $E_j:=\{x\in B_{r_0}:x_1>\lambda_j, w_{\lambda_j}(x)<0\}$.
    Since $u$ is bounded, the positive part of $ C_{\lambda_j}(x)+mc^2-\mu$ is uniformly bounded on the support of $w_{\lambda_j}^-$. Taking $-W_{\lambda_j}^-$ as a test function, we obtain
    \begin{equation*}
        \int_{R_{\lambda_j}}\left(c^2|\nabla W_{\lambda_j}^-|^2+m^2c^4|W_{\lambda_j}^-|^2\right)\,dx\,dy=c\int_{x_1>\lambda_j}\left( C_{\lambda_j}(x)+mc^2-\mu\right)|w_{\lambda_j}^-|^2\,dx\leq C\int_{E_j}|w_{\lambda_j}^-|^2\,dx.
    \end{equation*}
    We next show that $|E_j|\to0$.
   Fix $\delta>0$. By \eqref{weta>0}, $w_\eta$ is strictly positive on the compact set $ K_\delta:=\overline{B_{r_0}}\cap\{x_1\geq \eta+\delta\}$.
    Moreover, since $u$ is continuous and $\lambda_j\to\eta$, we have $w_{\lambda_j}\to w_\eta$
    uniformly on $K_\delta$. Therefore, for all sufficiently large $j$, $w_{\lambda_j}>0$ on $K_\delta$.
    It follows that $ E_j\subset B_{r_0}\cap\{\lambda_j<x_1<\eta+\delta\}$.
     First letting $j\to\infty$ and then $\delta\to0$, we obtain $|E_j|\to0$.
    By H\"older's inequality, the Sobolev embedding theorem and Lemma \ref{norms},
    \begin{equation*}
         \int_{E_j}|w_{\lambda_j}^-|^2\,dx\leq |E_j|^{\frac{1}{N}}\left(\int_{\mathbb{R}^N}|w_{\lambda_j}^-|^{\frac{2N}{N-1}}\right)^{\frac{N-1}{N}}\leq C|E_j|^{\frac{1}{N}}\int_{R_{\lambda_j}}\left(c^2|\nabla W_{\lambda_j}^-|^2+m^2c^4|W_{\lambda_j}^-|^2\right)\,dx\,dy.
    \end{equation*}
    It then follows that
    \begin{equation*}
        \int_{R_{\lambda_j}}\left(c^2|\nabla W_{\lambda_j}^-|^2+m^2c^4|W_{\lambda_j}^-|^2\right)\,dx\,dy\leq C|E_j|^{\frac{1}{N}}\int_{R_{\lambda_j}}\left(c^2|\nabla W_{\lambda_j}^-|^2+m^2c^4|W_{\lambda_j}^-|^2\right)\,dx\,dy.
    \end{equation*}
    Since $C|E_j|^{\frac{1}{N}}<1$ for all sufficiently large $j$,
    we have $W_{\lambda_j}^-\equiv0$ in $R_{\lambda_j}$.
    This leads to a contradiction. Hence, $W_\eta\equiv0$ in $R_\eta$.
    Moreover, $U$, and consequently $u$, is symmetric with respect to the hyperplane $x_{1}=\eta$.

   We next prove strict monotonicity. For every $\lambda>\eta$, one has $W_\lambda\geq0$ in $R_\lambda$.
   Moreover, $W_\lambda\not\equiv0$. Otherwise, $U$ would be symmetric with respect to both $M_\lambda$ and $M_\eta$. Thus,
   \begin{align*}
       U(x_1,x_2,\ldots,x_N,y)&=U(2\eta-x_1,x_2,\ldots,x_N,y)\\
       &=U(2\lambda-(2\eta-x_1),x_2,\ldots,x_N,y)\\
       &=U(x_1+2(\lambda-\eta),x_2,\ldots,x_N,y).
   \end{align*}
  Consequently, $U$ would be periodic in the $x_1$-direction, contradicting $U(x,y)\to0$ as $|x|+y\to\infty$. The strong maximum principle and the Hopf lemma therefore give
   \begin{equation*}
       W_\lambda>0\quad\text{in }\{(x,y)\in R_\lambda:y>0\}\quad\text{and}\quad  w_\lambda(x)>0\quad\text{for }x_1>\lambda. 
   \end{equation*}
   Let $\eta\leq s<t$ and choose $\lambda=\frac{s+t}{2}$.
   Since the reflection of $(t,x_2,\ldots,x_N)$ with respect to $x_{1}=\lambda$ is $(s,x_2,\ldots,x_N)$, we obtain
   \begin{equation*}
       u(s,x_2,\ldots,x_N)-u(t,x_2,\ldots,x_N)=w_\lambda(t,x_2,\ldots,x_N)>0.
   \end{equation*}
   Hence $u$ is strictly decreasing in the $x_{1}$-direction for $x_{1}>\eta$. By symmetry, it is strictly increasing for $x_1<\eta$.

   Rotational invariance allows the same argument in every direction $e\in\mathbb{S}^{N-1}$.
   Thus $u$ is symmetric about a hyperplane $H_e=\{x\in\mathbb{R}^N:x\cdot e=\eta_e\}$ and is strictly decreasing away from it.
   Since $u$ is positive, continuous and tends to zero at infinity, it attains its maximum at some $x_{0}\in\mathbb R^{N}$. Strict monotonicity on the line $x_0+\mathbb{R}e$ forces $\eta_e=x_0\cdot e$. Hence every $H_e$ passes through $x_0$, so $u$ is invariant under all reflections through $x_0$ and is therefore radial about $x_0$. The directional strict monotonicity gives strict decrease in $|x-x_0|$.
\end{proof}

\textbf{Proof of Theorem \ref{result3}.} The conclusion follows directly from Lemmas \ref{Linfty-regularity}--\ref{radiallysymmetric}.

\section{Nonrelativistic limit, uniqueness and nondegeneracy}
We first derive estimates uniform in $c$ and prove that, after suitable translations, positive action ground states converge to the Gausson $\mathfrak{g}$ in $H^1(\mathbb{R}^N)$.
Combining this convergence with the nondegeneracy of $\mathfrak{g}$, we then establish uniqueness and nondegeneracy for all sufficiently large $c$.
We begin by recalling the limiting equation
    \begin{equation}\label{limitproblem}
    -\frac{1}{2m}\Delta u+\mu u=u\log |u|\quad\text{in }\mathbb{R}^N.
    \end{equation}
Every positive $C^2$ solution of \eqref{limitproblem} that tends to zero at infinity  is, up to a translation, the Gausson
\begin{equation*}
\mathfrak{g}(x)=e^{\mu+\frac{N}{2}}e^{-\frac{m}{2}|x|^2}.
\end{equation*}
For each $c\geq1$, let $u_c$ be an arbitrary positive action ground state of \eqref{problem}, whose existence follows from Theorem \ref{result1}.
By Theorem \ref{result3}, after a translation, we may assume that $u_c$ is radial about the origin and strictly decreasing with respect to $|x|$.
Let $U_c$ be its unique extension determined by \eqref{boundaryvalueproblem}. Then Lemma \ref{groundstates} yields
\begin{equation*}
    E_c(u_c)=I_c(U_c)=m_c.
\end{equation*}

\begin{lemma}\label{energybounded}
    It holds that
    \begin{equation*}
        \sup_{c\geq1}E_c(u_c)<\infty.
    \end{equation*}
\end{lemma}
\begin{proof}
Fix $\varphi\in\mathcal{S}(\mathbb{R}^N)\setminus\{0\}$, let $V_c$ be its unique extension, and choose $t_c>0$ so that $t_cV_c\in\mathcal{N}_c$. The equality cases in Lemma \ref{norms} give
\begin{equation*}
    \log t_c=\frac{\int_{\mathbb{R}^N}\left(P_c(\xi)+\mu\right)|\widehat{\varphi}(\xi)|^2\,d\xi-\int_{\mathbb{R}^N}\varphi^2\log|\varphi|\,dx}{\int_{\mathbb{R}^N}\varphi^2\,dx}.
\end{equation*}
Since
\begin{equation*}
    0\leq P_c(\xi)\leq\frac{|\xi|^2}{2m},
\end{equation*}
the numbers $t_c$ are bounded above uniformly in $c$. Therefore,
\begin{equation*}
    E_c(u_c)=m_c\leq I_c(t_cV_c)=\frac{t_c^2}{4}|\varphi|_2^2\leq C,
\end{equation*}
which proves the assertion.
\end{proof}

\begin{lemma}\label{1/2normbounded}
    It holds that
    \begin{equation}\label{1/2normboundedness}
        \sup_{c\geq1}\Vert u_c\Vert_{H^{1/2}(\mathbb{R}^N)}<\infty.
    \end{equation}
\end{lemma}
\begin{proof}
   For every $c\geq1$, we have
    \begin{equation*}
        \sqrt{\frac{|\xi|^2}{c^2}+m^2}
        \leq |\xi|+m.
    \end{equation*}
    Hence,
    \begin{equation*}
         P_c(\xi)
        \geq
        \frac{|\xi|^2}{|\xi|+2m}.
    \end{equation*}
    Therefore, there exists a constant $C_{m,\mu}>0$, independent of $c$, such that
    \begin{equation}\label{symbol-coercive}
        P_c(\xi)+\mu
        \geq
        C_{m,\mu}(1+|\xi|^2)^\frac{1}{2}.
    \end{equation}
   Since $u_c$ solves \eqref{problem},  the Plancherel theorem gives
    \begin{equation}\label{fls1}
        \begin{aligned}
        \int_{\mathbb{R}^N}u_c^2\log |u_c|\,dx&=\int_{\mathbb{R}^N}\left(P_c(\xi)+\mu
    \right)|\widehat{u_c}(\xi)|^2\,d\xi\\
    &\geq C_{m,\mu}\int_{\mathbb{R}^N}\left(1+|\xi|^2\right)^\frac{1}{2}|\widehat{u_c}(\xi)|^2\,d\xi\\
    &= C_{m,\mu}\Vert u_c\Vert^2_{H^{1/2}(\mathbb{R}^N)}.
    \end{aligned}
    \end{equation}
    On the other hand, testing \eqref{problem} with $u_c$ gives
 \begin{equation*}
     E_c(u_c)=\frac{1}{4}\int_{\mathbb{R}^N}u_c^2\,dx.
 \end{equation*}
 Then Lemma \ref{energybounded} implies that $\{u_c\}$ is bounded in $L^2(\mathbb{R}^N)$. Let $\alpha>0$ be small enough, using the fractional logarithmic Sobolev inequality \eqref{fls}, we obtain
    \begin{equation}\label{fls2}
        \begin{aligned}
        \int_{\mathbb{R}^N}u_c^2\log |u_c|\,dx&\leq\frac{1}{2}C_{m,\mu}\Vert u_c\Vert_{H^{1/2}(\mathbb{R}^N)}^2+C_1(\log |u_c|_2^2+C_2)|u_c|_2^2\\
        &\leq\frac{1}{2}C_{m,\mu}\Vert u_c\Vert_{H^{1/2}(\mathbb{R}^N)}^2+C_3.
    \end{aligned}
    \end{equation}
    Combining \eqref{fls1} and \eqref{fls2}, we conclude
    \begin{equation*}
        \sup_{c\geq1}\Vert u_c\Vert_{H^{1/2}(\mathbb{R}^N)}<\infty.
    \end{equation*}
\end{proof}

Lemma \ref{expdecay} gives exponential decay for each fixed $c>\sqrt{\frac{\mu}{2m}}$. To make this decay uniform with respect to $c$, we shall establish a uniform $L^\infty$-bound and choose both the radius $R$ and the boundary constant $A$ in the comparison argument independently of $c$. We then fix $0<\beta<\sqrt{2m\mu}$. For all sufficiently large $c$, this choice satisfies the condition on $\beta$ in Lemma \ref{expdecay}, allowing the same comparison argument to yield uniform exponential decay.
\begin{lemma}\label{uniformexponentialdecay}
    There exist constants $c_0\geq1$, $\beta>0$ and $C>0$, independent of $c$, such that, for every $c\geq c_0$,
    \begin{equation*}
        0<u_c(x)\leq Ce^{-\beta|x|}\quad\text{for all }x\in\mathbb{R}^N.
    \end{equation*}
\end{lemma}
\begin{proof}
   Recall that $u_c$ is a positive action ground state of
   \begin{equation*}
       (P_c(D)+\mu)u_c=u_c\log u_c\quad\text{in }\mathbb{R}^N,
   \end{equation*}
   which is radial about the origin and strictly decreasing with respect to $|x|$.
   We divide the proof into several steps.
   
     \medskip
    \noindent
    \textbf{Step 1.} We first establish a uniform $L^\infty$-bound.

    By Lemma \ref{1/2normbounded} and the Sobolev embedding theorem,
   \begin{equation*}
       \sup_{c\geq1}|u_c|_{\frac{2N}{N-1}}<\infty.
   \end{equation*}
   Choose $0<\delta\leq\frac{1}{N-1}$ and define $ G_c:=u_c^\delta\chi_{\{u_c\geq1\}}$.
   Since $2N\delta\leq
   \frac{2N}{N-1}$,
   it follows that
   \begin{equation*}
         \sup_{c\geq1}|G_c|_{2N}<\infty.
   \end{equation*}
   Let $p\geq2$. We use the bilinear representation of $P_c(D)$ from \cite[Formula (2.1)]{AscioneLhorinczi2024}. Together with
   \begin{equation*}
       (a-b)\left(a^{p-1}-b^{p-1}\right)\geq\frac{4(p-1)}{p^2}\left(a^\frac{p}{2}-b^\frac{p}{2}\right)^2,\quad a,b\geq0,
   \end{equation*}
   it yields the Stroock-Varopoulos estimate
   \begin{equation}\label{SVestimate}
       \left\langle P_c(D)u_c,u_c^{p-1}\right\rangle\geq a_p\left\langle P_c(D)u_c^\frac{p}{2}, u_c^\frac{p}{2}\right\rangle,\quad a_p:=\frac{4(p-1)}{p^2}.
   \end{equation}
   The test function is admissible because Lemma \ref{Linfty-regularity} and the Lipschitz continuity of $s\mapsto s^q$ on bounded intervals give $u_c^q\in H^{1/2}(\mathbb{R}^N)$ for every $q\geq1$.
   
   Testing \eqref{problem} with $u_c^{p-1}$ and using
   $a_p\leq1$, \eqref{SVestimate},
   we find
   \begin{equation*}
         a_p\left\langle \left(P_c(D)+\mu\right)u_c^{\frac{p}{2}},u_c^{\frac{p}{2}}\right\rangle
       \leq\left\langle P_c(D)u_c,u_c^{p-1}\right\rangle+\mu\left|u_c^{\frac{p}{2}}\right|_2^2=\int_{\mathbb{R}^N}u_c^p\log u_c\,dx.
   \end{equation*}
   Moreover, using the fact
   \begin{equation*}
       \log s\leq\frac{s^\delta}{\delta},\quad\forall s\geq1,
   \end{equation*}
   we obtain
   \begin{equation*}
        a_p\left\langle \left(P_c(D)+\mu\right)u_c^{\frac{p}{2}},u_c^{\frac{p}{2}}\right\rangle\leq\int_{u_c\geq1}u_c^p\log u_c\,dx\leq
        \frac{1}{\delta}\int_{u_c\geq1}u_c^{p+\delta}\,dx=\frac{1}{\delta}\int_{\mathbb{R}^N}G_cu_c^p\,dx.
   \end{equation*}
    H\"older's inequality, \eqref{symbol-coercive} and the Sobolev embedding theorem then give
   \begin{equation*}
       a_pC\left|u_c^{
       \frac{p}{2}
       }\right|_{\frac{2N}{N-1}}^2\leq\frac{1}{\delta}\int_{\mathbb{R}^N}G_cu_c^p\,dx\leq C|G_c|_{2N}\left|u_c^{
       \frac{p}{2}}\right|_2\left|u_c^{
       \frac{p}{2}
       }\right|_{\frac{2N}{N-1}}.
   \end{equation*}
   Consequently,
   \begin{equation*}
       |u_c|_{\frac{pN}{N-1}}\leq (Cp)^{\frac{2}{p}}|u_c|_p,\quad p\geq2.
   \end{equation*}
   Iterating this inequality with $ p_j=2\left(\frac{N}{N-1}\right)^j$,
   we obtain
   \begin{equation*}
       |u_c|_{p_k}\leq\left(\prod_{j=0}^{k-1}(Cp_j)^{\frac{2}{p_j}}\right)|u_c|_2.
   \end{equation*}
   Since
   \begin{equation*}
       \log \left(\prod_{j=0}^{k-1}(Cp_j)^{\frac{2}{p_j}}\right)=2\sum_{j=0}^{k-1}\frac{\log(Cp_j)}{p_j}\quad\text{and}\quad   \sum_{j=0}^{\infty}\frac{\log (Cp_j)}{p_j}<\infty,
   \end{equation*}
   we conclude that
   \begin{equation}\label{uniformLinftyestimate}
       \sup_{c\geq1}|u_c|_\infty\leq M_\infty
   \end{equation}
   for some $M_\infty>0$ independent of $c$.
    
     \medskip
    \noindent
    \textbf{Step 2.} 
    Next, we show a uniform control for the extension.
    
   Since $u_c$ is radially decreasing, for every $r>0$,
    \begin{equation*}
        u_c^2(r)\leq\frac{\int_{B_r}u_c^2(x)\,dx}{|B_r|}.
    \end{equation*}
    By Lemma \ref{1/2normbounded}, there exists a constant $M_2>0$, independent of $c$, such that
    \begin{equation*}
        \sup_{c\geq1}|u_c|_2\leq M_2.
    \end{equation*}
    As a result,
    \begin{equation}\label{uniformlyto0}
        u_c(r)\leq \frac{M_2}{|B_1|^{
        \frac{1}{2}
        }}r^{-\frac{N}{2}}.
    \end{equation}
   Hence, we can choose $R>0$, independent of $c$, such that $0<u_c(x)\leq1$ for $|x|\geq R$.
   Let $U_c$ be the unique extension of $u_c$. Its Fourier representation is
  \begin{equation*}
        U_c(x,y)=\frac{1}{(2\pi)^{\frac{N}{2}}}\int_{\mathbb{R}^N}e^{ix\cdot\xi}e^{-\sqrt{|\xi|^2+m^2c^2}y}\widehat{u_c}(\xi)\,d\xi.
    \end{equation*}
    From the identity
    \begin{equation*}
        e^{-\sqrt{|\xi|^2+m^2c^2}y}
        =
       \frac{y}{2\sqrt{\pi}}
        \int_{0}^{\infty}
         t^{-\frac{3}{2}}
         e^{-\frac{y^2}{4t}}
         e^{-m^2c^2t}
         e^{-t|\xi|^2}\,dt,
    \end{equation*}
    it follows that
    \begin{equation*}
        U_c(x,y)=\left(K_{c,y}*u_c\right)(x),
    \end{equation*}
    where 
    \begin{equation*}
        K_{c,y}(x)=
       \frac{y}{2\sqrt{\pi}}
       \int_{0}^{\infty}
         t^{-\frac{3}{2}}
         e^{-\frac{y^2}{4t}}
        e^{-m^2c^2t}
        H_t(x)\,dt,\quad  H_t(x)
        =
        \frac{1}{(4\pi t)^{\frac{N}{2}}}
         e^{-\frac{|x|^2}{4t}}.
    \end{equation*}
    Since both $K_{c,y}$ and $u_c$ are nonnegative, radial, and radially decreasing, their convolution $U_c(\cdot,y)$ has the same properties.
    By the Plancherel theorem,
    \begin{equation*}
        |U_c(\cdot,y)|_2^2=\int_{\mathbb{R}^N}e^{-2\sqrt{|\xi|^2+m^2c^2}y}|\widehat{u_c}(\xi)|^2\,d\xi\leq e^{-2mcy}|u_c|_2^2.
    \end{equation*}
    Thus,
    \begin{equation*}
        |U_c(\cdot,y)|_2\leq M_2e^{-mcy}.
    \end{equation*}
    Using the radial monotonicity, we find
    \begin{equation*}
        U_c(R,y)^2|B_R|\leq\int_{B_R}U_c(x,y)^2\,dx\leq |U_c(\cdot,y)|_2^2.
    \end{equation*}
    Therefore,
    \begin{equation}\label{UcRy}
        U_c(R,y)\leq Ae^{-mcy},
    \end{equation}
    where
    \begin{equation*}
        A:=\frac{M_2}{|B_R|^{\frac{1}{2}}}
    \end{equation*}
    is independent of $c$.

     \medskip
    \noindent
    \textbf{Step 3.} 
   Finally, we prove the uniform exponential decay.

    Fix $0<\beta<\sqrt{2m\mu}$ and choose $c_0\geq1$ large enough such that, for every $c\geq c_0$,
    \begin{equation*}
        mc^2>\mu,\quad\beta<\left(2m\mu-\frac{\mu^2}{c^2}\right)^{\frac{1}{2}}.
    \end{equation*}
    Define $\theta_c:=\sqrt{m^2c^2-\beta^2}$.
    Then $0<\theta_c<mc$ and $c\theta_c>mc^2-\mu$.
    For $|x|\geq R$ and $y\geq0$, define $ W_c(x,y):=Ae^{-\beta(|x|-R)}e^{-\theta_cy}$.
    By \eqref{UcRy}, we have
    \begin{equation*}
        U_c(R,y)\leq Ae^{-mcy}\leq Ae^{-\theta_cy}=W_c(R,y).
    \end{equation*}
    A direct computation gives
    \begin{equation*}
          (-c^2\Delta_{x,y}+m^2c^4)W_c=\frac{c^2\beta (N-1)}{|x|}W_c\geq0.
    \end{equation*}
    Set $\Omega_R:=\left\{(x,y)\in\mathbb{R}^{N+1}_+:|x|>R\right\}$
    and define $Z_c:=U_c-W_c$.
    It holds that
    \begin{equation*}
        (-c^2\Delta_{x,y}+m^2c^4)Z_c\leq0\quad\text{in }\Omega_R\quad\text{and}\quad Z_c\leq0\quad\text{on }\{|x|=R, y\geq0\}.
    \end{equation*}
    On $y=0$, using $0<u_c\leq1$ for $|x|\geq R$, we have $u_c\log u_c\leq0$.
    Therefore,
    \begin{align*}
       -c\frac{\partial Z_c}{\partial y}(x,0)&=(mc^2-\mu)u_c+u_c\log u_c-c\theta_c W_c(x,0)\\
       &\leq c\theta_c u_c-c\theta_c W_c(x,0)=c\theta_c Z_c(x,0)\quad\text{for }|x|\geq R.
    \end{align*}
    Thus, all the hypotheses of the comparison argument used in the proof of Lemma \ref{expdecay} are satisfied. Repeating the estimates in \eqref{expdecay7} and \eqref{expdecay8} gives
   \begin{equation*}
       (mc-\theta_c)\int_{|x|> R}|Z_c^+(x,0)|^2\,dx\leq 0.
   \end{equation*}
   Since $\theta_c<mc$, we deduce that $Z_c^+(x,0)=0$ for $|x|>R$.
   Consequently, 
   \begin{equation*}
       u_c(x)\leq W_c(x,0)= Ae^{-\beta(|x|-R)}\quad\text{for }|x|>R.
   \end{equation*}
    For $|x|\leq R$, the uniform $L^\infty$-estimate \eqref{uniformLinftyestimate} gives
   \begin{equation*}
       u_c(x)\leq M_\infty\leq M_\infty e^{\beta R}e^{-\beta|x|}.
   \end{equation*}
   Taking $C:=e^{\beta R}\max\{A,M_\infty\}$,
   we conclude that
   \begin{equation*}
      0<u_c(x)\le Ce^{-\beta|x|}\quad\text{for all } x\in\mathbb{R}^N\text{ and every }c\geq c_0.
   \end{equation*}
   This completes the proof.
   \end{proof}

\begin{lemma}\label{H1normboundedness}
    There exists $c_1\geq c_0$ such that
    \begin{equation}\label{H1normbounded}
         \sup_{c\geq c_1}\Vert u_c\Vert_{H^1(\mathbb{R}^N)}<\infty.
    \end{equation}
\end{lemma}
\begin{proof}
Choose $0<\delta<\min\left\{1,\frac{2}{N-1}\right\}$,
so that $2+\delta\leq\frac{2N}{N-1}$.
By Lemma \ref{1/2normbounded} and the Sobolev embedding theorem, we have
\begin{equation*}
    \sup_{c\geq1}|u_c|_p<\infty,\quad 2\leq p\leq\frac{2N}{N-1}.
\end{equation*}
On the other hand, by Lemma \ref{uniformexponentialdecay},
\begin{equation*}
    \sup_{c\geq c_0}\int_{\mathbb{R}^N}|u_c|^{2-\delta}\,dx\leq C^{2-\delta}\int_{\mathbb{R}^N}e^{-\beta(2-\delta)|x|}\,dx<\infty.
\end{equation*}
Consequently,
\begin{equation}\label{H1normbounded1}
     \begin{aligned}
     \int_{\mathbb{R}^N}u_c^2\left(\log |u_c|\right)^2\,dx&=\int_{u_c\leq1}u_c^2\left(\log |u_c|\right)^2\,dx+\int_{u_c>1}u_c^2\left(\log |u_c|\right)^2\,dx\\
     &\leq C_\delta\left(\int_{\mathbb{R}^N}|u_c|^{2-\delta}\,dx+\int_{\mathbb{R}^N}|u_c|^{2+\delta}\,dx\right)\leq M,\quad c\geq c_0,
 \end{aligned}
 \end{equation}
 for some constant $M>0$ independent of $c$,
 which implies that $u_c\log|u_c|\in L^2(\mathbb{R}^N)$.
    Since $u_c$ solves \eqref{problem}, we have
    \begin{equation}\label{uc'sequation}
        \sqrt{-c^2\Delta+m^2c^4}u_c=(mc^2-\mu)u_c+u_c\log |u_c|.
    \end{equation}
    The right-hand side belongs to $L^2(\mathbb{R}^N)$. Hence, by the Plancherel theorem,
    \begin{equation*}
        \int_{\mathbb{R}^N}\left(c^2|\xi|^2+m^2c^4\right)|\widehat{u_c}(\xi)|^2\,d\xi<\infty.
    \end{equation*}
    Therefore, $u_c\in H^1(\mathbb{R}^N)$ for every $c\geq c_0$.

    It remains to establish the uniform boundedness in $H^1(\mathbb{R}^N)$.
    Taking the $L^2$-norm on both sides of the equation above, we obtain
    \begin{equation*}
        c^2|\nabla u_c|_2^2+m^2c^4|u_c|_2^2=(mc^2-\mu)^2|u_c|_2^2+2(mc^2-\mu)\int_{\mathbb{R}^N}u_c^2\log|u_c|\,dx+\int_{\mathbb{R}^N}u_c^2\left(\log |u_c|\right)^2\,dx.
    \end{equation*}
    Take $c_1\geq\max\left\{c_0,\sqrt{\frac{\mu}{m}}\right\}$.
    Then there exists $C>0$, independent of $c\geq c_1$, such that
    \begin{equation}\label{H1normbounded2}
        \begin{aligned}
         |\nabla u_c|_2^2+2m\mu|u_c|_2^2&=\frac{\mu^2}{c^2}|u_c|_2^2+2\left(m-\frac{\mu}{c^2}\right)\int_{\mathbb{R}^N}u_c^2\log|u_c|\,dx+\frac{1}{c^2}\int_{\mathbb{R}^N}u_c^2\left(\log |u_c|\right)^2\,dx\\
         &\leq \frac{\mu^2}{c^2}|u_c|_2^2+2C_\delta\left(m-\frac{\mu}{c^2}\right)\int_{\mathbb{R}^N}|u_c|^{2+\delta}\,dx+\frac{M}{c^2}\leq C.
    \end{aligned}
    \end{equation}
     As a result,
     \begin{equation*}
         \sup_{c\geq c_1}\Vert u_c\Vert_{H^1(\mathbb{R}^N)}<\infty.
     \end{equation*}
\end{proof}

\begin{lemma}\label{weaksolution}
    Let $\{u_c\}_{c\geq c_1}$ be the family of positive action ground states fixed above. Then up to a subsequence, there exists $v\in H^1(\mathbb{R}^N)$ such that $u_c\rightharpoonup v$ in $H^1(\mathbb{R}^N)$. Moreover, $v\in\mathcal{X}$ and $v$ is a weak solution of \eqref{limitproblem}.
\end{lemma}
\begin{proof}
    By Lemma \ref{H1normboundedness}, the family $\{u_c\}_{c\geq c_1}$ is bounded in $H^1(\mathbb{R}^N)$. Hence, after passing to a subsequence, there exists $v\in H^1(\mathbb{R}^N)$ such that
    \[
      	\begin{cases}
      		u_c\rightharpoonup v\text{ in }H^1(\mathbb{R}^N),\\
      		u_c\to v\text{ in }L_{loc}^q(\mathbb{R}^N),\quad 1\leq q<2^*,\\
      		u_c\to v\quad \text{a.e. in }\mathbb{R}^N.
      	\end{cases}
      	\]
        Since $u_c$ is a weak solution of \eqref{problem}, for every $\varphi\in C_c^\infty(\mathbb{R}^N),$
     \begin{equation}\label{veryweaksolution0}
         \int_{\mathbb{R}^N}u_c\left(P_c(D)+\mu\right)\varphi \,dx=\int_{\mathbb{R}^N}u_c\varphi\log |u_c|\,dx.
     \end{equation}
     Let $K$ be the support of $\varphi$. Noting that for every $q\in (2,2^*)$, there exists $C_q>0$ such that
     \begin{equation*}
         |s\log |s||\leq C_q(1+|s|^{q-1}),\quad s\in\mathbb{R}.
     \end{equation*}
    It follows from \cite[Theorem A.4]{Willem1996} that
     \begin{equation*}
         u_c\log |u_c|\to v\log |v|\quad\text{in }L^{q'}(K),
     \end{equation*}
     where $q'=\frac{q}{q-1}$. Thus, as $c\to\infty$,
     \begin{equation}\label{veryweaksolution1}
         \left|\int_{\mathbb{R}^N}\left(u_c\log |u_c|-v\log |v|\right)\varphi\,dx\right|\leq\left| u_c\log |u_c|-v\log |v|\right|_{L^{q'}(K)}|\varphi|_{L^q(K)}\to0.
     \end{equation}
     Besides, since $u_c\rightharpoonup v$ in $L^2(\mathbb{R}^N)$,
     we also have
\begin{equation}\label{veryweaksolution2}
         \int_{\mathbb{R}^N}(u_c-v)\varphi\,dx\to0.
     \end{equation}
     Since
    \begin{equation*}
        P_c(\xi)=\frac{|\xi|^2}{\sqrt{\frac{|\xi|^2}{c^2}+m^2}+m}\to\frac{|\xi|^2}{2m}\quad\text{and}\quad 0\leq P_c(\xi)\leq\frac{|\xi|^2}{2m},
    \end{equation*}
    Plancherel's theorem and dominated convergence show that
     \begin{equation}\label{veryweaksolution3}
         P_c(D)\varphi\to-\frac{1}{2m}\Delta \varphi\quad\text{in }L^2(\mathbb{R}^N).
     \end{equation}
   Together with the $L^2$-boundedness and weak convergence of $u_c$, this gives
    \begin{equation}\label{veryweaksolution4}
           \int_{\mathbb{R}^N}u_cP_c(D)\varphi\,dx\to-\frac{1}{2m}\int_{\mathbb{R}^N}v\Delta\varphi\,dx.
    \end{equation}
    Combining this with \eqref{veryweaksolution0}--\eqref{veryweaksolution2} yields
    \begin{equation*}
        \int_{\mathbb{R}^N}\left(-\frac{1}{2m}v\Delta\varphi+\mu v\varphi\right)\,dx=\int_{\mathbb{R}^N}v\varphi\log|v|\,dx.
    \end{equation*}
    Moreover, since $v\in H^1(\mathbb{R}^N)$, integration by parts in the weak sense yields
    \begin{equation*}
         \int_{\mathbb{R}^N}\left(\frac{1}{2m}\nabla v\cdot\nabla\varphi+\mu v\varphi\right)\,dx=\int_{\mathbb{R}^N}v\varphi\log|v|\,dx,\quad\forall\varphi\in C_c^\infty(\mathbb{R}^N).
    \end{equation*}
    It remains to show that $v\in\mathcal{X}$.
    Choose the same $\delta>0$ as in the proof of Lemma \ref{H1normboundedness}. The uniform boundedness of $u_c$ in $L^{2-\delta}(\mathbb{R}^N)$ and $L^{2+\delta}(\mathbb{R}^N)$ yields
    \begin{equation*}
        \sup_{c\geq c_1}\int_{\mathbb{R}^N}u_c^2|\log|u_c||\,dx<\infty.
    \end{equation*}
    Since $u_c\to v$ a.e. in $\mathbb{R}^N$,
    Fatou's lemma gives
    \begin{equation*}
        \int_{\mathbb{R}^N}v^2|\log|v||\,dx\leq \liminf_{c\to\infty}\int_{\mathbb{R}^N}u_c^2|\log|u_c||\,dx<\infty.
    \end{equation*}
    Thus, $v\in\mathcal{X}$, and
    therefore $v$ is a weak solution of \eqref{limitproblem}.
\end{proof}

We now proceed to the proof of Theorem \ref{result2}.
\begin{proof}
    Let $\{u_c\}_{c\geq c_1}$ be the family of positive action ground states fixed at the beginning of this section, with each $u_c$ radial about the origin and strictly decreasing as a function of $|x|$.
     By Lemma
    \ref{weaksolution}, up to a subsequence, there exists
    $v\in H^1(\mathbb{R}^N)$ such that $u_c\rightharpoonup v$ in $H^1(\mathbb{R}^N)$, 
    and $v\in\mathcal{X}$ is a weak solution of \eqref{limitproblem}.

    We first show that $v$ is nontrivial. Choose the same $\delta>0$ as in the proof of Lemma \ref{H1normboundedness}. It follows from  Lemma \ref{Sobolevembedding} and \eqref{fls1} that
    \begin{equation*}
         |u_c|_{2+\delta}^2\leq C_S\Vert u_c\Vert_{H^{1/2}(\mathbb{R}^N)}^2
        \leq
        \frac{C_S}{C_{m,\mu}}\int_{\mathbb{R}^N}u_c^2\log |u_c|\,dx
        \leq
        C|u_c|_{2+\delta}^{2+\delta}.
    \end{equation*}
    Thus, there exists $\rho>0$, independent of $c$, such that $|u_c|_{2+\delta}\geq\rho>0$.
    Since $u_c$ is radial and bounded in $H^{1/2}(\mathbb{R}^N)$, the compact embedding in Lemma \ref{Sobolevembedding} yields, after passing to a further subsequence, $u_c\to v$ in
    $L^{2+\delta}(\mathbb{R}^N)$.
    Consequently,
    \begin{equation*}
        |v|_{2+\delta}=\lim_{c\to\infty}|u_c|_{2+\delta}\geq\rho,
    \end{equation*}
    and hence $v\not\equiv0$.
    
    Since $u_c>0$ and $u_c\to v$ a.e. in $\mathbb{R}^N$, we have $v\geq0$. Besides, $v$ is radial about the origin. 
    By Lemma \ref{uniformexponentialdecay} and the almost everywhere convergence, we have
    \begin{equation*}
        0\leq v(x)\leq Ce^{-\beta|x|}\quad\text{for a.e. }x\in\mathbb{R}^N. 
    \end{equation*}
    In particular, $v\in L^\infty(\mathbb{R}^N)$,
    and standard local elliptic regularity yields $v\in C^2(\mathbb{R}^N)$.
    The positivity argument in \cite{dAvenia2014} also implies that $v$ is strictly positive. 
    By continuity, the preceding bound holds everywhere, and hence $v(x)\to0$ as $|x|\to\infty$.
     By the classification of positive $C^2$ solutions that vanish at infinity, 
    $v$ is a translate of the Gausson $\mathfrak{g}$.
    Since $v$ is radial about the origin, we deduce that $v=\mathfrak{g}$.
     It remains to prove the strong convergence in $H^1(\mathbb{R}^N).$ Since $\mathfrak{g}$ solves \eqref{limitproblem}, one has
    \begin{equation}\label{utog1}
        \int_{\mathbb{R}^N}\left(\frac{1}{2m}|\nabla\mathfrak{g}|^2+\mu\mathfrak{g}^2\right)\,dx=\int_{\mathbb{R}^N}\mathfrak{g}^2\log \mathfrak{g}\,dx.
    \end{equation}
    Besides, the weak convergence $u_c\rightharpoonup\mathfrak{g}$ in $H^1(\mathbb{R}^N)$
    yields that, as $c\to\infty$,
    \begin{equation}\label{utog2}
        \int_{\mathbb{R}^N}\left(\frac{1}{2m}\nabla u_c\cdot\nabla\mathfrak{g}+\mu u_c\mathfrak{g}\right)\,dx\to\int_{\mathbb{R}^N}\left(\frac{1}{2m}|\nabla\mathfrak{g}|^2+\mu\mathfrak{g}^2\right)\,dx=\int_{\mathbb{R}^N}\mathfrak{g}^2\log \mathfrak{g}\,dx.
    \end{equation}
    On the other hand, using \eqref{H1normbounded2}, together with the uniform boundedness of $u_c$ in
    $L^{2-\delta}(\mathbb{R}^N)$, $L^2(\mathbb{R}^N)$ and $L^{2+\delta}(\mathbb{R}^N)$, we obtain
    \begin{equation}\label{utog3}
        \limsup_{c\to\infty}\left(\frac{1}{2m}|\nabla u_c|_2^2+\mu|u_c|_2^2\right)\leq\limsup_{c\to\infty}\int_{\mathbb{R}^N}u_c^2\log |u_c|\,dx.
    \end{equation}
    Combining \eqref{utog1}--\eqref{utog3}, we conclude that
    \begin{align*}
        &\limsup_{c\to\infty}\left(\frac{1}{2m}|\nabla(u_c-\mathfrak{g})|_2^2+\mu|u_c-\mathfrak{g}|_2^2\right)\leq\limsup_{c\to\infty}\int_{\mathbb{R}^N}u_c^2\log |u_c|\,dx-\int_{\mathbb{R}^N}\mathfrak{g}^2\log \mathfrak{g}\,dx.
    \end{align*}
     In addition, it holds that
    \begin{equation*}
        \sup_{c\geq c_1}\int_{\mathbb{R}^N}u_c^2|\log |u_c||\,dx<\infty.
    \end{equation*}
    Since Lemma \ref{Sobolevembedding} implies that $u_c\to\mathfrak{g}$ in $L^{2+\delta}(\mathbb{R}^N)$,
     Lemma \ref{logBrezisLieb} applies and gives
    \begin{align*}
        \limsup_{c\to\infty}\left(\frac{1}{2m}|\nabla(u_c-\mathfrak{g})|_2^2+\mu|u_c-\mathfrak{g}|_2^2\right)&
        \leq\limsup_{c\to\infty}\int_{\mathbb{R}^N}|u_c-\mathfrak{g}|^2\log |u_c-\mathfrak{g}|\,dx\\
        &\leq C_\delta\limsup_{c\to\infty}\int_{\mathbb{R}^N}|u_c-\mathfrak{g}|^{2+\delta}\,dx\\
        &=0
    \end{align*}
    Therefore, we conclude that $u_c\to\mathfrak{g}$ in $H^1(\mathbb{R}^N)$ as $c\to\infty$.
    Finally,
    \begin{equation*}
        m_c=\frac{1}{4}|u_c|_2^2\to\frac{1}{4}|\mathfrak{g}|_2^2=m_\infty,
    \end{equation*}
    which also gives the stated formula for $m_\infty$.
    This completes the proof.
\end{proof}

Next, we give the proof of Theorem \ref{result4}.
\begin{proof}
    Suppose by contradiction that there exist a sequence $c_n\to\infty$ and two sequences of nontrivial solutions $u_n,\widetilde{u}_n\in\mathcal{X}$ of \eqref{problem}, such that $E_{c_n}(u_n)<2m_{c_n}$, $E_{c_n}(\widetilde{u}_n)<2m_{c_n}$,
     but $u_n$ and $\widetilde{u}_n$ do not coincide up to translations and a change of sign.
     By Lemmas \ref{constantsign} and \ref{radiallysymmetric}, after translations and changes of sign, we may assume that $u_n$ and $\widetilde{u}_n$ are positive, radial about the origin, and strictly decreasing with respect to the radial variable. In particular, $u_n\not\equiv \widetilde{u}_n$ for every $n$.
     Let $w$ be any nontrivial solution of \eqref{problem} with $E_c(w)<2m_c$. Testing the equation with $w$, we obtain 
     \begin{equation*}
           |w|_2^2=4E_c(w)<8m_c.
     \end{equation*}
     Moreover, Lemma \ref{energybounded} gives a constant $C>0$, independent of $c$, such that
     $|w|_2^2\leq C$.
     The proofs of Lemmas \ref{1/2normbounded}--\ref{weaksolution} use only the equation, this uniform $L^2$-bound, the positivity and radial monotonicity of the solutions, and do not otherwise rely on the ground state property. Their arguments therefore apply to both sequences and yield
     \begin{equation*}
         \sup_{n}\left(\Vert u_n\Vert_{H^1(\mathbb{R}^N)}+\Vert \widetilde{u}_n\Vert_{H^1(\mathbb{R}^N)}\right)<\infty,
     \end{equation*}
     and
     \begin{equation*}
         0<u_n(x),\widetilde{u}_n(x)\leq Ce^{-\beta|x|}\quad\text{for all }x\in\mathbb{R}^N,
     \end{equation*}
     for some constants $C,\beta>0$ independent of $n$.

     We now repeat the arguments used in the proof of Theorem \ref{result2}. After passing to a common subsequence, both sequences converge weakly in $H^1(\mathbb R^N)$ to nontrivial nonnegative solutions of the limiting problem \eqref{limitproblem}. 
     The regularity and positivity arguments in that proof show that both limits are translates of the Gausson $\mathfrak{g}$.
     Since the symmetry centers have been fixed at the origin, both limits must be $\mathfrak{g}$. The $H^1$-convergence argument in the proof of Theorem \ref{result2} then gives
    \begin{equation*}
    u_n\to\mathfrak{g},\quad\widetilde{u}_n\to\mathfrak{g}\quad\text{in }H^1(\mathbb{R}^N).
    \end{equation*}
    After passing to a further common subsequence, we may also assume that
    \begin{equation*}
        u_n(x)\to\mathfrak{g}(x),\quad\widetilde{u}_n(x)\to\mathfrak{g}(x)\text{ a.e. }x\in\mathbb{R}^N.
    \end{equation*}
    We next establish a local estimate that will be used below. Let $K\subset\mathbb{R}^N$ be compact and choose $R_K>0$ such that $K\subset B_{R_K}$. Set $  A_K:=B_{R_K+1}\setminus\overline{B_{R_K}}$.
    The strong $L^2$-convergence gives
    \begin{equation*}
        \int_{A_K}u_n^2\,dx\to \int_{A_K}\mathfrak{g}^2\,dx>0,\quad   \int_{A_K}\widetilde{u}_n^2\,dx\to \int_{A_K}\mathfrak{g}^2\,dx>0.
    \end{equation*}
    Since $u_n$ and $\widetilde{u}_n$ are radially decreasing, for all sufficiently large $n$,
    \begin{equation*}
        u_n(R_K)^2\geq\frac{1}{|A_K|}\int_{A_K}u_n^2\,dx\geq\frac{1}{2|A_K|}\int_{A_K}\mathfrak{g}^2\,dx
    \end{equation*}
    and, similarly,
    \begin{equation*}
        \widetilde{u}_n(R_K)^2\geq\frac{1}{|A_K|}\int_{A_K}\widetilde{u}_n^2\,dx\geq\frac{1}{2|A_K|}\int_{A_K}\mathfrak{g}^2\,dx.
    \end{equation*}
    Radial monotonicity therefore implies that there exists $\alpha_K>0$, independent of $n$, such that
    \begin{equation*}
        u_n(x), \widetilde{u}_n(x)\geq\alpha_K\quad\text{for  }x\in K,
    \end{equation*}
    and all sufficiently large $n$.
    On the other hand, the uniform exponential decay yields a constant $\beta_K>0$, independent of $n$, such that
    \begin{equation*}
         u_n(x), \widetilde{u}_n(x)\leq\beta_K\quad\text{for }x\in K.
    \end{equation*}
    Consequently,
    \begin{equation*}
        0<\alpha_K\leq   u_n(x), \widetilde{u}_n(x)\leq \beta_K\quad\text{for }x\in K,
    \end{equation*}
    and for all sufficiently large $n$.
    
    Define 
    \begin{equation*}
        z_n:=\frac{u_n-\widetilde{u}_n}{\Vert u_n-\widetilde{u}_n\Vert_{H^{1/2}(\mathbb{R}^N)}}.
    \end{equation*}
    Then $z_n$ is radial and
    \begin{equation*}
        \Vert z_n\Vert_{H^{1/2}(\mathbb{R}^N)}=1.
    \end{equation*}
    Therefore, up to a subsequence, there exists $z\in H_r^{1/2}(\mathbb{R}^N)$ such that
    \begin{equation*}
        \begin{cases}
            z_n\rightharpoonup z\quad\text{in }H^{1/2}(\mathbb{R}^N),\\
            z_n\to z\quad\text{in }L^p_{loc}(\mathbb{R}^N),\quad 1\leq p<\frac{2N}{N-1},\\
            z_n\to z\quad\text{a.e. in }\mathbb{R}^N.
        \end{cases}
    \end{equation*}
    Define $a_n(x):=\mathcal{D}_f(u_n(x),\widetilde{u}_n(x))$,
    where $\mathcal{D}_f$ is given in \eqref{Df}.
    Subtracting the equations satisfied by $u_n$ and $\widetilde{u}_n$, and then dividing by
    $\Vert u_n-\widetilde{u}_n\Vert_{H^{1/2}}$, gives
    \begin{equation}\label{zc'sequation}
        \left(P_{c_n}(D)+\mu\right)z_n=a_nz_n.
    \end{equation}
    We claim that, for every compact set $K\subset\mathbb{R}^N$,
    \begin{equation}\label{acconvergence}
        a_n\to f'(\mathfrak{g})=\log\mathfrak{g}+1\quad\text{in }L^2(K).
    \end{equation}
    Indeed, we have proved that, for all sufficiently large $n$,
    \begin{equation}\label{localestimate}
        0<\alpha_K\leq u_n, \widetilde{u}_n\leq\beta_K\quad\text{on }K,\quad u_n, \widetilde{u}_n\to\mathfrak{g}\text{ a.e. on }K.
    \end{equation}
    Therefore, \eqref{acconvergence}
    follows directly from Lemma \ref{log-divided-difference} (ii).
    Let $\varphi\in C_c^\infty(\mathbb{R}^N)$ and choose a compact set $K\subset\mathbb{R}^N$ such that $\text{supp}\varphi\subset K$. Testing \eqref{zc'sequation} with $\varphi$, we obtain
    \begin{equation}\label{zc'sintegralequation}
        \int_{\mathbb{R}^N}z_n\left(P_{c_n}(D)+\mu\right)\varphi\,dx=\int_{\mathbb{R}^N}a_nz_n\varphi\,dx.
    \end{equation}
    For every $\varphi\in C_c^\infty(\mathbb{R}^N)$,
    \begin{equation*}
        P_{c_n}(D)\varphi\to-\frac{1}{2m}\Delta\varphi\quad\text{in }L^2(\mathbb{R}^N).
    \end{equation*}
    Thus,
    \begin{equation*}
    \begin{aligned}
        \int_{\mathbb{R}^N}z_nP_{c_n}(D)\varphi\,dx&=\int_{\mathbb{R}^N}z_n\left(P_{c_n}(D)\varphi+\frac{1}{2m}\Delta\varphi\right)\,dx-\frac{1}{2m}\int_{\mathbb{R}^N}z_n\Delta\varphi\,dx\\
        &\to-\frac{1}{2m}\int_{\mathbb{R}^N}z\Delta\varphi\,dx.
    \end{aligned}
    \end{equation*}
    Besides, by \eqref{acconvergence},
    \begin{equation*}
        \left|\int_{\mathbb{R}^N}(a_n-f'(\mathfrak{g}))z_n\varphi\,dx\right|\leq|\varphi|_{L^\infty(K)}|a_n-f'(\mathfrak{g})|_{L^2(K)}|z_n|_{L^2(K)}\to 0.
    \end{equation*}
    Since $f'(\mathfrak{g})$ is bounded on $K$ and $z_n\to z$ in $L^2(K)$, we also have
    \begin{equation*}
        \int_{\mathbb{R}^N}f'(\mathfrak{g})z_n\varphi\,dx\to\int_{\mathbb{R}^N}f'(\mathfrak{g})z\varphi\,dx.
    \end{equation*}
    Passing to the limit in \eqref{zc'sintegralequation}, we conclude that
    \begin{equation*}
       -\frac{1}{2m}\int_{\mathbb{R}^N}z\Delta\varphi\,dx+\mu\int_{\mathbb{R}^N}z\varphi\,dx=\int_{\mathbb{R}^N}f'(\mathfrak{g})z\varphi\,dx.
    \end{equation*}
    Thus,
    \begin{equation}\label{lz=0}
        \mathcal{L}_\infty z=0
    \end{equation}
    in the distributional sense.
    Since
    \begin{equation*}
        \left(\frac{m}{2}|x|^2-\frac{N}{2}-1\right)z\in L_{loc}^2(\mathbb{R}^N),
    \end{equation*}
    local elliptic regularity implies that $z\in H_{loc}^2(\mathbb{R}^N)$.
    
    We next show that
    \begin{equation}\label{zinH1}
        z\in H^1(\mathbb{R}^N)\quad\text{and}\quad |x|z\in L^2(\mathbb{R}^N).
    \end{equation}
    Let $\eta_R\in C_c^\infty(\mathbb{R}^N)$ satisfy
    \begin{equation*}
        0\leq\eta_R\leq1,\quad\eta_R=1\text{ for }|x|\leq R,\quad\eta_R=0\text{ for }|x|\geq 2R,\quad\text{and }|\nabla \eta_R|\leq\frac{C}{R}.
    \end{equation*}
   Testing \eqref{lz=0} with $\eta_R^2z$, we obtain
   \begin{equation*}
        \frac{1}{2m}\int_{\mathbb{R}^N}\eta_R^2|\nabla z|^2\,dx+\frac{1}{m}\int_{\mathbb{R}^N}\eta_Rz\nabla z\cdot\nabla\eta_R\,dx+\frac{m}{2}\int_{\mathbb{R}^N}|x|^2\eta_R^2z^2\,dx=\left(\frac{N}{2}+1\right)\int_{\mathbb{R}^N}\eta_R^2z^2\,dx.
    \end{equation*}
    By Young's inequality,
    \begin{equation*}
        \frac{1}{m}\left|\int_{\mathbb{R}^N}\eta_Rz\nabla z\cdot\nabla\eta_R\,dx\right|\leq\frac{1}{4m}\int_{\mathbb{R}^N}\eta_R^2|\nabla z|^2\,dx+\frac{1}{m}\int_{\mathbb{R}^N}z^2|\nabla\eta_R|^2\,dx.
    \end{equation*}
    Consequently,
    \begin{equation*}
         \frac{1}{4m}\int_{\mathbb{R}^N}\eta_R^2|\nabla z|^2\,dx+\frac{m}{2}\int_{\mathbb{R}^N}|x|^2\eta_R^2z^2\,dx\leq\left(\frac{N}{2}+1+\frac{C^2}{mR^2}\right)|z|_2^2.
    \end{equation*}
    Letting $R\to\infty$, Fatou's Lemma yields that
    \begin{equation*}
        \frac{1}{4m}\int_{\mathbb{R}^N}|\nabla z|^2\,dx+\frac{m}{2}\int_{\mathbb{R}^N}|x|^2z^2\,dx\leq\left(\frac{N}{2}+1\right)|z|_2^2<\infty.
    \end{equation*}
    which implies \eqref{zinH1}.
   Hence $z\in\Sigma_\infty$. The distributional equation gives $\ell_\infty(z,\varphi)=0$ for every $\varphi\in C_c^\infty(\mathbb{R}^N)$. By density and the continuity of $\ell_\infty$, this equality holds for every $\varphi\in\Sigma_\infty$. Thus $z\in\text{Ker}(\mathcal{L}_\infty)$.
    The nondegeneracy of the Gausson $\mathfrak{g}$ yields
    \begin{equation*}
        \text{Ker}(\mathcal{L}_\infty)=\text{span}\left\{\partial_{x_1}\mathfrak{g},\ldots,\partial_{x_N}\mathfrak{g}\right\}.
    \end{equation*}
    Since $z$ is radial, it holds that $z=0$.
    Therefore, $z_n\to0$ in $L_{loc}^2(\mathbb{R}^N)$.
    Using again the uniform exponential decay of $u_n$ and $\widetilde{u}_n$, we can choose $R>0$, independent of $n$, such that
    \begin{equation*}
        0<u_n(x), \widetilde{u}_n(x)\leq e^{-1}\quad\text{for }|x|\geq R.
    \end{equation*}
   Hence, by Lemma \ref{log-divided-difference} (iii), $a_n(x)\leq0$ for $|x|\geq R$.
    On the other hand, applying the local estimate \eqref{localestimate} and using
    Lemma \ref{log-divided-difference} (i), we find a constant $C_R>0$, independent of $n$, such that $a_n(x)\leq C_R$ for $|x|\leq R$.
    Testing \eqref{zc'sequation} with $z_n$, we obtain
    \begin{equation*}
        \left\langle\left(P_{c_n}(D)+\mu\right)z_n,z_n\right\rangle=\int_{\mathbb{R}^N}a_nz_n^2\,dx\leq C_R\int_{B_R}z_n^2\,dx\to0.
    \end{equation*}
     Nevertheless, by \eqref{symbol-coercive},
    \begin{equation*}
        \left\langle\left(P_{c_n}(D)+\mu\right)z_n,z_n\right\rangle\geq C_{m,\mu}\Vert z_n\Vert_{H^{1/2}(\mathbb{R}^N)}^2=C_{m,\mu}>0,
    \end{equation*}
    which leads to a contradiction. Therefore, there exists $c_*\ge c_1$ such that any two nontrivial solutions with action below $2m_c$ coincide up to translations and a change of sign whenever $c\geq c_*$.

    Finally, since every action ground state has action $m_c<2m_c$, any two positive action ground states coincide up to translations. This proves the last assertion and completes the proof.
\end{proof}

We end this section with the proof of Theorem \ref{result5}.
\begin{proof}
Since equation \eqref{problem} is invariant under translations, we first claim that
\begin{equation*}
       \partial_{x_j}u_c\in\Sigma_c,\quad \ell_c(\partial_{x_j}u_c,\psi)=0\quad\text{for every }\psi\in \Sigma_c,\quad j=1,\ldots,N.
    \end{equation*}
Indeed, it follows from Lemma
\ref{H1normboundedness} that $u_c\in H^1(\mathbb{R}^N)$. Moreover,
$u_c\log u_c\in L^2(\mathbb{R}^N)$ by \eqref{H1normbounded1}. For
$j\in\{1,\ldots,N\}$ and $0<|h|\leq1$, set
\begin{equation*}
    v_h(x):=\frac{u_c(x+he_j)-u_c(x)}{h},\quad
    a_h(x):=\mathcal{D}_f(u_c(x+he_j),u_c(x)).
\end{equation*}
Translation invariance of $P_c(D)$ gives
\begin{equation}\label{translation-difference-equation}
    (P_c(D)+\mu)v_h=a_hv_h.
\end{equation}
For each fixed $h$, one has $v_h\in H^1(\mathbb{R}^N)$ and
\begin{equation*}
    a_hv_h=\frac{f(u_c(\cdot+he_j))-f(u_c)}{h}\in L^2(\mathbb{R}^N).
\end{equation*}
Since $u_c\in H^1(\mathbb{R}^N)$, the family $\{v_h\}_{0<|h|\leq1}$ is
bounded in $L^2(\mathbb{R}^N)$ and
$v_h\to\partial_{x_j}u_c$ in $L^2(\mathbb{R}^N)$ as $h\to0$.

On the other hand, by Lemma \ref{uniformexponentialdecay}, we can choose $R>0$ such that
$u_c(x),u_c(x+he_j)\leq e^{-1}$ whenever $|x|\geq R$ and $|h|\leq1$.
Lemma \ref{log-divided-difference} (iii) therefore gives $a_h\leq0$
on $\mathbb{R}^N\setminus B_R$. Since $u_c$ is positive and continuous,
Lemma \ref{log-divided-difference} (i) also yields
$|a_h|\leq C_R$ on $B_R$, uniformly for $|h|\leq1$.
Taking the $L^2$ inner product of \eqref{translation-difference-equation}
with $v_h$ and using \eqref{symbol-coercive}, we obtain
\begin{equation*}
    C_{m,\mu}\Vert v_h\Vert_{H^{1/2}(\mathbb{R}^N)}^2
    \leq \left\langle(P_c(D)+\mu)v_h,v_h\right\rangle
    =\int_{\mathbb{R}^N}a_hv_h^2\,dx
    \leq C_R|v_h|_{L^2(B_R)}^2\leq C.
\end{equation*}
The same identity, together with the nonnegativity of its left-hand side,
gives
\begin{equation}\label{translation-weight-bound}
    \int_{\mathbb{R}^N\setminus B_R}(-a_h)v_h^2\,dx
    =\int_{B_R}a_hv_h^2\,dx
       -\left\langle(P_c(D)+\mu)v_h,v_h\right\rangle
    \leq C.
\end{equation}
After passing to a subsequence as $h\to0$, we have
$v_h\to\partial_{x_j}u_c$ almost everywhere and
$a_h\to1+\log u_c$ almost everywhere. Fatou's lemma applied to
\eqref{translation-weight-bound} yields
\begin{equation*}
    \int_{\mathbb{R}^N\setminus B_R}
       |1+\log u_c|\,|\partial_{x_j}u_c|^2\,dx<\infty.
\end{equation*}
In addition, the uniform $H^{1/2}$ bound implies
$\partial_{x_j}u_c\in H^{1/2}(\mathbb{R}^N)$. Since $\log u_c$ is bounded
on $B_R$, the preceding estimate proves
$\partial_{x_j}u_c\in\Sigma_c$.

For every $\phi\in C_c^\infty(\mathbb{R}^N)$, the positivity and
continuity of $u_c$ imply that $a_h\to1+\log u_c$ uniformly on
$\text{supp}\phi$. Passing to the limit in the weak formulation of
\eqref{translation-difference-equation}, we obtain
$\ell_c(\partial_{x_j}u_c,\phi)=0$. Finally, a standard cutoff-and-mollification argument yields that
$C_c^\infty(\mathbb{R}^N)$ is dense in $\Sigma_c$. It then follows that
\begin{equation*}
    \ell_c(\partial_{x_j}u_c,\psi)=0
    \quad\text{for every }\psi\in\Sigma_c,
    \quad j=1,\ldots,N.
\end{equation*}
Consequently,
\begin{equation}\label{subset}
    \text{span}\{\partial_{x_1}u_c,\ldots,
    \partial_{x_N}u_c\}\subset\text{Ker}(\mathcal{L}_c).
\end{equation}
    It remains to prove the reverse inclusion.
    Suppose by contradiction that there exist $c_n\to\infty$, positive action ground states $u_n:=u_{c_n}$, and nonzero functions
    \begin{equation}\label{L2orth}
       z_n\in\text{Ker}(\mathcal{L}_{c_n})\cap\text{span}\{\partial_{x_1} u_n,\ldots,\partial_{x_N}u_n\}^\perp.
    \end{equation}
    After normalization, we may assume that
    \begin{equation}\label{zn'snormalization}
        |z_n|_2=1,\quad  \left\langle
       z_n,\partial_{x_j}u_n
        \right\rangle=0,
       \quad j=1,\ldots,N.
    \end{equation}
    Since $z_n\in\text{Ker}(\mathcal{L}_{c_n})$, it satisfies
    \begin{equation}\label{zn'sequation}
       \int_{\mathbb{R}^N}(P_{c_n}(\xi)+\mu)\widehat{z_n}(\xi)\overline{\widehat{\psi}(\xi)}\,d\xi=\int_{\mathbb{R}^N}(1+\log u_n)z_n\psi\,dx\quad\text{for every }\psi\in \Sigma_{c_n}. 
    \end{equation}
    By translating $u_n$ and $z_n$ by the same vector, we may assume that each $u_n$ is radial about the origin and strictly decreasing with respect to $|x|$.
    Theorem \ref{result2} then yields, after passing to a subsequence,
    \begin{equation}\label{unconvegeg}
        u_n\to\mathfrak{g}\quad\text{in }H^1(\mathbb{R}^N).
    \end{equation}
    The uniform exponential decay gives a constant $R>0$, independent of $n$, such that
    \begin{equation}\label{un'scontrol1}
        0<u_n(x)\leq e^{-1}\quad\text{for }|x|\ge R.
    \end{equation}
    Then
     \begin{equation*}
         1+\log u_n(x)\leq0
         \quad\text{for }|x|\geq R.
     \end{equation*}
    Besides, the uniform $L^\infty$ estimate for $u_n$ yields a constant $C_R>0$, independent of $n$, such that
    \begin{equation}\label{un'scontrol2}
        1+\log u_n(x)\leq C_R
      \quad\text{for }|x|\leq R.
    \end{equation}
    Testing \eqref{zn'sequation} against $z_n$ and using \eqref{un'scontrol1}--\eqref{un'scontrol2}, we obtain
    \begin{equation}\label{upperbound}
         \left\langle\left(P_{c_n}(D)+\mu\right)z_n,z_n\right\rangle=\int_{\mathbb{R}^N}(1+\log u_n)z_n^2\,dx\leq C_R\int_{B_R}z_n^2\,dx.
    \end{equation}
    On the other hand, the uniform coercivity estimate \eqref{symbol-coercive} implies
    \begin{equation}\label{lowerbound}
        \left\langle\left(P_{c_n}(D)+\mu\right)z_n,z_n\right\rangle\geq C_{m,\mu}\Vert z_n\Vert_{H^{1/2}(\mathbb{R}^N)}^2.
    \end{equation}
    Consequently, $\{z_n\}$ is bounded in $H^{1/2}(\mathbb{R}^N)$. Thus, after passing to a further subsequence, there exists $z\in H^{1/2}(\mathbb {R}^N)$ such that
    \begin{equation*}
        z_n\rightharpoonup z\quad\text{in }H^{1/2}(\mathbb{R}^N),\quad z_n\to z\quad\text{in }L_{loc}^2(\mathbb{R}^N).
    \end{equation*}
    Furthermore, by \eqref{zn'snormalization}, \eqref{upperbound} and \eqref{lowerbound},
    \begin{equation*}
        C_{m,\mu}\leq C_{m,\mu}\Vert z_n\Vert_{H^{1/2}(\mathbb{R}^N)}^2\leq  \left\langle\left(P_{c_n}(D)+\mu\right)z_n,z_n\right\rangle\leq C_R\int_{B_R}z_n^2\,dx.
    \end{equation*}
    Passing to the limit and using the strong convergence in $L^2(B_R)$, we obtain
    \begin{equation}\label{zneq0}
        \int_{B_R}z^2\,dx\geq\frac{C_{m,\mu}}{C_R}>0.
    \end{equation}
    
    We next identify the equation satisfied by $z$. By the local lower-bound argument used in the proof of Theorem \ref{result4}, the radial monotonicity of $u_n$, together with \eqref{unconvegeg}, implies that $u_n$ is uniformly bounded away from zero on every compact subset of $\mathbb R^N$. Hence,
    \begin{equation*}
        1+\log u_n\to 1+\log \mathfrak{g}\quad\text{in }L_{loc}^2(\mathbb{R}^N).
    \end{equation*}
    For every $\varphi\in C_c^\infty(\mathbb R^N)$, one also has
    \begin{equation*}
         P_{c_n}(D)\varphi\to-\frac{1}{2m}\Delta \varphi\quad\text{in }L^2(\mathbb{R}^N).
     \end{equation*}
     Testing \eqref{zn'sequation} against $\varphi$ and passing to the limit, we find
     \begin{equation*}
         -\frac{1}{2m}\int_{\mathbb{R}^N}z\Delta\varphi\,dx+\mu\int_{\mathbb{R}^N}z\varphi\,dx=\int_{\mathbb{R}^N}(1+\log \mathfrak{g})z\varphi\,dx.
     \end{equation*}
     Thus,
     \begin{equation*}
         \mathcal{L}_\infty z=0
     \end{equation*}
     in the distributional sense. The cutoff argument used in the proof of Theorem \ref{result4} also gives $z\in H^1(\mathbb{R}^N)$ and $|x|z\in L^2(\mathbb{R}^N)$.
     Therefore, $z\in\Sigma_\infty$.
     The same density argument as in the proof of Theorem \ref{result4} then shows that $z\in\text{Ker}(\mathcal{L}_\infty)$.
     On the other hand,
     by the nondegeneracy of the Gausson $\mathfrak{g}$,
     \begin{equation*}
    \text{Ker}(\mathcal{L}_\infty)=\text{span}\{\partial_{x_1}\mathfrak{g},\ldots,\partial_{x_N}\mathfrak{g}\}.
\end{equation*}
    Moreover, since $u_n\to\mathfrak{g}$ strongly in $H^1(\mathbb{R}^N)$,
    \begin{equation}\label{partialconvergence}
        \partial_{x_j}u_n\to\partial_{x_j}\mathfrak{g}\quad\text{in }L^2(\mathbb{R}^N),\quad j=1,\ldots,N.
    \end{equation}
    Combining \eqref{L2orth}, \eqref{partialconvergence} and the weak convergence $z_n\rightharpoonup z$ in $H^{1/2}(\mathbb{R}^N)$, we obtain
    \begin{equation*}
        \left\langle z,\partial_{x_j}\mathfrak{g}\right\rangle=\lim_{n\to\infty}  \left\langle z_n,\partial_{x_j}u_n\right\rangle=0,\quad j=1,\ldots, N.
    \end{equation*}
    As a result,
    \begin{equation*}
        z\perp\text{Ker}(\mathcal{L}_\infty).
    \end{equation*}
    It then follows that $z=0$, which contradicts \eqref{zneq0}. Hence the reverse inclusion in \eqref{subset} holds for every sufficiently large $c$. Therefore,
     \begin{equation*}
        \text{Ker}(\mathcal{L}_c)=\text{span}\{\partial_{x_1}u_c,\ldots,\partial_{x_N}u_c\}.
    \end{equation*}
    This completes the proof.
\end{proof}

\section{Energy ground states}
The results established in the preceding sections have been formulated in terms of action ground states at a fixed frequency. We now turn to the corresponding problem of minimizing the Hamiltonian energy under a prescribed-mass constraint. We first derive the corresponding scaling relations and then use them to transfer the preceding existence, qualitative, convergence, uniqueness, and nondegeneracy results to the prescribed-mass setting. 
 
 For an arbitrary frequency $\lambda\in\mathbb{R}$, define
\begin{equation*}
    \mathcal{S}_{c,\lambda}(u):=\mathcal{H}_c(u)+\frac{\lambda}{2}\mathcal{M}(u)
\end{equation*}
and the associated Nehari functional
\begin{equation*}
    \mathcal{K}_{c,\lambda}(u):=\int_{\mathbb{R}^N}\left(P_c(\xi)+\lambda\right)|\widehat{u}(\xi)|^2\,d\xi-\int_{\mathbb{R}^N}u^2\log|u|\,dx.
\end{equation*}
Let
\begin{equation*}
    \mathcal{N}_{c,\lambda}
:=
\left\{
u\in\mathcal{X}\setminus\{0\}:
\mathcal{K}_{c,\lambda}(u)=0
\right\}
\end{equation*}
and set
\begin{equation*}
    m_c(\lambda):=\inf_{u\in\mathcal{N}_{c,\lambda}}\mathcal{S}_{c,\lambda}(u).
\end{equation*}
To connect the frequency-dependent formulation with the action ground states studied above, we first identify its level at the reference frequency $\lambda=\mu$ with the previously established level $m_c$.
For every $u\in\mathcal{N}_{c,\mu}$, its extension $U$ belongs to $\mathcal{N}_c$ and satisfies $I_c(U)=\mathcal{S}_{c,\mu}(u)$. Hence $m_c\leq m_c(\mu)$. Conversely, let $U\in\mathcal{N}_c$ and write $u=U(\cdot,0)$. Lemma \ref{norms} yields
\begin{equation*}
    \mathcal{K}_{c,\mu}(u)\le J_c(U)=0.
\end{equation*}
Thus, with
\begin{equation*}
    \log t=\frac{\mathcal{K}_{c,\mu}(u)}{|u|_2^2}\leq0,
\end{equation*}
we have $tu\in\mathcal{N}_{c,\mu}$ and
\begin{equation*}
    m_c(\mu)
\leq\mathcal{S}_{c,\mu}(tu)
=\frac{t^2}{4}|u|_2^2
\leq\frac{1}{4}|u|_2^2
=I_c(U).
\end{equation*}
Taking the infimum over $U\in\mathcal{N}_c$ gives $m_c(\mu)=m_c$.

\begin{lemma}\label{groundstate-Neharilevel}
  For each $\lambda\in\mathbb{R}$,  every minimizer of $\mathcal{S}_{c,\lambda}$ on $\mathcal{N}_{c,\lambda}$ is a nontrivial weak solution of 
  \begin{equation*}
      (P_c(D)+\lambda)u=u\log|u|.
  \end{equation*}
\end{lemma}
\begin{proof}
    Let $u\in\mathcal{N}_{c,\lambda}$ satisfy $\mathcal{S}_{c,\lambda}(u)=m_c(\lambda)$, and fix $\varphi\in C_c^\infty(\mathbb{R}^N)$. For $s$ sufficiently close to zero, set
    \begin{equation*}
        w_s:=u+s\varphi,\quad M(s):=\int_{\mathbb{R}^N}w_s^2\,dx,\quad K(s):=\mathcal{K}_{c,\lambda}(w_s).
    \end{equation*}
    Then $w_s\in\mathcal{X}\setminus\{0\}$. Moreover, the functions $M$ and $K$ are continuously differentiable near $s=0$. Indeed, the function
    \begin{equation*}
        h(t):=
      \begin{cases}
      t^2\log|t|,&t\neq0,\\
      0,&t=0
      \end{cases}
    \end{equation*}
    belongs to $C^1(\mathbb{R})$, with 
    \begin{equation*}
        h'(t)=2t\log |t|+t, t\neq0,\quad h'(0)=0.
    \end{equation*}
    For any sufficiently small $\epsilon>0$,
    \begin{equation*}
        |h'(t)|\leq c_\epsilon(1+|t|^{1+\epsilon}).
    \end{equation*}
    On the compact support of $\varphi$, this estimate and the Sobolev embedding justify differentiation under the integral sign. Consequently,
    \begin{equation*}
        M'(0)=2\int_{\mathbb{R}^N}u\varphi\,dx,\quad K'(0)=2\int_{\mathbb{R}^N}(P_c(\xi)+\lambda)\widehat{u}(\xi)\overline{\widehat{\varphi}(\xi)}\,d\xi-\int_{\mathbb{R}^N}(2u\log|u|+u)\varphi\,dx.
    \end{equation*}
    Define
    \begin{equation*}
        \log t(s):=\frac{K(s)}{M(s)}.
    \end{equation*}
    By the scaling identity
    \begin{equation*}
        \mathcal{K}_{c,\lambda}(tv)
      =t^2\left(
       \mathcal{K}_{c,\lambda}(v)-\log t\int_{\mathbb R^N}v^2\,dx\right),
    \end{equation*}
    we have $t(s)w_s\in\mathcal{N}_{c,\lambda}$. Moreover, $K(0)=0$, so $t(0)=1$. The minimality of $u$ therefore yields that the function
    \begin{equation*}
        F(s):=\mathcal{S}_{c,\lambda}(t(s)w_s)=\frac{1}{4}t(s)^2M(s)
    \end{equation*}
    has a local minimum at $s=0$. Therefore,
    \begin{equation*}
        0=F'(0)=\frac{1}{2}K'(0)+\frac{1}{4}M'(0)=\int_{\mathbb{R}^N}(P_c(\xi)+\lambda)\widehat{u}(\xi)\overline{\widehat{\varphi}(\xi)}\,d\xi-\int_{\mathbb{R}^N}u\varphi\log|u|\,dx.
    \end{equation*}
    Since $\varphi\in C_c^\infty(\mathbb{R}^N)$ was arbitrary, $u$ is a weak solution of the above equation. Its nontriviality follows from $u\in\mathcal{N}_{c,\lambda}$.
\end{proof}
Let
\begin{equation*}
    \mathcal{G}^a_{c,\lambda}
:=
\left\{
u\in\mathcal{N}_{c,\lambda}:
\mathcal{S}_{c,\lambda}(u)=m_c(\lambda)
\right\}.
\end{equation*}

The preceding lemma and the frequency-scaling argument below show that this is precisely the set of action ground states at frequency $\lambda$.

\begin{lemma}\label{frequencyscaling}
   For every $\lambda,\nu\in\mathbb R$,
    \begin{equation*}
        \mathcal{N}_{c,\lambda}
          =
    e^{\lambda-\nu}\mathcal{N}_{c,\nu}\quad\text{and}\quad    m_c(\lambda)
      =
      e^{2(\lambda-\nu)}m_c(\nu).
    \end{equation*}
    Moreover,
   \begin{equation*}
       \mathcal{G}^{a}_{c,\lambda}
       =
       e^{\lambda-\nu}
       \mathcal{G}^{a}_{c,\nu}.
   \end{equation*}
\end{lemma}
\begin{proof}
    Set $a:=e^{\lambda-\nu}$.
    Since $\log|au|=\log a+\log|u|$,
    a direct calculation gives $\mathcal{K}_{c,\lambda}(au)=a^2  \mathcal{K}_{c,\nu}(u)$.
    Hence,
    \begin{equation*}
        u\in\mathcal{N}_{c,\nu}
     \iff
     au\in\mathcal{N}_{c,\lambda},
    \end{equation*}
    and therefore $\mathcal{N}_{c,\lambda}
          =
    a\mathcal{N}_{c,\nu}$.
    For every $u\in\mathcal{N}_{c,\nu}$, the Nehari identity gives
    \begin{equation*}
        \mathcal{S}_{c,\nu}(u)=\frac{1}{4}\mathcal{M}(u).
    \end{equation*}
    Likewise, $au\in\mathcal{N}_{c,\lambda}$, and therefore
    \begin{equation*}
        \mathcal{S}_{c,\lambda}(au)=\frac{1}{4}\mathcal{M}(au)=a^2 \mathcal{S}_{c,\nu}(u).
    \end{equation*}
    Taking infima yields
    \begin{equation*}
        m_c(\lambda)=a^2m_c(\nu)=e^{2(\lambda-\nu)}m_c(\nu).
    \end{equation*}
    Moreover, applying the same argument to minimizers proves $ \mathcal{G}^{a}_{c,\lambda}
       =
       e^{\lambda-\nu}
       \mathcal{G}^{a}_{c,\nu}$.
\end{proof}
The same scaling maps nontrivial weak solutions at frequency $\nu$ bijectively onto those at frequency $\lambda$, and multiplies their action levels by $e^{2(\lambda-\nu)}$. Together with the existence and ground-state characterization at frequency $\mu$, this proves that $\mathcal{G}^a_{c,\lambda}$ is nonempty and consists precisely of all action ground states at frequency $\lambda$.

\begin{proposition}\label{Exactaction-energy equivalence}
    Let $\lambda\in\mathbb{R}$ and set
   \begin{equation*}
       \rho_\lambda:=4m_c(\lambda).
   \end{equation*}
   Then
   \begin{equation*}
       \mathcal G^{e}_{c,\rho_\lambda}
       =
      \mathcal G^{a}_{c,\lambda}.
   \end{equation*}
\end{proposition}
\begin{proof}
    Let $u\in \mathcal G^{a}_{c,\lambda}$. The Nehari identity gives 
    \begin{equation*}
        \mathcal{M}(u)=4\mathcal{S}_{c,\lambda}(u)=4m_c(\lambda)=\rho_\lambda.
    \end{equation*}
    If some $v\in\mathcal{X}$ with $\mathcal{M}(v)=\rho_\lambda$ satisfied $\mathcal{H}_c(v)<\mathcal{H}_c(u)$, then
    \begin{equation*}
        \mathcal{S}_{c,\lambda}(v)< \mathcal{S}_{c,\lambda}(u)=\frac{1}{4}\rho_\lambda.
    \end{equation*}
    Since
    \begin{equation*}
         \mathcal{S}_{c,\lambda}(w)=\frac{1}{2}\mathcal{K}_{c,\lambda}(w)+\frac{1}{4}\mathcal{M}(w)
    \end{equation*}
    for every $w\in\mathcal{X}$, this implies $ \mathcal{K}_{c,\lambda}(v)<0$.
    The unique $t>0$ satisfying $tv\in\mathcal{N}_{c,\lambda}$ is determined by
    \begin{equation*}
        \log t=\frac{\mathcal{K}_{c,\lambda}(v)}{\mathcal{M}(v)}.
    \end{equation*}
    Hence $t<1$. Moreover,
    \begin{equation*}
        m_c(\lambda)\leq \mathcal{S}_{c,\lambda}(tv)=\frac{t^2}{4}\rho_\lambda<\frac{1}{4}\rho_\lambda=m_c(\lambda),
    \end{equation*}
    a contradiction. Thus $u$ is an energy ground state of mass $\rho_\lambda$.

    Conversely, let $v\in\mathcal{G}^e_{c,\rho_\lambda}$ and choose $u\in\mathcal{G}^a_{c,\lambda}$. The first part shows that $u$ is an energy minimizer on the same mass sphere, so $\mathcal{H}_c(v)=\mathcal{H}_c(u)$. 
    Therefore,
    \begin{equation*}
        \mathcal{S}_{c,\lambda}(v)
         =
        \mathcal{S}_{c,\lambda}(u)
         =
        m_c(\lambda)=\frac{1}{4}\mathcal{M}(v),
    \end{equation*}
    which implies $ \mathcal{K}_{c,\lambda}(v)=0$.
    Hence $v\in\mathcal{N}_{c,\lambda}$ and attains $m_c(\lambda)$, proving $v\in\mathcal{G}_{c,\lambda}^a$.
\end{proof}

We are now in a position to prove Theorems \ref{result6}--\ref{result8}.
\begin{proof}
    Fix $\rho>0$, and define
    \begin{equation*}
        \alpha_{c,\rho}:=\left(\frac{\rho}{4m_c}\right)^\frac{1}{2},\quad \lambda_{c,\rho}:=\mu+\log \alpha_{c,\rho}.
    \end{equation*}
    By Lemma \ref{frequencyscaling},
    \begin{equation*}
        m_c(\lambda_{c,\rho})=\alpha_{c,\rho}^2m_c=\frac{\rho}{4},
    \end{equation*}
    and hence $\rho=4m_c(\lambda_{c,\rho})$.
    Proposition \ref{Exactaction-energy equivalence}, followed once more by Lemma \ref{frequencyscaling}, therefore gives
    \begin{equation*}
       \mathcal G^{e}_{c,\rho}
       =
      \mathcal G^{a}_{c,\lambda_{c,\rho}}=\alpha_{c,\rho} \mathcal G^{a}_{c,\mu}.
   \end{equation*}
   This proves the existence of energy ground states and immediately transfers the sign, symmetry, radial monotonicity, regularity and decay properties from Theorems \ref{result1} and \ref{result3}.

   We next establish the nonrelativistic limit. Let $u_{c,\rho}\in\mathcal{G}^e_{c,\rho}$ be any positive energy ground state. By the preceding correspondence,
   \begin{equation*}
       u_c:=\alpha_{c,\rho}^{-1}u_{c,\rho}\in\mathcal{G}_{c,\mu}^a.
   \end{equation*}
   After suitable translations, we may assume that $u_c$ is centered at the origin.
   By Theorem \ref{result2},
   \begin{equation*}
       u_c\to\mathfrak{g}_\mu\quad\text{in }H^1(\mathbb{R}^N),\quad \mathfrak{g}_\mu(x):=e^{\mu+\frac{N}{2}}e^{-\frac{m}{2}|x|^2}.
   \end{equation*}
   Since every action ground state satisfies
   $4m_c=\mathcal{M}(u_c)$,
   we have
   \begin{equation*}
       4m_c\to \mathcal{M}(\mathfrak{g}_\mu)=e^{2\mu+N}\left(\frac{\pi}{m}\right)^\frac{N}{2}.
   \end{equation*}
   Consequently,
   \begin{equation*}
       \alpha_{c,\rho}\to\frac{\rho^\frac{1}{2}}{|\mathfrak{g}_\mu|_2}.
   \end{equation*}
   It then follows that
   \begin{equation*}
       \alpha_{c,\rho}u_c\to\frac{\rho^\frac{1}{2}}{|\mathfrak{g}_\mu|_2}\mathfrak{g}_\mu=:\mathfrak{g}_\rho\quad\text{in }H^1(\mathbb{R}^N).
   \end{equation*}
   A direct computation gives
   \begin{equation*}
       \mathfrak{g}_\rho(x)=\rho^\frac{1}{2}\left(\frac{m}{\pi}\right)^\frac{N}{4}e^{-\frac{m}{2}|x|^2},\quad  |\mathfrak g_\rho|_2^2=\rho,
   \end{equation*}
   so $\mathfrak g_\rho$ is precisely the Gausson with the prescribed mass $\rho$.
   Moreover,
   \begin{equation*}
       \lambda_{c,\rho}=\mu+\frac{1}{2}\log\left(\frac{\rho}{4m_c}\right)\to\lambda_{\infty,\rho}:=\frac{1}{2}\log\rho+\frac{N}{4}\log\left(\frac{m}{\pi}\right)-\frac{N}{2}.
   \end{equation*}
   Since $u_{c,\rho}\in\mathcal N_{c,\lambda_{c,\rho}}$, the Nehari identity gives
    \begin{equation*}
        \mathcal{S}_{c,\lambda_{c,\rho}}(u_{c,\rho})
        =
        \frac{1}{4}\mathcal{M}(u_{c,\rho})
        =
        \frac{\rho}{4}.
    \end{equation*}
    On the other hand,
    \begin{equation*}
        \mathcal{S}_{c,\lambda_{c,\rho}}(u_{c,\rho})
        =
       \mathcal{H}_c(u_{c,\rho})
        +
        \frac{\lambda_{c,\rho}}{2}\rho.
    \end{equation*}
    Therefore, the energy ground-state level satisfies
    \begin{equation*}
        e_c(\rho)
        =
       \frac{\rho}{4}
        -\frac{\rho}{2}\lambda_{c,\rho}=
       \frac{\rho}{4}
       -\frac{\mu\rho}{2}
       -\frac{\rho}{4}
        \log\left(\frac{\rho}{4m_c}\right).
    \end{equation*}
    Furthermore,
    \begin{equation*}
        e_c(\rho)
      \to
      e_\infty(\rho)= \frac{\rho}{4}
      -
       \frac{\rho}{2}
      \lambda_{\infty,\rho}= \frac{\rho}{4}
      \left(N+1-\log\rho-\frac{N}{2}\log\left(\frac{m}{\pi}\right)\right).
    \end{equation*}
   For sufficiently large $c$, Theorem \ref{result4} gives uniqueness of $\mathcal{G}^a_{c,\mu}$ up to translations and a change of sign. The correspondence therefore gives the same uniqueness result for $\mathcal{G}^e_{c,\rho}$.

   It remains to prove the nondegeneracy. Set
   $u_{c,\rho}:=\alpha_{c,\rho}u_c$.
   Since
   \begin{equation*}
       \lambda_{c,\rho}
       =
      \mu+\log\alpha_{c,\rho},\quad  \log u_{c,\rho}
     =
     \log\alpha_{c,\rho}+\log u_c,
   \end{equation*}
   we find that $\Sigma_{c,\rho}=\Sigma_c.$
   Indeed, since $\alpha_{c,\rho}>0$, for every $\varphi\in H^{1/2}(\mathbb{R}^N)$,
   \begin{equation*}
       |\log u_{c,\rho}|\varphi^2\in L^1(\mathbb{R}^N)\iff |\log u_c|\varphi^2\in L^1(\mathbb{R}^N).
   \end{equation*}
   Moreover, for every $\varphi,\psi\in \Sigma_{c,\rho}$,
  \begin{equation}\label{Lcrho=Lc}
      \begin{aligned}
           \ell_{c,\rho}(\varphi,\psi)&=\int_{\mathbb{R}^N}P_c(\xi)\widehat{\varphi}(\xi)\overline{\widehat{\psi}(\xi)}\,d\xi+\int_{\mathbb{R}^N}(\lambda_{c,\rho}-1-\log u_{c,\rho})\varphi\psi\,dx\\
           &=\int_{\mathbb{R}^N}P_c(\xi)\widehat{\varphi}(\xi)\overline{\widehat{\psi}(\xi)}\,d\xi+\int_{\mathbb{R}^N}(\mu-1-\log u_c)\varphi\psi\,dx\\
           &=\ell_c(\varphi,\psi).
      \end{aligned}
  \end{equation}
   We first prove
   \begin{equation*}
       \text{Ker}(\mathcal{L}_{c,\rho})\subset\text{span}\left\{\partial_{x_1}u_{c,\rho},\ldots,\partial_{x_N}u_{c,\rho}\right\}.
   \end{equation*}
   Let $\varphi\in\text{Ker}(\mathcal{L}_{c,\rho})$. Then $\varphi\in  T_{u_{c,\rho}}S_\rho$ and $\ell_{c,\rho}(\varphi,\psi)=0$ for every $\psi\in T_{u_{c,\rho}}S_\rho$.
   Since $u_{c,\rho}$ satisfies
   \begin{equation*}
       (P_c(D)+\lambda_{c,\rho})u_{c,\rho}=u_{c,\rho}\log u_{c,\rho},
   \end{equation*}
   we have, for every $\varphi\in\Sigma_{c,\rho}$,
   \begin{equation*}
   \begin{aligned}
       \ell_{c,\rho}(\varphi,u_{c,\rho})&=\int_{\mathbb{R}^N}P_c(\xi)\widehat{\varphi}(\xi)\overline{\widehat{u_{c,\rho}}(\xi)}\,d\xi+\int_{\mathbb{R}^N}(\lambda_{c,\rho}-1-\log u_{c,\rho})\varphi u_{c,\rho}\,dx\\
       &=-\int_{\mathbb{R}^N}u_{c,\rho}\varphi\,dx.
   \end{aligned}
   \end{equation*}
   Since $\varphi\in T_{u_{c,\rho}}S_\rho$, it follows that $\ell_{c,\rho}(\varphi,u_{c,\rho})=0$.
   Now let $\psi\in\Sigma_{c,\rho}$ be arbitrary. Set
   \begin{equation*}
       \psi_0:=\psi-\frac{\langle\psi,u_{c,\rho}\rangle}{\rho}u_{c,\rho}.
   \end{equation*}
   Since $|u_{c,\rho}|_2^2=\rho$, we have $ \langle\psi_0, u_{c,\rho}\rangle=0$,
   and hence $\psi_0\in T_{u_{c,\rho}}S_\rho$. Moreover,
   \begin{equation*}
       \ell_c(\varphi,\psi)=\ell_{c,\rho}(\varphi,\psi)=\ell_{c,\rho}(\varphi,\psi_0)+\frac{\langle\psi,u_{c,\rho}\rangle}{\rho}\ell_{c,\rho}(\varphi,u_{c,\rho})=0.
   \end{equation*}
   Since $\psi\in\Sigma_c$ was arbitrary, we conclude that $\varphi\in\text{Ker}(\mathcal{L}_c)$. Theorem \ref{result5} then yields
   \begin{equation*}
       \varphi\in\text{span}\left\{\partial_{x_1}u_c,\ldots,\partial_{x_N}u_c\right\}.
   \end{equation*}
   Because $u_{c,\rho}=\alpha_{c,\rho}u_c$, we have $\partial_{x_j}u_{c,\rho}=\alpha _{c,\rho}\partial_{x_j}u_c, j=1,\ldots, N$. Thus,
   \begin{equation*}
       \varphi\in\text{span}\left\{\partial_{x_1}u_{c,\rho},\ldots,\partial_{x_N}u_{c,\rho}\right\}.
   \end{equation*}
   We now prove the reverse inclusion. Theorem \ref{result5} gives $\partial_{x_j}u_c\in\text{Ker}(\mathcal{L}_c), j=1,\ldots, N$. Hence,
   \begin{equation*}
       \ell_c(\partial_{x_j}u_c,\psi)=0\quad\text{for every }\psi\in\Sigma_c.
   \end{equation*}
   Using \eqref{Lcrho=Lc} and $\partial_{x_j}u_{c,\rho}=\alpha_{c,\rho}\partial_{x_j}u_c$, we obtain
   \begin{equation}\label{partialucrho}
       \ell_{c,\rho}(\partial_{x_j}u_{c,\rho},\psi)=0\quad\text{for every }\psi\in\Sigma_{c,\rho}.
   \end{equation}
   In particular, \eqref{partialucrho} holds for every $\psi\in T_{u_{c,\rho}}S_\rho$.
   Furthermore, using the decay of $u_{c,\rho}$, we have
   \begin{equation*}
       \int_{\mathbb{R}^N}u_{c,\rho}\partial_{x_j}u_{c,\rho}\,dx=\frac{1}{2}\int_{\mathbb{R}^N}\partial_{x_j}(u_{c,\rho}^2)\,dx=0,
   \end{equation*}
   which implies that $\partial_{x_j}u_{c,\rho}\in T_{u_{c,\rho}}S_\rho$, and consequently,
   \begin{equation*}
       \partial_{x_j}u_{c,\rho}\in\text{Ker}(\mathcal{L}_{c,\rho}),\quad j=1,\ldots, N.
   \end{equation*}
   As a result,
   \begin{equation*}
       \text{Ker}(\mathcal{L}_{c,\rho})=\text{span}\left\{\partial_{x_1}u_{c,\rho},\ldots,\partial_{x_N}u_{c,\rho}\right\}.
   \end{equation*}
This proves the nondegeneracy assertion and completes the proof.
\end{proof}

\noindent{\bf Data availability.}	The manuscript has no associated data.

\noindent{\bf Conflict of interest.}	The authors declare that they have no conflict of interest.

	\bibliographystyle{abbrv}
	\bibliography{reference}
\end{document}